\documentclass[12pt,centertags,oneside]{amsart}
\usepackage{amsmath,amstext,amsthm,amscd,typearea}
\usepackage{amssymb}
\usepackage{fancyhdr}
\usepackage[mathscr]{eucal}
\usepackage{dsfont}
\usepackage{mathrsfs}
\usepackage{mathrsfs}
\usepackage{concmath}
\usepackage{charter}

\usepackage[a4paper,width=16.2cm,top=3cm,bottom=3cm]{geometry}

\numberwithin{equation}{section}

\makeatletter
      \def\@setcopyright{}
      \def\serieslogo@{}
      \makeatother

\newcommand{\Complex}{\mathbb C}
\newcommand{\Real}{\mathbb R}
\newcommand{\N}{\mathbb N}
\newcommand{\ddbar}{\overline\partial}
\newcommand{\pr}{\partial}
\newcommand{\ol}{\overline}
\newcommand{\Td}{\widetilde}
\newcommand{\norm}[1]{\left\Vert#1\right\Vert}
\newcommand{\abs}[1]{\left\vert#1\right\vert}
\newcommand{\set}[1]{\left\{#1\right\}}
\newcommand{\To}{\rightarrow}
\DeclareMathOperator{\Ker}{Ker}
\newcommand{\pit}{\mathit\Pi}
\newcommand{\cali}[1]{\mathscr{#1}}
\newcommand{\cH}{\cali{H}}
\newcommand{\mU}{\mathcal{U}}
\newcommand{\mV}{\mathcal{V}}

\theoremstyle{plain}
\newtheorem{thm}{Theorem}[section]
\newtheorem{cor}[thm]{Corollary}
\newtheorem{lem}[thm]{Lemma}
\newtheorem{prop}[thm]{Proposition}

\theoremstyle{definition}
\newtheorem{defn}[thm]{Definition}

\theoremstyle{remark}
\newtheorem{rem}[thm]{Remark}
\newtheorem{ex}[thm]{Example}
\numberwithin{equation}{section}

\begin{document}
\title[]{Szeg\"{o} kernel asymptotics and Morse inequalities on CR manifolds}
\author[]{Chin-Yu Hsiao}
\address{Department of Mathematics, Chalmers University of Technology\\ \& University of G\"{o}teborg, SE-412 96 G\"{o}teborg, Sweden}
\curraddr{Universit{\"a}t zu K{\"o}ln,  Mathematisches Institut,
    Weyertal 86-90,   50931 K{\"o}ln, Germany}
\thanks{First-named author was supported by a postdoctoral grant of the Swedish research council and partially by the DFG funded project MA 2469/2-1}
\email{chsiao@math.uni-koeln.de}
\author[]{George Marinescu}
\address{Universit{\"a}t zu K{\"o}ln,  Mathematisches Institut,
    Weyertal 86-90,   50931 K{\"o}ln, Germany\\
    \& Institute of Mathematics `Simion Stoilow', Romanian Academy,
Bucharest, Romania}
\thanks{Second-named author was partially supported by the SFB/TR 12}
\email{gmarines@math.uni-koeln.de}
\begin{abstract}
Let $X$ be an abstract compact orientable CR manifold of dimension $2n-1$, $n\geqslant2$, and let
$L^k$ be the $k$-th tensor power of a CR complex line bundle $L$ over $X$. We assume that condition $Y(q)$ holds at each point of $X$.
In this paper we obtain a scaling upper-bound for the Szeg\"o kernel on $(0, q)$-forms with values in $L^k$, for large $k$.
After integration, this gives
weak Morse inequalities, analogues of the holomorphic Morse inequalities of Demailly.
By a refined spectral analysis we obtain also strong Morse inequalities.
We apply the strong Morse inequalities to the embedding of some convex-concave manifolds.
\end{abstract}
\maketitle
\tableofcontents

\section{Introduction and statement of the main results} \label{morse-sec:cr}

The purpose of this paper is to establish analogues of the holomorphic
Morse inequalities of Demailly for CR manifolds.
Demailly  \cite{De:85} proved remarkable asymptotic Morse inequalities
for the ${\overline\partial}$ complex constructed over the line bundle
$L^k$ as $k\to\infty$, where $L$ is a holomorphic
hermitian line bundle. Shortly after, Bismut \cite{B87c} gave a heat
equation proof of Demailly's inequalities which involves probability
theory. Later Demailly \cite{De:91} and Bouche \cite{Bou:96} replaced
the probability technique by a classical heat kernel argument. The book \cite{MM07} introduced an argument
based on the asymptotic of the heat kernel
of the Kodaira Laplacian by using rescaling of the coordinates and functional analytic
techniques inspired by Bismut-Lebeau \cite[\S 11]{BL91} (see also Bismut-Vasserot \cite{BVa89}). A different approach was introduced by Berndtsson \cite{Bern:02} and developed by Berman
\cite{Be04,Be05}; they work with the Bergman kernel and use the mean value estimate for eigensections
of the Kodaira Laplacian.
The idea of all these proofs is localization of the analytic objects (eigenfunctions, kernels) and scaling techniques.
See also Fu-Jacobowitz \cite{FJ07} for related results on domains of finite type.

Inspired by Bismut's paper, Getzler~\cite{Ge89} gave an expression involving local data for the large $k$ limit  of the trace of heat kernel of the $\overline\partial_b$-Laplacian on $L^k$, where $L$ is a CR line bundle
over a CR strongly pseudoconvex manifold.
But Getzler didn't infer Morse inequalities for the $\overline\partial_b$-complex from these asymptotics.

In this paper we introduce a method that
produces Morse inequalities with computable bounds for the growth of the
$\overline\partial_b$ coholmology and also allows more general CR manifolds to be considered.
Our approach is related to the techniques of Berman~\cite{Be04} and Shaw-Wang~\cite{SW05}.

In a project developed jointly with R. Ponge \cite{P05}, we use the heat kernel asymptotics and Heisenberg calculus to prove holomorphic Morse inequalities for a line bundle endowed with the CR Chern connection.
This method predicts similar results and applications as of the present paper.

For a complex manifold with boundary, the
$\overline\partial_b$-cohomology of the boundary is linked to the
$\overline\partial$-cohomology of the interior, cf. Kohn-Rossi
\cite{KR65}, Andreotti-Hill \cite{AH72a,AH72b}.
Stephen S.T. Yau \cite{YauSST81} exhibited the relation between the
$\overline\partial_b$-cohomology of the boundary of a strictly
pseudoconvex Stein analytic space with isolated singularities and
invariants of the singular points.
Holomorphic Morse inequalites for manifolds with boundary were obtained
by Berman \cite{Be05} and in \cite{Ma04,MarTC} (cf.\,also \cite[Ch.\,3]{MM07}). The bounds in the Morse
inequalities appearing in this paper are similar to the boundary terms
in Berman's result \cite{Be05}. For the relation between the boundary
and interior cohomology of high tensor powers $L^k$ see also \cite{Ma96}.

On the other hand, the study of the $\overline\partial_b$-complex on an abstract CR
manifold has important consequences for the embedability and deformation
of the CR-structure, see the embedding theorem of Boutet de Monvel
\cite{BdM1:74b} for strictly pseudoconvex CR manifolds and the paper of
Epstein-Henkin \cite{EH00}.

In this paper we will study the large $k$ behavior of the Szeg\"{o}
kernel function $\pit^{(q)}_k(x)$, which is the restriction to the
diagonal of the integral kernel of the projection $\pit^{(q)}_k$ on the
harmonic $(0,q)$-forms with values in $L^k$. The Szeg\"{o} kernel for
functions on a strictly pseudoconvex CR manifold was studied by
Boutet de Monvel \cite{BdM1:74a} and Boutet de Monvel-Sj\"ostrand
\cite{BouSj76} and has important applications in complex analysis and
geometry.


\subsection{Terminology and Notations}
Let $(X,T^{1,0}X)$ be a CR manifold of dimension $2n-1$, $n\geqslant2$, i.e.\ $T^{1,0}X$ is a subbundle of rank $n-1$ of the complexified tangent bundle $\Complex TX$ satisfying $T^{1,0}X\cap \ol{T^{1,0}X}=\{0\}$ and the integrability condition (see e.g.\  \cite[Def.\,7.1.1]{CS01}).
We shall always assume that $X$ is compact, connected and orientable.

Fix a smooth Hermitian metric $\langle\,\cdot\,,\cdot\,\rangle$ on $\Complex TX$ so that $T^{1,0}X$
is orthogonal to $T^{0,1}X:=\ol{T^{1,0}X}$ and $\langle u,v\rangle$ is real if $u$, $v$ are real tangent vectors.
Then there is a real non-vanishing vector field $T$ on $X$ which is pointwise orthogonal to
$T^{1, 0}X\oplus T^{0, 1}X$.

Denote by $T^{*1,0}X$ and $T^{*0,1}X$ the dual bundles of $T^{1,0}X$ and $T^{0,1}X$, respectively.
They can be identified with subbundles of the complexified cotangent bundle $\Complex T^*X$.
Define the vector bundle of $(0, q)$ forms by $\Lambda^{0, q}T^*X:=\Lambda^{q}T^{*0,1}X$.
The Hermitian metric $\langle\,\cdot\,,\cdot\,\rangle$ on $\Complex TX$ induces,
by duality, a Hermitian metric on $\Complex T^*X$ and also on the bundle of $(0, q)$ forms $\Lambda^{0, q}T^*X$. We shall also denote all these induced metrics by $\langle\,\cdot\,,\cdot\,\rangle$.

Let $D\subset X$ be an open set. Let $\Omega^{0,q}(D)$ denote the space of smooth sections of $\Lambda^{0,q}T^*X$ over $D$. Similarly, if $E$ is a vector bundle over $D$, then we let $\Omega^{0,q}(D, E)$
denote the space of smooth sections of $\Lambda^{0,q}T^*X\otimes E$ over $D$. Let $\Omega^{0,q}_c(D, E)$ be the subspace of
$\Omega^{0,q}(D, E)$ whose elements have compact support in $D$.

If $w\in T^{*0,1}_zX$, let
$(w\wedge)^*: \Lambda^{0,q+1}T^*_zX\To \Lambda^{0,q}T^*_zX,\ q\geqslant0$,
be the adjoint of the left exterior multiplication
$w\wedge: \Lambda^{0,q}T^*_zX\To \Lambda^{0,q+1}T^*_zX$, $u\mapsto w\wedge u$\,:
\begin{equation} \label{s1-e1}
\langle w\wedge u, v\rangle=\langle u,(w\wedge)^*v\rangle\,,
\end{equation}
for all $u\in\Lambda^{0,q}T^*_zX$, $v\in\Lambda^{0,q+1}T^*_zX$.
Notice that $(w\wedge)^*$ depends $\mathbb{C}$-anti-linearly on $w$.

In the sequel we will denote by $\langle\,\cdot\,,\cdot\,\rangle$ both scalar products as well as the duality bracket between vector fields and forms.

Locally we can choose an orthonormal frame $\omega_1,\ldots,\omega_{n-1}$
of the bundle $T^{*1,0}X$. Then $\ol\omega_1,\ldots,\ol\omega_{n-1}$
is an orthonormal frame of the bundle $T^{*0,1}X$. The real $(2n-2)$ form
$\omega=i^{n-1}\omega_1\wedge\ol\omega_1\wedge\cdots\wedge\omega_{n-1}\wedge\ol\omega_{n-1}$
is independent of the choice of the orthonormal frame. Thus $\omega$ is globally
defined. Locally there is a real $1$-form $\omega_0$ of length one which is orthogonal to
$T^{*1,0}X\oplus T^{*0,1}X$. The form $\omega_0$ is unique up to the choice of sign.
Since $X$ is orientable, there is a nowhere vanishing $(2n-1)$ form $Q$ on $X$.
Thus, $\omega_0$ can be specified uniquely by requiring that $\omega\wedge\omega_0=fQ$,
where $f$ is a positive function. Therefore $\omega_0$, so chosen, is globally defined.
We call $\omega_0$
the uniquely determined global real $1$-form.
We choose a vector field $T$ so that
\begin{equation}\label{s1-d11}
\norm{T}=1\,,\quad \langle T, \omega_0\rangle=-1\,.
\end{equation}
Therefore $T$ is uniquely determined. We call $T$ the uniquely determined global real vector field. We have the
pointwise orthogonal decompositions:
\begin{equation} \label{s1-e3}\begin{split}
\Complex T^*X&=T^{*1,0}X\oplus T^{*0,1}X\oplus\set{\lambda\omega_0;\,
\lambda\in\Complex},  \\
\Complex TX&=T^{1,0}X\oplus T^{0,1}X\oplus\set{\lambda T;\,\lambda\in\Complex}.
\end{split}\end{equation}
\begin{defn} \label{s1-d2}
For $p\in X$, the \emph{Levi form} $\mathcal{L}_p$ is the Hermitian quadratic form on $T^{1,0}_pX$ defined as follows. For any $U,\ V\in T^{1,0}_pX$, pick $\mathcal{U},\mathcal{V}\in
C^\infty(X;\, T^{1,0}X)$ such that
$\mathcal{U}(p)=U$, $\mathcal{V}(p)=V$. Set
\begin{equation} \label{s1-e4}
\mathcal{L}_p(U,\ol V)=\frac{1}{2i}\big\langle\big[\mathcal{U}\ ,\ol{\mathcal{V}}\,\big](p)\ ,\omega_0(p)\big\rangle\,,
\end{equation}
where $\big[\mathcal{U}\ ,\ol{\mathcal{V}}\,\big]=\mathcal{U}\ \ol{\mathcal{V}}-\ol{\mathcal{V}}\ \mathcal{U}$ denotes the commutator of $\mathcal{U}$ and $\ol{\mathcal{V}}$.
Note that $\mathcal{L}_p$ does not depend of the choices of $\mathcal{U}$ and $\mathcal{V}$.

Consider an arbitrary Hermitian metric $\langle\,\cdot\,,\cdot\,\rangle$ on $T^{1,0}X$. Since $\mathcal{L}_p$ is a Hermitian form there exists a local orthonormal basis $\{\mathcal{U}_1,\ldots,\mathcal{U}_{n-1}\}$ of
$(T^{1,0}X,\langle\,\cdot\,,\cdot\,\rangle)$ such that $\mathcal{L}_p$ is diagonal in this basis,
$\mathcal{L}_p(\mathcal{U}_i,\ol{\mathcal{U}}_j)=\delta_{ij}\lambda_i(p)$.
The diagonal entries $\{\lambda_1(p),\ldots,\lambda_{n-1}(p)\}$ are
called the \emph{eigenvalues} of the Levi form
at $p\in X$ with respect to $\langle\,\cdot\,,\cdot\,\rangle$.

Given $q\in\{0,\ldots,n-1\}$, the Levi form is said to satisfy \emph{condition $Y(q)$} at $p\in X$, if $\mathcal{L}_p$ has at least either $\max{(q+1, n-q)}$ eigenvalues of the same sign or $\min{(q+1,n-q)}$ pairs of eigenvalues with opposite signs. Note that the sign of the eigenvalues does not depend on the choice of the metric $\langle\,\cdot\,,\cdot\,\rangle$.
\end{defn}

For example, if the Levi form is non-degenerate of constant signature $(n_-,n_+)$, where $n_-$ is the number of negative eigenvalues
and $n_-+n_+=n-1$, then $Y(q)$ holds if and only if $q\neq n_-,n_+$.
\subsection{CR complex line bundles, Semi-classical $\ddbar_b\text{-Complex}$ and $\Box_b$}
Let
\begin{equation} \label{s1-e5}
\ddbar_b:\Omega^{0,q}(X)\To\Omega^{0,q+1}(X)
\end{equation}
be the tangential Cauchy-Riemann operator. We say that a function $u\in C^\infty (X)$ is Cauchy-Riemann (CR for short) if
$\ddbar_b u=0$.

\begin{defn} \label{s1-d1}
Let $L$ be a complex line bundle over $X$. We say that $L$ is a Cauchy-Riemann (CR) complex line bundle over $X$
if its transition functions are CR.
\end{defn}

From now on, we let $(L,h^L)$ be a CR Hermitian line bundle over $X$, where
the Hermitian fiber metric on $L$ is denoted by $h^L$. We will denote by
$\phi$ the local weights of the Hermitian metric. More precisely, if
$s$ is a local trivializing
section of $L$ on an open subset $D\subset X$, then the local weight of $h^L$ with respect to $s$ is the function $\phi\in C^\infty(D; \Real)$ for which
\begin{equation} \label{s1-e6}
\abs{s(x)}^2_{h^L}=e^{-\phi(x)}\,,\quad x\in D.
\end{equation}

Let $L^k$, $k>0$, be the $k$-th tensor power of the line bundle $L$. The Hermitian fiber metric on $L$ induces a Hermitian
fiber metric on $L^k$ that we shall denote by $h^{L^k}$. If $s$ is a local trivializing section
of $L$ then $s^k$ is a local trivializing section of $L^k$. For $f\in\Omega^{0,q}(X, L^k)$, we denote the poinwise norm $\abs{f(x)}^2:=\abs{f(x)}^2_{h^{L^k}}$. We write $\ddbar_{b,k}$ to denote the tangential
Cauchy-Riemann operator acting on forms with values in $L^k$, defined locally by:
\begin{equation} \label{s1-e7}
\ddbar_{b,k}:\Omega^{0,q}(X, L^k)\To\Omega^{0,q+1}(X, L^k)\,,\quad \ddbar_{b,k}(s^ku):=s^k\ddbar_bu,
\end{equation}
where $s$ is a local trivialization of $L$ on an open subset $D\subset X$
and $u\in\Omega^{0,q}(D)$.
We obtain a $\ddbar_{b,k}$-complex $(\Omega^{0,\bullet}(X, L^k),\ddbar_{b,k})$ with cohomology
\begin{equation}\label{db-cohom}
H^{\bullet}_b(X,L^k):=\ker\ddbar_{b,k}/\operatorname{Im} \ddbar_{b,k}.
\end{equation}
We denote by $dv_X=dv_X(x)$ the volume form on $X$ induced by the fixed Hermitian metric $\langle\,\cdot\,,\cdot\,\rangle$ on $\Complex TX$. Then we get natural global $L^2$ inner products $(\ |\ )_k$, $(\ |\ )$
on $\Omega^{0,q}(X, L^k)$ and $\Omega^{0,q}(X)$, respectively. We denote by $L^2_{(0,q)}(X, L^k)$ the completion of $\Omega^{0,q}(X, L^k)$ with respect to $(\ |\ )_k$. Let
\begin{equation} \label{s1-e8}
\ol{\pr}^{\,*}_{b,k}:\Omega^{0,q+1}(X, L^k)\To\Omega^{0,q}(X, L^k)
\end{equation}
be the formal adjoint of $\ddbar_{b,k}$ with respect to $(\ |\ )_k$\,. The \emph{Kohn-Laplacian} with values in $L^k$ is given by
\begin{equation} \label{s1-e9}
\Box_{b,k}^{(q)}=\ol{\pr}^{\,*}_{b,k}\ddbar_{b,k}+\ddbar_{b,k}\ol{\pr}^{\,*}_{b,k}:
\Omega^{0,q}(X, L^k)\To\Omega^{0,q}(X, L^k).
\end{equation}
We extend $\Box_{b,k}^{(q)}$ to the $L^2$ space by
$\Box_{b,k}^{(q)}:{\rm Dom\,}\Box_{b,k}^{(q)}\subset L^2_{(0,q)}(X, L^k)\To L^2_{(0,q)}(X, L^k)$,
where
${\rm Dom\,}\Box^{(q)}_{b,k}:=\{u\in L^2_{(0,q)}(X, L^k);
\Box^{(q)}_{b,k}u\in L^2_{(0,q)}(X, L^k)\}$.
Consider the space of harmonic forms
\begin{equation} \label{s1-e12}
\cH_b^q(X, L^k):=\Ker\Box^{(q)}_{b,k}\,.
\end{equation}
Now, we assume that $Y(q)$ holds. By \cite[7.6-7.8]{Ko65}, \cite[5.4.11-12]{FK72}, \cite[Props.\,8.4.8-9]{CS01}, condition $Y(q)$ implies that
$\Box^{(q)}_{b,k}$ is hypoelliptic, has compact resolvent and the strong Hodge decomposition holds. Hence
\begin{equation} \label{s2-e11}
\dim\cH_b^q(X, L^k)<\infty\,,\quad \cH_b^q(X, L^k)\subset\Omega^{0,q}(X, L^k)\,,\quad \cH_b^q(X, L^k)\cong H_b^q(X, L^k) \,.
\end{equation}

Let $f_j\in\Omega^{0,q}(X, L^k)$, $j=1,\ldots,N$, be an orthonormal frame for the space
$\cH_b^q(X, L^k)$. The \emph{Szeg\"{o} kernel function} is defined by
\begin{equation} \label{s2-e1}
\pit^{(q)}_k(x)=\sum^N_{j=1}\abs{f_j(x)}_{h^{L^k}}^2=:\sum^N_{j=1}\abs{f_j(x)}^2\,.
\end{equation}
It is easy to see that $\pit^{(q)}_k(x)$ is independent of the choice of orthonormal frame and
\begin{equation} \label{s1-e13}
\dim \cH_b^q(X, L^k)=\int_X\!\pit^{(q)}_k(x)dv_X(x).
\end{equation}
\subsection{The main results}
We will express the bound of the Szeg\"o kernel with the help of the following Hermitian form.

\begin{defn} \label{s1-d3}
Let $s$ be a local trivializing section of $L$ and $\phi$ the corresponding local weight as in \eqref{s1-e6}.
For $p\in D$ we define the Hermitian quadratic form $M^\phi_p$ on $T^{1,0}_pX$ by
\begin{equation} \label{s1-e14}
M^\phi_p(U, \ol V)=\frac{1}{2}\Big\langle U\wedge\ol V, d\big(\ddbar_b\phi-\pr_b\phi\big)(p)\Big\rangle,\ \ U, V\in T^{1,0}_pX,
\end{equation}
where $d$ is the usual exterior derivative and $\ol{\pr_b\phi}=\ddbar_b\ol\phi$.
\end{defn}

In Proposition \ref{s7-p1} we show that in the embedded case $M^\phi_p$ is the restriction of the Chern curvature of the holomorphic extension of $L$.
But in the abstract case the definition of $M^\phi_p$ depends on the choice of local trivializations.
However, set
\begin{equation}\label{s1-e15}
\begin{split}
\Real_{\phi(p),\,q}=\big\{s\in\Real;\, \text{$M^\phi_p+s\mathcal{L}_p$ has exactly $q$ negative eigenvalues} \\
\text{and $n-1-q$ positive eigenvalues}\big\}\,.
\end{split}
\end{equation}
Note that, although the eigenvalues of the Hermitian quadratic form $M^\phi_p+s\mathcal{L}_p$, $s\in\Real$,
are calculated with respect to some Hermitian metric $\langle\,\cdot\,,\cdot\,\rangle$, their sign does not depend on the choice of $\langle\,\cdot\,,\cdot\,\rangle$, cf.\ also Definition \ref{s1-d2}.

It is not difficult to see that if $Y(q)$ holds at each point of $X$ then
\begin{equation} \label{s1-e16-11}
\text{$\Real_{\phi(x),\,q}$ is locally uniformly bounded at each point $x\in X$, for all local weights $\phi$.}
\end{equation}
Note that \eqref{s5-e14} implies that if $\Real_{\phi(x),\,q}$ is bounded for one weight $\phi_0$ at $x$, then it is bounded
for all weights $\phi$ at $x$.

Denote by  $\det(M^\phi_x+s\mathcal{L}_x)$ the product of all the eigenvalues of $M^\phi_x+s\mathcal{L}_x$.
It turns out (see Proposition \ref{s5-p1}) that the integral
\begin{equation} \label{s1-e16-01}
\int_{\Real_{\phi(x),q}}\abs{\det(M^\phi_x+s\mathcal{L}_x)}ds\in\overline{\Real}
\end{equation}
does not depend on the choice of $\phi$.
 Assuming \eqref{s1-e16-11} holds, the function
\begin{equation} \label{s1-e16-1}
X\longrightarrow{\Real}\,,\quad x\longmapsto\int_{\Real_{\phi(x),q}}\abs{\det(M^\phi_x+s\mathcal{L}_x)}ds
\end{equation}
is well-defined.
Since $M^\phi_x$ and $\mathcal{L}_x$ are continuous functions of $x\in X$, we conclude that the function \eqref{s1-e16-1}
is continuous.
One of the main results of this work is the following.

\begin{thm} \label{t-main1}
Assume that 
condition $Y(q)$ holds at each point of $X$.
Then 
\begin{equation} \label{s1-e18}
\sup\big\{k^{-n}\pit^{(q)}_k(x):k\in\N,\,x\in X\big\}<\infty.
\end{equation}
Furthermore, we have
\begin{equation} \label{s1-e19}
\limsup_{k\To\infty}k^{-n}\pit^{(q)}_k(x)\leqslant\frac{1}{2(2\pi)^{n}}\int_{\Real_{\phi(x),q}}\abs{\det(M^\phi_x+s\mathcal{L}_x)}ds\,,\quad\text{for all $x\in X$.}
\end{equation}
\end{thm}
\noindent
From \eqref{s1-e13}, Theorem~\ref{t-main1} and Fatou's lemma, we get weak Morse inequalities on CR manifolds.
\begin{thm} \label{t-main2}
Assume that condition $Y(q)$ holds at each point of $X$. Then for $k\to\infty$
\begin{equation} \label{s1-e191}
\dim H^q_b(X, L^k)\leqslant  \frac{k^{n}}{2(2\pi)^{n}}\int_X\int_{\Real_{\phi(x),q}}\abs{\det(M^\phi_x+s\mathcal{L}_x)}ds\,dv_X(x)+o(k^n)\,.
\end{equation}
\end{thm}

By the classical work of Kohn \cite[Th.\,7.6]{Ko65}, \cite[Th.\,5.4.11--12]{FK72}, \cite[Cor.\,8.4.7--8]{CS01}, we know that if $Y(q)$ holds, then $\Box^{(q)}_{b,k}$ has a discrete
spectrum, each eigenvalues occurs with finite multiplicity and all eigenforms are smooth. For $\lambda\in\Real$, let
$\cH^q_{b,\,\leqslant \lambda}(X, L^k)$ denote the spectral space spanned by the eigenforms of $\Box^{(q)}_{b,k}$ whose eigenvalues are less than or equal to
$\lambda$. We denote by $\pit^{(q)}_{k,\,\leqslant \lambda}$ the restriction to the diagonal of the integral kernel of the orthogonal projector on $\cH^q_{b,\leqslant \lambda}(X,L^k)$ and call it the Szeg\"{o} kernel function of
the space $\cH^q_{b,\leqslant \lambda}(X,L^k)$. Then
$\pit^{(q)}_{k,\,\leqslant \lambda}(x)=\sum^M_{j=1}\abs{g_j(x)}^2$,
where $g_j(x)\in\Omega^{0,q}(X, L^k)$, $j=1,\ldots,M$, is any orthonormal frame for the space
$\cH_{b,\,\leqslant\lambda}^q(X, L^k)$. 

\begin{thm} \label{t-main3}
Assume that condition $Y(q)$ holds at each point of $X$. Then
for any sequence $\nu_k>0$ with $\nu_k\To0$ as $k\To\infty$, there is a constant $C'_0$ independent of $k$,
such that
\begin{equation} \label{s1-e18-1}
k^{-n}\pit^{(q)}_{k,\,\leqslant  k\nu_k}(x)\leqslant  C'_0
\end{equation}
for all $x\in X$. Moreover, there is a sequence $\mu_k>0$, $\mu_k\To0$, as $k\To\infty$, such that for any sequence
$\nu_k>0$ with $\lim_{k\To\infty}\frac{\mu_k}{\nu_k}=0$ and $\nu_k\To0$ as $k\To\infty$, we have
\begin{equation} \label{s1-e19-1}
\lim_{k\To\infty}k^{-n}\pit^{(q)}_{k,\,\leqslant  k\nu_k}(x)=\frac{1}{2(2\pi)^{n}}\int_{\Real_{\phi(x),q}}\abs{\det(M^\phi_x+s\mathcal{L}_x)}ds,
\end{equation}
for all $x\in X$.
\end{thm}

By integrating \eqref{s1-e19-1} we obtain the following semi-classical Weyl law: 

\begin{thm} \label{t-main4}
Assume that condition $Y(q)$ holds at each point of $X$. Then
there is a sequence $\mu_k>0$, $\mu_k\To0$, as $k\To\infty$, such that for any sequence
$\nu_k>0$ with $\lim_{k\To\infty}\frac{\mu_k}{\nu_k}=0$ and $\nu_k\To0$ as $k\To\infty$, we have
\[{\rm dim\,}\cH^q_{b,\,\leqslant  k\nu_k}(X, L^k)=\frac{k^{n}}{2(2\pi)^{n}}\int_X\int_{\Real_{\phi(x),q}}\abs{\det(M^\phi_x+s\mathcal{L}_x)}ds\,dv_X(x)+o(k^n).\]
\end{thm}
\noindent
From Theorem~\ref{t-main4} and the linear algebra argument from Demailly~\cite{De:85} and~ \cite{Ma96}, we obtain strong Morse inequalities on CR manifolds (see \S6):

\begin{thm} \label{t-main5}
Let $q\in\{0,\ldots,n-1\}$.
If $Y(j)$ holds for all $j=0,1,\ldots,q$, then as $k\to\infty$
\begin{equation}\label{s1-e19-11}
\begin{split}
\sum^q_{j=0}(-1)^{q-j}&{\rm dim\,}H^j_b(X, L^k)\\
&\leqslant  \frac{k^{n}}{2(2\pi)^{n}}\sum^q_{j=0}(-1)^{q-j}\int_X\int_{\Real_{\phi(x),j}}\abs{\det(M^\phi_x+s\mathcal{L}_x)}ds\,dv_X(x)+o(k^n).
\end{split}
\end{equation}
\noindent
If $Y(j)$ holds for all $j=q,q+1,\ldots,n-1$, then as $k\to\infty$
\begin{equation}\label{s1-e19-12}
\begin{split}
\sum^{n-1}_{j=q}(-1)^{q-j}&{\rm dim\,}H^j_b(X, L^k)\\
&\leqslant  \frac{k^{n}}{2(2\pi)^{n}}\sum^{n-1}_{j=q}(-1)^{q-j}\int_X\int_{\Real_{\phi(x),j}}\abs{\det(M^\phi_x+s\mathcal{L}_x)}ds\,dv_X(x)+o(k^n).
\end{split}
\end{equation}
\end{thm}

\begin{rem}
(i) Assume that the Levi form of $X$ has at least $q+1$ negative and $q+1$ positive eigenvalues, $q\in\{0,\ldots,n-2\}$.
Then $Y(j)$ hods for all $j=0,1,\ldots,q$ and $j=n-q-1,\ldots,n-1$.
\\[2pt]
\
(ii) Let $n_+, n_0, n_-\in\{0,1,\ldots,n-1\}$ with $n_++n_0+n_-=n-1$. Assume that
the Levi form of $X$ has $n_-$ everywhere negative eigenvalues, $n_+$ everywhere positive eigenvalues and $n_0$ eigenvalues
which vanish at some point on $X$. (The Levi form is non-degenerate if and only if $n_0=0$.)
Then $Y(j)$ holds for all $j\leqslant\min\{n_-,n_+\}-1$ and $j\geqslant\max\{n_-,n_+\}+n_0+1$. Thus
Theorem \ref{t-main5} shows that \eqref{s1-e19-11} holds for all $q\leqslant\min\{n_-,n_+\}-1$ and
\eqref{s1-e19-12} holds for all $q\geqslant\max\{n_-,n_+\}+n_0+1$.
\\[2pt]
\
(iii) Theorems~\ref{t-main1}\,--\ref{t-main5} have straightforward generalizations to the case when the forms take values
in $L^k\otimes E$, for a given CR vector bundle $E$ over $X$. In this case the right side gets multiplied by $\operatorname{rank}(E)$. For example, \eqref{s1-e19} becomes
\[
\limsup_{k\To\infty}k^{-n}\pit^{(q)}_k(x)\leqslant\frac{1}{2(2\pi)^{n}}\operatorname{rank}(E)\int_{\Real_{\phi(x),q}}\abs{\det(M^\phi_x+s\mathcal{L}_x)}ds\,,\quad\text{for all $x\in X$,}
\]
and similarly for other results.
\end{rem}

In section 6.1, we will state our main results in the embedded case, that is, when $X$ is a real hypersurface of a complex manifold $M$ and the bundle $L$ is the restriction of a holomorphic line bundle over $M$.
In this case the form $M^\phi_p$ is the restriction to $T^{1, 0}_pX$ of the curvature form $R^L$.
To wit, we deduce from the weak Morse inequalities (Theorem~\ref{t-main1}):

\begin{cor} \label{t-main6}
Let $M$ be a complex manifold of dimension $n$ and let $D=\{p\in
M:r(p)<0\}$ be a strongly pseudoconvex compact domain with smooth
definition function $r:M\to\Real$ which is strictly plurisubharmonic in
a neighbourhood of $X=\partial D$.  Let $(L,h^L)$ be a Hermitian
holomorphic line bundle whose curvature is proportional to the Levi form
of $D$ on $X$, i.e.\! there exists a smooth function $\lambda:X\to\Real$
such that $R^L=\lambda \mathcal{L}_r$ on the holomorphic tangent bundle of $X$. Then
$\dim H^q_b(X,L^k)=o(k^n)$ as $k\to\infty$ for all $1\leqslant  q\leqslant  n-2$.
\end{cor}

\begin{ex}\label{t-main7}
Let $N$ be a compact complex manifold of dimension $n$ and $(E,h^E)$ be
a positive line bundle on $N$.
Let $D=\{v\in E^*;|v|_{h^{E^*}}<1\}$ be the Grauert tube, set $X=\partial D$ and let
$\pi:X\to N$ be the canonical projection. Then we can apply Corollary \ref{t-main6} and we obtain that the $\ddbar_b$-cohomology of the CR line bundle $L:=\pi^*E$ satisfies $\dim H^q_b(X,L^k)=o(k^{n+1})$ as $k\to\infty$ for all
$1\leqslant  q\leqslant  n-1$.
\end{ex}


To exemplify the use of the strong Morse inequalities on CR manifolds, we formulate a condition to guarantee that high tensor powers of a CR line bundle have many CR sections in the embedded
case.

\begin{thm} \label{th-app}
Let $M'$ be a complex manifold and let $X\subset M'$ be a compact real hypersurface, $X=\rho^{-1}(0)$ for some $\rho\in C^\infty(M')$, $d\rho|_X\neq0$.
We assume that the Levi form $\mathcal{L}_x$ of $X$ 
has at least
two negative and two positive eigenvalues everywhere.
Furthermore, let $L$ be a Hermitian holomorphic
line bundle over $M'$ with curvature $R^L$. We denote by $R^L_X$ the restriction of $R^L$ to $T^{1,0}X$.
Assume that 
\begin{equation}\label{**}
\int_X\int_{\Real_{\phi(x),0}}\abs{\det(R^L_X+s\mathcal{L}_x)}dsdv_X(x)>
\int_X\int_{\Real_{\phi(x),1}}\abs{\det(R^L_X+s\mathcal{L}_x)}dsdv_X(x).
\end{equation} 
Then there is a positive constant $c$ independent of $k$, such that $\dim H^0_b(X, L^k)\geqslant ck^n$.
\end{thm} 

When $R^L$ is positive we formulate a condition to guarantee that \eqref{**} is satisfied: 

\begin{thm}\label{th-app1}
With the same notations as in Theorem~\ref{th-app}. We assume that the Levi form of $X$ 
has at least two negative and two positive eigenvalues everywhere and 
$R^L>0$. Let $\lambda_1\leq\ldots\leq\lambda_{n-1}$ be the eigenvalues of the Levi form with respect to $R^L_X$. Assume that $\lambda_{n_--1}=\lambda_{n_-}<0<\lambda_{n_-+1}=\lambda_{n_-+2}$ on $X$. 
Then there is a positive constant $c$ independent of $k$, such that ${\rm dim\,}H^0_b(X,L^k)\geq ck^n$.
\end{thm}

In \S6, we will give examples which satisfy the assumptions of Theorem~\ref{th-app1}. Now we wish to give an application in the context of pseudoconvex-pseudoconcave manifolds.
Keeping in mind the notion of $q$-pseudoconvexity and $q$-pseudoconcavity of
Andreotti-Grauert \cite{AG:62} we introduce the following.
\begin{defn}
A complex manifold $M$ with $\dim_\Complex M=n$ is called a
\emph{$(n-2)$-convex-concave strip} if there exists a smooth proper map
$\rho:M\to\Real$ whose Levi form $\partial\ddbar\rho$ has at least three
negative and three positive eigenvalues on $M$. The function $\rho$ is
called an exhaustion function.
\end{defn}
In particular an $(n-2)$-convex-concave strip is $(n-2)$-concave in the sense of Andreotti-Grauert, thus Andreotti-pseudoconcave (see \cite[Def.\,3.4.3]{MM07}). For such manifolds one can extend the concept of big line bundle, well-known in the case of compact manifolds (e.\ g.\ \cite[Def.\,2.2.5]{MM07}).
Let $L$ be a holomorphic line bundle over an Andreotti-pseudoconcave manifold. By \cite[Th.\,3.4.5]{MM07}
there exists $C>0$ such that
\begin{equation} \label{gm3.6}
\dim H^0(M,L^k)\leqslant Ck^{\,\varrho_{k}},\,\quad \text{ for $k\geqslant 1$},
\end{equation}
where $\varrho_{k}=\max_{M\setminus B_k}\operatorname{rank}\Phi_k$ is the maximum rank of the Kodaira map
\begin{equation}
\Phi_k:M\setminus B_k\to \mathbb{P}(H^0(M,L^k)^*)\,,\quad
\Phi_k(p)=\{s\in H^0(M,L^k): s(p)=0\}\,,
\end{equation}
and $B_k$ is the base locus of $H^0(M,L^k)$. We can thus define the \emph{Kodaira-Iitaka dimension} of $L$ by
$\kappa(L):=\max\{\varrho_k:k\in\N\}$. The line bundle is called \emph{big} if $\kappa(L)=\dim M$.

If $M$ is connected we can consider the field of meromorphic functions $\mathcal{M}_M$ on $M$.
Also by \cite[Th.\,3.4.5]{MM07} this
is an algebraic field of transcendence degree $a(M)$ over $\Complex$ and $\kappa(L)\leqslant a(M)\leqslant \dim M$.

\begin{thm} \label{th-strip}
Let $M$ be a connected $(n-2)$-convex-concave strip with exhaustion function
$\rho$. Let $a\in\Real$ be a regular value of $\rho$ and set
$X:=\{\rho=a\}$. Assume that there exists a holomorphic line bundle
$L\to M$ whose curvature form
$R^L$ satisfies \eqref{**}. Then the line bundle
$L$ is big. Therefore, the transcendence degree of the meromorphic
function field $\mathcal{M}_M$ equals $n=\dim_\Complex M$ and the
Kodaira map $\Phi_k:M\cdots\longrightarrow\mathbb{P}(H^0(M,L^k)^*)$ is
an immersion
outside a proper analytic set.
\end{thm}

\subsection{Sketch of the proof of Theorem \ref{t-main1}}
To simplify the exposition we consider only the case $q=0$, i.e.\,we  show how to pointwise estimate the function $\limsup\limits_{k\To\infty}\pit^{(0)}_k$.
It is easy to see that for all $x\in X$ we have
\[
\pit^{(0)}_k(x)=S^{(0)}_k(x):=\sup_{\alpha\in H^0_b(X,L^k),\norm{\alpha}=1}\abs{\alpha(x)}^2\,,
\]
where $S^{(0)}_k(x)$ is called the extremal function.
For a given point $p\in X$, by definition,
there is a sequence $u_k\in H^0_b(X, L^k)$, $\norm{u_k}=1$, such that
$\limsup_{k\To\infty}k^{-n}S^{(0)}_k(p)=\lim_{k\To\infty}k^{-n}\abs{u_k(p)}^2$.
Near $p$, take local coordinates $(x, \theta)=(z, \theta)=(x_1,\ldots,x_{2n-2},\theta)$, $z_j=x_{2j-1}+ix_{2j}$, $j=1,\ldots,n-1$, $(x(p), \theta(p))=0$, such that
$\frac{\pr}{\pr z_j}=\frac{1}{2}(\frac{\pr}{\pr x_{2j-1}}-i\frac{\pr}{\pr x_{2j}})$, $j=1,\ldots,n-1$, is an orthonormal basis for $T^{1, 0}_pX$ and
the Levi form and local weight are given by $\mathcal{L}_p=\sum ^{n-1}_{j=1}\lambda_jd z_j\otimes d\ol z_j$ and
\[\phi=\beta\theta+\sum^{n-1}_{j,\,t=1}\mu_{j,\,t}\,\ol z_jz_t+R(z)
+O(\abs{z}\abs{\theta})+O(\abs{\theta}^2)+O(\abs{(z, \theta)}^3),\]
where $R(z)=O(\abs{z}^2)$, $\frac{\pr}{\pr\ol z_j}R=0$, $j=1,\ldots,n-1$.
Let $F_k(z, \theta):=(\frac{z}{\sqrt{k}}, \frac{\theta}{k})$ be the scaling map. For $r>0$, let
$D_r=\set{(z, \theta)=(x, \theta);\, \abs{x_j}<r, \abs{\theta}<r, j=1,\ldots,2n-2}$.
Now, we consider the restriction of $u_k$ to the domain $F_k(D_{\log k})$.
The function $\alpha_k:=k^{-\frac{n}{2}}F^*_k(e^{-kR}u_k)\in C^\infty(D_{\log k})$, satisfies $\limsup_{k\To\infty}k^{-n}S^{(0)}(p)=\lim_{k\To\infty}\abs{\alpha_k(0)}^2$, where $F^*_kf\in C^\infty(D_{\log k})$ denotes the scaled function
$f\big(\frac{x}{\sqrt{k}}, \frac{\theta}{k}\big)$, $f\in C^\infty(F_k(D_{\log k}))$. Moreover, $\alpha_k$ is harmonic with respect to
the scaled Kohn-Laplacian $\Box^{(0)}_{s,(k)}$ (cf.\,\eqref{s3-e8-9-1}). The point is that $\Box^{(0)}_{s,(k)}$ converges in some sense to the model Laplacian $\Box^{(0)}_{b,H_n}$ on $H_n:=\Complex^{n-1}\times\Real$ (cf.\,\eqref{s3-e221}).
In fact, $\Box^{(0)}_{b,H_n}$ is the Kohn-Laplacian defined with respect to the CR structure $U_{j,H_n}:=\frac{\pr}{\pr z_j}-\frac{1}{\sqrt{2}}i\lambda_j\ol z_j\frac{\pr}{\pr\theta}$, $j=1,\ldots,n-1$, and the weight $e^{-\psi_0}$, $\psi_0=\beta\theta+\sum^{n-1}_{j,t=1}\mu_{j,\,t}\ol z_jz_t$.
Since $Y(q)$ holds, $\Box^{(0)}_{s,(k)}$ is hypoelliptic with loss of one derivative. Thus, the standard techniques for partial differential operators
(Rellich's theorem and Sobolev embedding theorem) yield a subsequence $\alpha_{k_j}$ converging uniformly with all the derivatives
on any compact subset of $H_n$ to a smooth function $\alpha$, which is harmonic with respect to $\Box^{(0)}_{b,H_n}$. This implies
that
\[
\limsup_{k\To\infty}k^{-n}S^{(0)}_k(p)=\abs{\alpha(0)}^2\leqslant  S^{(0)}_{H_n}(0):=
\sup_{\Box^{(0)}_{b,H_n}f=0, \norm{f}_{\psi_0}=1}\abs{f(0)}^2\,.
\]
Computing the extremal function in the model case explicitly (see \S4) finishes the proof of \eqref{s1-e19}.

This paper is organized as follows. In \S2 we first introduce the extremal function and we relate it to the Szeg\"o kernel function.
Then we introduce the scaled Kohn-Laplacian $\Box^{(q)}_{s,(k)}$ and prove the rough upper-bound for the Szeg\"o kernel function \eqref{s1-e18} (cf.\,Theorem \ref{s3-t1}). Moreover, by comparing the scaled operator $\Box^{(q)}_{s,(k)}$ to the Kohn-Laplacian $\Box^{(q)}_{b,H_n}$ on the Heisenberg group we estimate in Theorem \ref{s3-t2} the Szeg\"o kernel function on $X$ in terms of the extremal function on the Heisenberg group.
The latter is computed explicitely in \S3. In \S4 we use this information in order to prove the local Morse inequalities \eqref{s1-e19} and by integration the weak Morse inequalities \eqref{s1-e191}.
In \S5 we analyse the spectral function of $\Box^{(q)}_{b,(k)}$ and deduce the semi-classical Weyl law, thus proving Theorems \ref{t-main3}--\ref{t-main5}. In \S6 we specialize the previous results to the case of an embedded
CR manifold and prove Theorems \ref{th-app} and \ref{th-strip}. Moreover, we exemplify our results in two concrete examples, one of a Grauert tube over the torus and the other of a quotient of the Heisenberg group.

\section{The estimates of the Szeg\"{o} kernel function $\pit^{(q)}_k$} \label{morse-sec:estimates}
In this section, we assume that
condition $Y(q)$ holds at each point of $X$.
\subsection{The Szeg\"{o} kernel function $\pit^{(q)}_k(x)$ and the extremal function $S^{(q)}_{k,J}(x)$} \label{morse-sec:szego}

We first introduce some notations. For $p\in X$, we can choose a smooth orthonormal frame
$e_1,\ldots,e_{n-1}$ of $T^{*0,1}X$ over a neighborhood $U$ of $p$. We say that a multiindex $J=(j_1,\ldots,j_q)\in\{1,\ldots,n-1\}^q$ has length $q$ and write $\abs{J}=q$. We say that $J$ is strictly increasing if $1\leqslant  j_1<j_2<\cdots<j_q\leqslant  n-1$. For $J=(j_1,\ldots,j_q)$ we define $e_J:=e_{j_1}\wedge\cdots\wedge e_{j_q}$.
Then $\{e_J: \text{$\abs{J}=q$, $J$ strictly increasing}\}$
is an orthonormal frame for $\Lambda^{0,q}T^*X$ over $U$.

For $f\in\Omega^{0,q}(X, L^k)$, we may write
\[
f|_U=\sideset{}{'}\sum_{\abs{J}=q} f_Je_J\,,\quad \text{with $f_J=\langle f,e_J\rangle\in C^\infty(U;\, L^k)$}\,,
\]
where $\sum'$ means that the summation is performed only over strictly increasing multiindices. We call $f_J$ the component of $f$ along $e_J$. It will be clear from the context what frame is being used. The \emph{extremal function} $S^{(q)}_{k,J}$ along the direction $e_J$ is defined by
\begin{equation} \label{s2-e2}
S^{(q)}_{k,J}(y)=\sup_{\alpha\in\,\cH_b^q(X, L^k),\,\norm{\alpha}=1}\abs{\alpha_J(y)}^2\,.
\end{equation}
\begin{lem} \label{s2-l1}
For every local orthonormal frame $\{e_J(y); \text{$\abs{J}=q$,\,$J$ strictly increasing}\}$ of $\Lambda^{0,q}T^*X$ over an open set $U\subset X$,
we have $\pit^{(q)}_k(y)=\sum_{\abs{J}=q}'S^{(q)}_{k,J}(y)$, for every $y\in U$.
\end{lem}

\begin{proof}
Let $(f_j)_{j=1,\ldots,N}$ be an orthonormal frame for the space
$\cH_b^q(X, L^k)$. 
On $U$ we write $\pit^{(q)}_k(y)=\sum'_{\abs{J}=q}\pit^{(q)}_{k,J}(y)$, where $\pit^{(q)}_{k,J}(y):=\sum_j\abs{f_{j,J}(y)}^2$.
It is easy to see that
$\pit^{(q)}_{k,J}(y)$ is independent of the choice of the orthonormal frame $(f_j)$. Take
$\alpha\in\cH^q_b(X, L^k)$ of unit norm. Since $\alpha$ is contained in an orthonormal base, obviously $|\alpha_J(y)|^2\leqslant  \pit^{(q)}_{k,J}(y)$. Thus,
\begin{equation} \label{s2-e4}
S^{(q)}_{k,J}(y)\leqslant  \pit^{(q)}_{k,J}(y)\,,\quad\text{for all strictly increasing $J$, $\abs{J}=q$.}
\end{equation}
Fix a point $p\in U$ and a strictly incresing multiindex $J$ with $\abs{J}=q$. For simplicity, we may assume that $\phi(p)=0$. Put
\[
\textstyle
u(y)=\Big(\sum^N_{j=1}\abs{f_{j,J}(p)}^2\Big)^{-1/2}\cdot\sum^N_{j=1}\ol{f_{j,J}(p)}f_j(y)\,.
\]
We can easily check that $u\in H^q_b(X, L^k)$ and $\norm{u}=1$. Hence, $|u_{J}(p)|^2\leqslant  S^{(q)}_{k,J}(p)$, therefore
\[
\pit^{(q)}_{k,J}(p)=\sum^N_{j=1}\abs{f_{j,J}(p)}^2=|u_{J}(p)|^2\leqslant  S^{(q)}_{k,J}(p)\,.
\]
By \eqref{s2-e4}, $\pit^{(q)}_{k,J}=S^{(q)}_{k,J}$ for all strictly increasing multiindices $J$ with $\abs{J}=q$,
so the lemma follows.
\end{proof}

\subsection{The scaling technique}

For a given point $p\in X$, let $U_1(y),\ldots, U_{n-1}(y)$
be an orthonormal frame of $T^{1, 0}_yX$ varying smoothly with $y$ in a neighborhood of\, $p$,
for which the Levi form is diagonal at $p$. Furthermore, let $s$ be a local trivializing section of $L$ on an open neighborhood of $p$ and $\abs{s}_{h^L}^2=e^{-\phi}$. We take local coordinates
$(x, \theta)=(z, \theta)=(x_1,\ldots,x_{2n-2},\theta)$, $z_j=x_{2j-1}+ix_{2j}$, $j=1,\ldots,n-1$,
defined on an open set $D$ of $p$ such that
\[\omega_0(p)=\sqrt{2}d\theta\;,\quad (x(p), \theta(p))=0\,,
\]
\[
\langle\frac{\pr}{\pr x_j}(p),\frac{\pr}{\pr x_t}(p)\rangle=2\delta_{j,t}\,\quad \langle\frac{\pr}{\pr x_j}(p),\frac{\pr}{\pr\theta}(p)\rangle=0\,,\quad
\langle\frac{\pr}{\pr\theta}(p),\frac{\pr}{\pr\theta}(p)\rangle=2\,,
\]
for $j, t=1,\ldots,2n-2$,
\begin{equation} \label{s1-e20}
U_j=\frac{\pr}{\pr z_j}-\frac{1}{\sqrt{2}}i\lambda_j\ol z_j\frac{\pr}{\pr\theta}-
\frac{1}{\sqrt{2}}c_j\theta\frac{\pr}{\pr\theta}+O(\abs{(z, \theta)}^2),\ \ j=1,\ldots,n-1,
\end{equation}
and
\begin{equation} \label{s1-e21}
\begin{split}
\phi=&\sum^{n-1}_{j=1}(\alpha_j z_j+\ol\alpha_j\ol z_j)+\beta\theta+\sum^{n-1}_{j,t=1}(a_{j,t}z_jz_t+\ol a_{j,t}\ol z_j\ol z_t)
+\sum^{n-1}_{j,t=1}\mu_{j,\,t}\ol z_jz_t\\
&+O(\abs{z}\abs{\theta})+O(\abs{\theta}^2)+O(\abs{(z, \theta)}^3),
\end{split}
\end{equation}
where
$\beta\in\Real, c_j, \alpha_j, a_{j,t}, \mu_{j,\,t}\in\Complex$, $\delta_{j,t}=1$ if $j=t$, $\delta_{j,t}=0$ if $j\neq t$, $\frac{\pr}{\pr z_j}=\frac{1}{2}(\frac{\pr}{\pr x_{2j-1}}-i\frac{\pr}{\pr x_{2j}})$, for $j,t=1,\ldots,n-1$ and $\lambda_j$, $j=1,\ldots,n-1$, are the eigenvalues of\, $\mathcal{L}_p$.
This is always possible, see \cite[p.\,157--160]{BG88}.
In this section, we work with this local coordinates and we identify $D$ with some open set in $\Real^{2n-1}$. Put
\begin{gather}
R(z, \theta)=\sum^{n-1}_{j=1}\alpha_j z_j+\sum^{n-1}_{j,t=1}a_{j,t}z_jz_t\,,\label{s1-e22}\\
\phi_0=\phi-R(z, \theta)-\ol{R(z, \theta)} =\beta\theta+\sum^{n-1}_{j,t=1}\mu_{j,\,t}\ol z_jz_t+O(\abs{z}\abs{\theta})+O(\abs{\theta}^2)+O(\abs{(z, \theta)}^3)\,.\label{s1-e23}
\end{gather}

Let $(\ |\ )_{k\phi}$ and $(\ |\ )_{k\phi_0}$ be the inner products on the space
$\Omega^{0,q}_c(D)$ defined as follows:
\[
(f\ |\ g)_{k\phi}=\int_D\!\langle f,g\rangle e^{-k\phi}dv_X\,, \quad (f\ |\ g)_{k\phi_0}=\int_D\!\langle f, g\rangle e^{-k\phi_0}dv_X\,,
\]
where $f, g\in\Omega^{0,q}_c(D)$. We denote by $L^2_{(0,q)}(D, k\phi)$ and $L^2_{(0,q)}(D, k\phi_0)$ the completions of $\Omega^{0,q}_c(D)$ with respect to $(\ |\ )_{k\phi}$ and $(\ |\ )_{k\phi_0}$, respectively.
We have the unitary identification
\begin{equation} \label{s1-e27}
\left\{\begin{aligned}
L^2_{(0,q)}(D, k\phi_0)&\leftrightarrow L^2_{(0,q)}(D, k\phi) \\
u&\rightarrow \Td u=e^{kR}u, \\
u=e^{-kR}\Td u&\leftarrow \Td u.
\end{aligned}
\right.
\end{equation}
Let
$\ddbar^{\,*,k\phi}_b:\Omega^{0,q+1}(D)\To\Omega^{0,q}(D)$
be the formal adjoint of $\ddbar_b$ with respect to $( \ |\ )_{k\phi}$. Put
\[
\Box^{(q)}_{b,k\phi}=\ddbar_b\ddbar^{\,*,k\phi}_b+\ddbar^{\,*,k\phi}_b\ddbar_b:\Omega^{0,q}(D)\To\Omega^{0,q}(D)\,.
\]
Let $u\in\Omega^{0,q}(D, L^k)$. Then there exists $\hat u\in\Omega^{0,q}(D)$ such that $u=s^k\hat u$ and we have $\Box^{(q)}_{b,k}u=s^k\Box^{(q)}_{b,k\phi}\hat u$.
In this section, we identify $u$ with $\hat u$ and $\Box^{(q)}_{b,k}$ with $\Box^{(q)}_{b, k\phi}$. Note that $\abs{u(0)}^2=\abs{\hat u(0)}^2e^{-k\phi(0)}=\abs{\hat u(0)}^2$.

Recall that $\alpha\wedge$ is the operator of left exterior multiplication with a form $\alpha$.
The adjoint of this operator is denoted by $(\alpha\wedge)^*$ (cf.\ \eqref{s1-e1}).

If $u\in\Omega^{0,q}(D)\cap L^2_{(0,q)}(D, k\phi_0)$, using \eqref{s1-e27}, we have
$\ddbar_b\Td u=\Td{\ddbar_s u}=e^{kR}\ddbar_su$,
where
\begin{equation} \label{s3-e2-5}
\ddbar_s=\ddbar_b+k(\ddbar_bR)\wedge\;.
\end{equation}
Let $(e_j(z, \theta))_{j=1,\ldots,n-1}$ denote the basis of $T^{*0,1}_{(z,\theta)}X$,
dual to $(\ol U_j(z,\theta))_{j=1,\ldots,n-1}$. Then
$\ddbar_b=\sum^{n-1}_{j=1}\bigr(e_j\wedge\ol U_j+(\ddbar_be_j)\wedge (e_j\wedge)^*\bigr)$. Note that $(e_j\wedge)^*$ equals the interior product $i_{\ol{U}_j}$ with $\ol{U}_j$.
Thus,
\begin{equation} \label{s3-e4}
\ddbar_s=\sum^{n-1}_{j=1}e_j\wedge\Bigr(\ol U_j+k(\ol U_jR)\Bigr)
+\sum^{n-1}_{j=1}(\ddbar_b e_j)\wedge (e_j\wedge)^*
\end{equation}
and correspondingly
\begin{equation} \label{s3-e5}
\ddbar^{\,*}_s=\sum^{n-1}_{j=1}(e_j\wedge)^*\Bigr(\ol U^{\,*,k\phi_0}_j+k(U_j\ol R)\Bigr)
+\sum^{n-1}_{j=1}e_j\wedge(\ddbar_b e_j\wedge)^*,
\end{equation}
where $\ddbar^{\,*,k\phi}_b\Td u=e^{kR}\ddbar^{\,*}_su$ and  $\ol U^{\,*,k\phi_0}_j$ is the formal adjoint of $\ol U_j$ with respect to $(\ |\ )_{k\phi_0}$, $j=1,\ldots,n-1$. We can check that
\begin{equation} \label{s3-e6}
\ol U_j^{\,*,k\phi_0}=-U_j+k(U_j\phi_0)+s_j(z, \theta),
\end{equation}
where $s_j\in C^\infty(D)$, $s_j$ is independent of $k$, $j=1,\ldots,n-1$. Put
\begin{equation} \label{s3-e7}
\Box^{(q)}_s=\ddbar_s\ddbar^{\,*}_s+\ddbar^{\,*}_s\ddbar_s:\Omega^{0,q}(D)\To\Omega^{0,q}(D).
\end{equation}
We have
\begin{equation} \label{s3-e7-0}
\Td{\Box^{(q)}_s u}=e^{kR}\Box^{(q)}_su=\Box^{(q)}_{b,k\phi}\Td u.
\end{equation}
\begin{prop}[{\cite[Prop.\,2.3]{Hsiao08}}] \label{s3-p0}
We have
\begin{equation} \label{s3-e7-1}
\begin{split}
\Box^{(q)}_s &=\ddbar_s\ddbar^{\,*}_s+\ddbar^{\,*}_s\ddbar_s  \\
       &=\sum^{n-1}_{j=1}\bigr(\ol U^{\,*,k\phi_0}_j+k(U_j\ol R)\bigr)\bigr(\ol U_j+k(\ol U_jR)\bigr)\\
       &+ \sum^{n-1}_{j,\,t=1}e_j\wedge (e_t\wedge)^*\big[\ol U_j+k(\ol U_j R)\ ,\ol U^{\,*,k\phi_0}_t+k(U_t\ol R)\big] \\
         &\quad +\epsilon(\ol U+k(\ol U R))+\epsilon(\ol U^{\,*,k\phi_0}+k(U\ol R))+f(z, \theta),
\end{split}
\end{equation}
where $\epsilon(\ol U+k(\ol U R))$ denotes remainder terms of the form $\sum a_j(z, \theta)\bigr(\ol U_j+k(\ol U_j R)\bigr)$ with $a_j$ smooth, matrix-valued and independent of $k$, for all $j$, and similarly for $\epsilon(\ol U^{\,*,k\phi_0}+k(\ol U R))$ and $f(z, \theta)\in C^\infty$ independent of $k$.
\end{prop}

For the convenience of the reader we recall some notations we used before. For $r>0$, let
$D_r=\set{(z, \theta)=(x, \theta)\in\Real^{2n-1};\, \abs{x_j}<r,\ \abs{\theta}<r,\ j=1,\ldots,2n-2}$.
Let $F_k$ be the scaling map:
$F_k(z, \theta)=(\frac{z}{\sqrt{k}}, \frac{\theta}{k})$.
From now on, we assume that $k$ is large enough so that $F_k(D_{\log k})\subset D$.
We define the scaled bundle $F^*_k\Lambda^{0,q}T^*X$ on $D_{\log k}$ to be the bundle whose fiber at $(z, \theta)\in D_{\log k}$ is
\[
F^*_k\Lambda^{0,q}T^*_{(z,\theta)}X:=\Bigr\{\textstyle\sideset{}{'}\sum_{\abs{J}=q}a_Je_J(\frac{z}{\sqrt{k}},\frac{\theta}{k});\, a_J\in\Complex,\abs{J}=q\Bigl\}\,.
\]
We take the Hermitian metric $\langle\,\cdot\,,\cdot\,\rangle_{F^*_k}$ on $F^*_k\Lambda^{0,q}T^*X$ so that at each point $(z, \theta)\in D_{\log k}$\,,
\[
\Bigr\{e_J\big(\tfrac{z}{\sqrt{k}}\;,\tfrac{\theta}{k}\big)\,; \text{$\abs{J}=q$, $J$ strictly increasing}\Bigl\}\,,
\]
is an orthonormal basis for $F^*_k\Lambda^{0,q}T^*_{(z,\theta)}X$. For $r>0$, let $F^*_k\Omega^{0,q}(D_r)$
denote the space of smooth sections of $F^*_k\Lambda^{0,q}T^*X$ over $D_r$. Let $F^*_k\Omega^{0,q}_c(D_r)$ be the subspace of
$F^*_k\Omega^{0,q}(D_r)$ whose elements have compact support in $D_r$.
Given $f\in\Omega^{0,q}(F_k(D_{\log k}))$ we write
$f=\sum'_{\abs{J}=q}f_Je_J$.
We define the scaled form $F_k^*f\in F^*_k\Omega^{0,q}(D_{\log k})$ by:
\[
F_k^*f=\sideset{}{'}\sum_{\abs{J}=q}f_J\Big(\frac{z}{\sqrt{k}}\;, \frac{\theta}{k}\Big)e_J\Big(\frac{z}{\sqrt{k}}\;,\frac{\theta}{k}\Big)\,.
\]

Let $P$ be a partial differential operator of order one on $F_k(D_{\log k})$ with $C^\infty$ coefficients. We write
$P=a(z, \theta)\frac{\pr}{\pr\theta}+\sum^{2n-2}_{j=1}a_j(z, \theta)\frac{\pr}{\pr x_j}$, $a, a_j\in C^\infty(F_k(D_{\log k}))$, $j=1,\ldots,2n-2$. The partial diffferential operator
$P_{(k)}$ on $D_{\log k}$ is given by
\begin{equation} \label{s3-e8-4}
P_{(k)}=\sqrt{k}F^*_ka\frac{\pr}{\pr\theta}+\sum^{2n-2}_{j=1}F^*_ka_j\frac{\pr}{\pr x_j}
=\sqrt{k}a\Big(\frac{z}{\sqrt{k}}\,,\frac{\theta}{k}\Big)\frac{\pr}{\pr\theta}+\sum^{2n-2}_{j=1}a_j\Big(\frac{z}{\sqrt{k}}\,, \frac{\theta}{k}\Big)\frac{\pr}{\pr x_j}\,.
\end{equation}
Let $f\in C^\infty(F_k(D_{\log k}))$. We can check that
\begin{equation} \label{s3-e8-5}
P_{(k)}(F^*_kf)=\frac{1}{\sqrt{k}}F^*_k(Pf).
\end{equation}

The scaled differential operator $\ddbar_{s,(k)}:F^*_k\Omega^{0,q}(D_{\log k})\To F^*_k\Omega^{0,q+1}(D_{\log k})$ is given by (compare to the formula \eqref{s3-e4} for $\ddbar_s$):
\begin{equation} \label{s3-e8-6}
\begin{split}
\ddbar_{s,(k)}=&\sum^{n-1}_{j=1}e_j\Big(\frac{z}{\sqrt{k}}\,,\frac{\theta}{k}\Big)\wedge\Bigr(\ol U_{j(k)}+\sqrt{k}F^*_k(\ol U_jR)\Bigr)\\
&+\sum^{n-1}_{j=1}\frac{1}{\sqrt{k}}(\ddbar_b e_j)\Big(\frac{z}{\sqrt{k}}\,,\frac{\theta}{k}\Big)\wedge \Big(e_j\Big(\frac{z}{\sqrt{k}}\,,\frac{\theta}{k}\Big)\wedge\Big)^*.
\end{split}
\end{equation}
From \eqref{s3-e4} and \eqref{s3-e8-5}, we can check that if $f\in\Omega^{0,q}(F_k(D_{\log k}))$, then
\begin{equation} \label{s3-e8-7}
\ddbar_{s,(k)}F^*_kf=\frac{1}{\sqrt{k}}F^*_k(\ddbar_sf).
\end{equation}

Let $(\ |\ )_{kF^*_k\phi_0}$ be the inner product on the space $F^*_k\Omega^{0,q}_c(D_{\log k})$
defined as follows:
\[
(f\ |\ g)_{kF^*_k\phi_0}=\int_{D_{\log k}}\!\langle f,g\rangle_{F^*_k}e^{-kF^*_k\phi_0}(F^*_km)(z, \theta)dv(z)dv(\theta)\,,
\]
where $dv_X=mdv(z)dv(\theta)$ is the volume form, $dv(z)=2^{n-1}dx_1\cdots dx_{2n-2}$, $dv(\theta)=\sqrt{2}d\theta$. Note that $m(0,0)=1$. Let $\ddbar^{\,*}_{s,(k)}:F^*_k\Omega^{0,q+1}(D_{\log k})\To F^*_k\Omega^{0,q}(D_{\log k})$ be the formal adjoint of $\ddbar_{s,(k)}$ with respect to $(\ |\ )_{kF^*_k\phi_0}$. Then, we can check that (compare the formulas for $\ddbar^{\,*}_s$, see \eqref{s3-e5} and \eqref{s3-e6})
\begin{equation} \label{s3-e8-8}
\begin{split}
\ddbar^{\,*}_{s,(k)}=&\sum^{n-1}_{j=1}\Big(e_j\Big(\frac{z}{\sqrt{k}},\frac{\theta}{k}\Big)\wedge\Big)^*\Bigr(-U_{j(k)}+\sqrt{k}F^*_k( U_j\ol R)+\sqrt{k}F^*_k(U_j\phi_0)+\frac{1}{\sqrt{k}}F^*_ks_j\Bigr)\\
&+\sum^{n-1}_{j=1}\frac{1}{\sqrt{k}}\,e_j\Big(\frac{z}{\sqrt{k}}\,,\frac{\theta}{k}\Big)\wedge\Big((\ddbar_b e_j)\Big(\frac{z}{\sqrt{k}}\,,\frac{\theta}{k}\Big)\wedge\Big)^*,
\end{split}
\end{equation}
where $s_j\in C^\infty(D_{\log k})$, $j=1,\ldots,n-1$, are independent of $k$.
We also have
\begin{equation} \label{s3-e8-9}
\ddbar^{\,*}_{s,(k)}F^*_kf=\frac{1}{\sqrt{k}}F^*_k(\ddbar^{\,*}_sf),\,\quad f\in\Omega^{0,q+1}(F_k(D_{\log k}))\,.
\end{equation}
We define now the \emph{scaled Kohn-Laplacian}:
\begin{equation} \label{s3-e8-9-1}
\Box^{(q)}_{s,(k)}:=\ddbar^{\,*}_{s,(k)}\ddbar_{s,(k)}+\ddbar_{s,(k)}\ddbar^{\,*}_{s,(k)}:F^*_k\Omega^{0,q}(D_{\log k})\To F^*_k\Omega^{0,q}(D_{\log k}).
\end{equation}
From \eqref{s3-e8-7} and \eqref{s3-e8-9}, we see that if $f\in\Omega^{0,q}(F_k(D_{\log k}))$, then
\begin{equation} \label{s3-e9}
(\Box^{(q)}_{s,(k)})F^*_kf=\frac{1}{k}F^*_k(\Box^{(q)}_sf).
\end{equation}

From \eqref{s1-e20} and \eqref{s1-e22}, we can check that
\begin{equation} \label{s3-e10}
\ol U_{j(k)}+\sqrt{k}F^*_k(\ol U_j R)=\frac{\pr}{\pr\ol z_j}+\frac{1}{\sqrt{2}}i\lambda_jz_j\frac{\pr}{\pr \theta}+\epsilon_kZ_{j,k}\,,\quad j=1,\ldots,n-1,
\end{equation}
on $D_{\log k}$, where $\epsilon_k$ is a sequence tending to zero with $k\To\infty$ and $Z_{j,k}$ is a first order differential operator and all the derivatives of the coefficients of $Z_{j,k}$ are uniformly bounded in $k$ on $D_{\log k}$, $j=1,\ldots,n-1$. Similarly, from \eqref{s1-e22} and \eqref{s1-e23}, we can check that
\begin{equation} \label{s3-e11}
\begin{split}
&-U_{t(k)}+\sqrt{k}F^*_k( U_t\ol R)+\sqrt{k}F^*_k(U_t\phi_0)+\frac{1}{\sqrt{k}}F^*_ks_t\\
&=-\frac{\pr}{\pr z_t}+\frac{1}{\sqrt{2}}\,i\lambda_t\ol z_t\frac{\pr}{\pr\theta}-\frac{1}{\sqrt{2}}\,i\lambda_t\ol z_t\beta+\sum^{n-1}_{j=1}\mu_{j,\,t}\,\ol{z}_j+\delta_kV_{t,\,k}\,,\quad  t=1,\ldots,n-1,
\end{split}
\end{equation}
on $D_{\log k}$, where $\delta_k$ is a sequence tending to zero with $k\To\infty$ and $V_{t,k}$ is a first order differential operator and all the derivatives of the coefficients of $V_{t,k}$ are uniformly bounded in $k$ on $D_{\log k}$, $t=1,\ldots,n-1$.
From \eqref{s3-e10}, \eqref{s3-e11} and \eqref{s3-e8-6}, \eqref{s3-e8-8}, \eqref{s3-e8-9-1}, it is straightforward to obtain the following.

\begin{prop} \label{s3-p1}
We have that
\begin{equation*} 
\begin{split}
\Box&^{(q)}_{s,(k)}=\sum^{n-1}_{j=1}\Bigr[\Bigr(-\frac{\pr}{\pr z_j}+\frac{i}{\sqrt{2}}\lambda_j\ol z_j\frac{\pr}{\pr \theta}-\frac{i}{\sqrt{2}}\lambda_j\ol z_j\beta+\sum^{n-1}_{t=1}\mu_{t,\,j}\,\ol z_t\Bigr)\Bigr(\frac{\pr}{\pr\ol z_j}+\frac{i}{\sqrt{2}}\lambda_jz_j\frac{\pr}{\pr\theta}\Bigr)\Bigr]\\
&
+\sum^{n-1}_{j,\,t=1}e_j\Big(\frac{z}{\sqrt{k}}\,, \frac{\theta}{k}\Big)\wedge \Big(e_t\Big(\frac{z}{\sqrt{k}}\,,\frac{\theta}{k}\Big)\wedge\Big)^*\Bigr(\Bigr(\mu_{j,\,t}-\frac{i}{\sqrt{2}}\lambda_j\delta_{j,\,t}\beta\Bigr)+\sqrt{2}i\lambda_j\delta_{j,\,t}\,\frac{\pr}{\pr \theta}\Bigr)+\varepsilon_kP_k,
\end{split}
\end{equation*}
on $D_{\log k}$, where $\varepsilon_k$ is a sequence tending to zero with $k\To\infty$, $P_k$ is a second order differential operator and all the derivatives of the coefficients of $P_k$ are uniformly bounded in $k$ on $D_{\log k}$.
\end{prop}

Let $D\subset D_{\log k}$ be an open set and let $W^s_{kF^*_k\phi_0}(D;\, F^*_k\Lambda^{0, q}T^*X)$,
$s\in\mathbb N_0:=\mathbb N\cup\set{0}$, denote the Sobolev space of order $s$ of sections of $F^*_k\Lambda^{0,q}T^*X$
over $D$ with respect to the weight $e^{-kF^*_k\phi_0}$. The Sobolev norm on this space is given by
\begin{equation} \label{s1-e37}
\norm{u}^2_{kF^*_k\phi_0,s,D}
=\sideset{}{'}\sum_{\substack{\alpha\in\mathbb{N}^{2n-1}_0,\;\abs{\alpha}\leqslant  s\;\\{\abs{J}=q}}}
\int_{D}\!\abs{\pr^\alpha_{x,\theta}u_J}^2e^{-kF^*_k\phi_0}(F^*_km)(z, \theta)dv(z)dv(\theta),
\end{equation}
where
$u=\sum'_{\abs{J}=q}u_Je_J\big(\frac{z}{\sqrt{k}},\frac{\theta}{k}\big)\in W^s_{kF^*_k\phi_0}
(D;\, F^*_k\Lambda^{0,q}T^*X)$ and $m$ is the volume form.
If $s=0$, we write $\norm{\cdot}_{kF^*_k\phi_0,D}$ to denote $\norm{\cdot}_{kF^*_k\phi_0,0,D}$.
We need the following

\begin{prop} \label{s3-p2}
For every $r>0$ with $D_{2r}\subset D_{\log k}$ and $s\in\mathbb N\cup\{0\}$, there is a constant $C_{r,s}>0$
independent of $k$, such that 
\begin{equation} \label{s3-e12}
\norm{u}^2_{kF^*_k\phi_0,s+1,D_{r}}\leqslant  C_{r,s}\Bigr(\norm{u}^2_{kF^*_k\phi_0,D_{2r}}+\big\|\Box^{(q)}_{s,(k)}u\big\|^2_{kF^*_k\phi_0,s,D_{2r}}\Bigl)\,,\;
u\in F^*_k\Omega^{0,q}(D_{\log k})\,.
\end{equation}
\end{prop}

\begin{proof}
Since $Y(q)$ holds, we see from the classical work of Kohn (\cite[Th.\,7.7]{Ko65}, \cite[Prop.\,5.4.10]{FK72}, \cite[Th.\,8.4.3]{CS01}), that
$\Box^{(q)}_{s,(k)}$ is hypoelliptic with loss of one derivative and we have \eqref{s3-e12}.
Since all the derivatives of the coefficients of the operator $\Box^{(q)}_{s,(k)}$ are uniformly bounded in $k$,
if we go through the proof of~\cite[pp.\,193--199]{CS01} (see also Remark~\ref{s3-r1} below), it is
straightforward to see that $C_{r,\,s}$ can be taken to be independent of $k$. 
\end{proof}

\begin{rem} \label{s3-r1}
Put
\begin{equation*}
\begin{split}
A=\{&\mbox{all the coefficients of $\Box^{(q)}_{s,(k)}$, $\ddbar_{s,(k)}$, $\ddbar^{\,*}_{s,(k)}$,
$\big[\,\ol U_{j(k)}\ ,U_{t(k)}\big]$, $\ol U_{j(k)}$, $U_{t(k)}$},\\
&j,t=1,\ldots,n-1,\ \mbox{and of $kF^*_k\phi_0$, $F^*_km$}\}
\end{split}
\end{equation*}
and $B=\set{\mbox{all the eigenvalues of $\mathcal{L}_p$}}$.
From the proof of Kohn, we see that for $r>0$, $s\in\mathbb{N}_0$, there exist a semi-norm $P$ on $C^\infty(D_{2r})$ and a strictly positive continuous function $F:\Real\To\Real_+$ such that
\begin{equation} \label{s3-e12-1}
\norm{u}^2_{kF^*_k\phi_0,s+1,D_{r}}\leqslant \Bigr(\textstyle\sum\limits_{f\in A}F(P(f))+\sum\limits_{\lambda\in B}F(\lambda)\Bigr) \Bigr(\big\|u\big\|^2_{kF^*_k\phi_0,D_{2r}}+\big\|\Box^{(q)}_{s,(k)}u\big\|^2_{kF^*_k\phi_0,s,D_{2r}}\Bigr),
\end{equation}
where $u\in F^*_k\Omega^{0,q}(D_{\log k})$.
Roughly speaking, the constant $C_{r,s}$ in \eqref{s3-e12} depends continuously on the eigenvalues of $\mathcal{L}_p$ and the elements of $A$ in the $C^\infty(D_{2r})$ topology. (See also the proof of \cite[Lemma\,4.1]{SW05}.)
\end{rem}
\begin{lem} \label{s3-l1}
Let $\alpha_k\in F^*_k\Omega^{0,q}(D_{\log k})$ with $\Box^{(q)}_{s,(k)}\alpha_k=0$ and $\norm{\alpha_k}_{kF^*_k\phi_0,D_{\log k}}\leqslant  1$. Then, there is a constant $C>0$ such that for all $k$ we have
$\abs{\alpha_k(0)}^2\leqslant  C$.
\end{lem}

\begin{proof}
Fix $r>0$, $r$ small and let $\chi\in C^\infty_0(D_{r})$, $\chi=1$ on $D_{\frac{r}{2}}$. Identify $\alpha_k$ with a form in $\Real^{2n-1}$ by extending with zero. Then
\begin{align*}
\abs{\chi(0)\alpha_k(0)}&=\abs{\int_{\Real^{2n-1}}\!\widehat{\chi\alpha_k}(\xi)d\xi}
=\abs{\int_{\Real^{2n-1}}\!(1+\abs{\xi}^2)^{-\frac{n}{2}}(1+\abs{\xi}^2)^{\frac{n}{2}}\widehat{\chi\alpha_k}(\xi)d\xi}\\
&\quad\leqslant \Bigl(\int_{\Real^{2n-1}}\!(1+\abs{\xi}^2)^{-n}d\xi\Bigr)^{\frac{1}{2}}\Bigl(\int_{\Real^{2n-1}}\!
(1+\abs{\xi}^2)^{n}\abs{\widehat{\chi\alpha_k}(\xi)}^2d\xi\Bigr)^{\frac{1}{2}}\\
&\quad\leqslant \Td c\norm{\alpha_k}_{kF^*_k\phi_0,n,D_r},
\end{align*}
where $\widehat{\chi\alpha_k}$ denotes the Fourier transform of $\chi\alpha_k$.
From \eqref{s3-e12} and using induction, we get
\[\norm{\alpha_k}^2_{kF^*_k\phi_0,n,D_r}\leq C\Bigr(\norm{\alpha_k}^2_{kF^*_k\phi_0,D_{r'}}+\sum^n_{m=1}\norm{(\Box^{(q)}_{s,(k)})^mu}^2_{kF^*_k\phi_0,D_{r'}}\Bigr)\]
for some $r'>0$, where $C>0$ is independent of $k$. Since $\Box^{(q)}_{s,(k)}u=0$, we conclude that
$\norm{\alpha_k}_{kF^*_k\phi_0,n,D_r}\leqslant  C$. The lemma follows.
\end{proof}

Now, we can prove the first part of Theorem \ref{t-main1} (estimate \eqref{s1-e18}).

\begin{thm} \label{s3-t1}
There is a constant $C_0>0$ such that for all $k$ and all $x\in X$ we have
\begin{equation} \label{s3-e14}
k^{-n}\pit^{(q)}_k(x)\leqslant  C_0
\end{equation}
\end{thm}

\begin{proof}
Let $u_k\in H^q_b(X, L^k)$, $\norm{u_k}=1$. Set
$\alpha_k:=k^{-\frac{n}{2}}F^*_k(e^{-kR}u_k)\in F^*_k\Omega^{0,q}(D_{\log k})$.
We recall that $R$ is given by \eqref{s1-e22}. (See also \eqref{s1-e27}.) We check that
$\norm{\alpha_k}_{kF^*_k\phi_0,D_{\log k}}\leqslant  1$.
Using \eqref{s3-e9} and \eqref{s3-e7-0}, it is not difficult to see that $\Box^{(q)}_{s,(k)}\alpha_k=0$ on $D_{\log k}$. From this and Lemma~\ref{s3-l1}, we see that there exists $C(0)>0$ such that for all $k$ we have $\abs{\alpha_k(0)}^2=k^{-n}\abs{u_k(0)}^2\leqslant  C(0)$.
We can apply this procedure for each point $x\in X$, so we can replace $0$ by $x$
in the previous estimate.
In view of Remark~\ref{s3-r1}, we see that we can find $C(x)>0$ and a neighborhood $D_x$ of $x$ such that
for all $k$ and all $y\in D_x$ we have $k^{-n}\abs{u_k(y)}^2\leqslant  C(x)$.
Since $X$ is compact we infer that
\[
C'_0=\sup\{k^{-n}\abs{u_k(x)}^2:k\in\N,x\in X\}<\infty
\]
Thus, for a local orthonormal frame $\{e_J;\text{$\abs{J}=q$, $J$ strictly increasing}\}$ we have \[\sup\{k^{-n}S^{(q)}_{k,J}(x):k\in\N,x\in X\}\leqslant  C'_0\] (see \eqref{s2-e2} for the definition of $S^{(q)}_{k,J}$). From this and Lemma~\ref{s2-l1}, the theorem follows.
\end{proof}

\subsection{The Heisenberg group $H_n$}

We pause and introduce some notations. We identify $\Real^{2n-1}$ with the Heisenberg group $H_n:=\Complex^{n-1}\times\Real$. We also write $(z, \theta)$ to denote the coordinates of $H_n$, $z=(z_1,\ldots,z_{n-1})\in\Complex^{n-1}$, $z_j=x_{2j-1}+ix_{2j}$, $j=1,\ldots,n-1$, and $\theta\in\Real$. Then
\[
\begin{split}
&\Big\{U_{j,H_n}=\frac{\pr}{\pr z_j}-\frac{1}{\sqrt{2}}i\lambda_j\ol z_j\frac{\pr}{\pr\theta}\,;\, j=1,\ldots,n-1\Big\}\,\\
&\Big\{U_{j,H_n}\,,\ \ol{U}_{j,H_n}\,,\ T=-\frac{1}{\sqrt{2}}\frac{\pr}{\pr\theta}\, ;\,j=1,\ldots,n-1\Big\}
\end{split}
\]
are orthonormal bases for the bundles $T^{1, 0}H_n$ and $\Complex TH_n$, respectively. Then
\[
\Big\{dz_j\, ,\ d\ol z_j\, ,\ \omega_0=\sqrt{2}d\theta+\sum^{n-1}_{j=1}(i\lambda_j\ol z_jdz_j-i\lambda_jz_jd\ol z_j)\,;j=1,\ldots,n-1\Big\}
\]
is the basis of $\Complex T^*H_n$ which is dual to $\{U_{j,H_n},\ol U_{j,H_n}, -T;j=1,\ldots,n-1\}$.
We take the Hermitian metric $\langle\,\cdot\,,\cdot\,\rangle$ on $\Lambda^{0,q}T^*H_n$ such that $\{d\ol z_J: \text{$\abs{J}=q$, $J$ strictly increasing}\}$ is an orthonormal basis of $\Lambda^{0,q}T^*H_n$. The Cauchy-Riemann operator $\ddbar_{b,H_n}$ on $H_n$ is given by
\begin{equation} \label{s3-e16}
\ddbar_{b,H_n}=\sum^{n-1}_{j=1}d\ol z_j\wedge\ol U_{j,H_n}:\Omega^{0,q}(H_n)\To\Omega^{0,q+1}(H_n).
\end{equation}
Put
$\psi_0(z, \theta)=\beta\theta+\sum^{n-1}_{j,t=1}\mu_{j,\,t}\ol z_jz_t\in C^\infty(H_n;\, \Real)$,
where $\beta$ and $\mu_{j,\,t}$, $j,t=1,\ldots,n-1$, are as in \eqref{s1-e21}. Note that
\begin{equation} \label{s3-e17-1}
\sup_{(z, \theta)\in D_{\log k}}\abs{kF^*_k\phi_0-\psi_0}\To0,\ \ \mbox{as $k\To\infty$}.
\end{equation}

Let $(\ |\ )_{\psi_0}$
be the inner product on $\Omega^{0,q}_c(H_n)$ defined as follows:
\[
(f\ |\ g)_{\psi_0}=\int_{H_n}\!\langle f,g\rangle e^{-\psi_0}dv(z)dv(\theta)\,, \quad f, g\in\Omega^{0,q}_c(H_n)\,,
\]
where $dv(z)=2^{n-1}dx_1dx_2\cdots dx_{2n-2}$, $dv(\theta)=\sqrt{2}d\theta$. Let $\ddbar^{\,*,\psi_0}_{b,H_n}:\Omega^{0,q+1}(H_n)\To\Omega^{0,q}(H_n)$
be the formal adjoint of $\ddbar_{b,H_n}$ with respect to $(\ |\ )_{\psi_0}$. We have
\begin{equation} \label{s3-e21}
\ddbar^{\,*,\psi_0}_{b,H_n}=\sum^{n-1}_{t=1}(d\ol z_t\wedge)^*\,\ol U_{t,H_n}^{\,*,\psi_0}:\Omega^{0,q+1}(H_n)\To \Omega^{0,q}(H_n),
\end{equation}
where
\begin{equation} \label{s3-e22}
\ol U_{t,H_n}^{\,*,\psi_0}=-U_{t,H_n}+U_{t,H_n}\psi_0=-U_{t,H_n}+\sum^{n-1}_{j=1}\mu_{j,\,t}\ol z_j-\frac{1}{\sqrt{2}}i\lambda_t\ol z_t\beta\,.
\end{equation}
The Kohn-Laplacian on $H_n$ is given by
\begin{equation}\label{s3-e221}
\Box^{(q)}_{b,H_n}=\ddbar_{b,H_n}\ddbar^{\,*,\psi_0}_{b,H_n}+\ddbar^{\,*,\psi_0}_{b,H_n}\ddbar_{b,H_n}:\Omega^{0,q}(H_n)\To\Omega^{0,q}(H_n)\,.
\end{equation}
From \eqref{s3-e16}, \eqref{s3-e21} and \eqref{s3-e22}, we can check that
\begin{equation} \label{s3-e23}
\begin{split}
&\Box^{(q)}_{b,H_n}\\
&=\sum^{n-1}_{j=1}\ol U^{\,*,\psi_0}_{j,H_n}\ol U_{j,H_n}
        +\sum^{n-1}_{j,\,t=1}d\ol z_j\wedge (d\ol z_t\wedge)^*\Bigr[\Bigr(\mu_{j,\,t}-\frac{i}{\sqrt{2}}\lambda_j\delta_{j,\,t}\beta\Bigr)+i\sqrt{2}\lambda_j\delta_{j,\,t}\,\frac{\pr}{\pr\theta}\Bigr]\\
        &=\sum^{n-1}_{j=1}\Bigr[\Bigr(-\frac{\pr}{\pr z_j}+\frac{i}{\sqrt{2}}\lambda_j\ol z_j\frac{\pr}{\pr\theta}+\sum^{n-1}_{t=1}\mu_{t,j}\ol z_t-\frac{1}{\sqrt{2}}i\lambda_j\ol z_j\beta\Bigr)\Bigr(\frac{\pr}{\pr\ol z_j}+\frac{i}{\sqrt{2}}\lambda_jz_j\frac{\pr}{\pr\theta}\Bigr)\Bigr]\\
        &\quad+\sum^{n-1}_{j,t=1}d\ol z_j\wedge (d\ol z_t\wedge)^*\Bigr[\Bigr(\mu_{j,\,t}-\frac{i}{\sqrt{2}}\lambda_j\delta_{j,\,t}\beta\Bigr)+i\sqrt{2}\lambda_j\delta_{j,\,t}\,\frac{\pr}{\pr\theta}\Bigr].
\end{split}
\end{equation}

\subsection{The estimates of the Szeg\"{o} kernel function $\pit^{(q)}_k$}

We need the following

\begin{prop} \label{s3-p3}
For each $k$, pick an element $\alpha_k\in F^*_k\Omega^{0,q}(D_{\log k})$ with
$\Box^{(q)}_{s,(k)}\alpha_k=0$
and $\norm{\alpha_k}_{kF^*_k\phi_0,D_{\log k}}\leqslant  1$. Identify $\alpha_k$ with a form on $H_n$ by extending it with zero and write $\alpha_k=\sum'_{\abs{J}=q}\alpha_{k,J}e_J(\frac{z}{\sqrt{k}},\frac{\theta}{k})$. Then there is a subsequence $\set{\alpha_{k_j}}$ of $\set{\alpha_k}$ such that for each strictly increasing multiindex $J$, $\abs{J}=q$, $\alpha_{k_j,\,J}$ converges uniformly with all its derivatives on any compact subset of $H_n$ to a smooth function $\alpha_J$. Furthermore, if we put $\alpha=\sum'_{\abs{J}=q}\alpha_Jd\ol z_J$, then
$\Box^{(q)}_{b,H_n}\alpha=0$.
\end{prop}

\begin{proof}
Fix a strictly increasing multiindex $J$, $\abs{J}=q$,
and $r>0$. From \eqref{s3-e12} and Remark~\ref{s3-r1}, we see that for
all $s>0$, there is a constant $C_{r,s}$, $C_{r,s}$ is independent of $k$, such that
$\norm{\alpha_{k,J}}_{s,D_r}\leqslant  C_{r,s}$
for all $k$. Rellich 's compactness theorem \cite[p.\,281]{Yo80} yields  a subsequence of $\set{\alpha_{k,J}}$, which converges in all Sobolev spaces $W^s(D_r)$ for $s>0$. From the Sobolev embedding theorem \cite[p.\,170]{Yo80}, we see that the sequence converges in all $C^l(D_r)$, $l\geqslant0$, $l\in\mathbb Z$, locally unformly. Choosing a diagonal sequence, with respect to a sequence of $D_r$ exhausting $H_n$, we get a subsequence $\set{\alpha_{k_j,J}}$ of $\set{\alpha_{k,J}}$ such that $\alpha_{k_j,J}$ converges uniformly with all derivatives on any compact subset of $H_n$ to a smooth function $\alpha_J$.

Let $J'$ be another strictly increasing multiindex, $\abs{J'}=q$. We can repeat the procedure above and get a subsequce $\set{\alpha_{k_{j_s},J'}}$ of $\set{\alpha_{k_j,J'}}$ such that $\alpha_{k_{j_s},J'}$ converges uniformly with all derivatives on any compact subset of $H_n$ to a smooth function $\alpha_{J'}$. Continuing in this way, we get the first statement of the proposition.

Now, we prove the second statement of the proposition. Let $P=(p_1,\ldots,p_q)$, $R=(r_1,\ldots,r_q)$ be multiindices, $\abs{P}=\abs{R}=q$. Define
\[\varepsilon^P_R=\left\{ \begin{array}{ll}
&0,\ \ \mbox{if $\set{p_1,\ldots,p_q}\neq\set{r_1,\ldots,r_q}$}, \\
&\mbox{the sign of permutation $\left(
\begin{array}[c]{c}
 P  \\
R
\end{array}\right)$
},\ \ \mbox{if $\set{p_1,\ldots,p_q}=\set{r_1,\ldots,r_q}$}.
\end{array}\right.\]
For $j, t=1,\ldots,n-1$, define
\[\sigma^{jtP}_{R}=\left\{ \begin{array}{ll}
&0,\ \ \mbox{if $d\ol z_j\wedge (d\ol z_t\wedge)^*(d\ol z^P)=0$}, \\
&\varepsilon^Q_R,\ \ \mbox{if $d\ol z_j\wedge (d\ol z_t\wedge)^*(d\ol z^P)=d\ol z^Q$, $\abs{Q}=q$}.
\end{array}\right.\]

We may assume that $\alpha_{k,J}$ converges uniformly with all derivatives on any compact subset of $H_n$ to a smooth function $\alpha_J$, for all strictly increasing $J$, $\abs{J}=q$. Since $\Box^{(q)}_{s,(k)}\alpha_k=0$, from the explicit formula of $\Box^{(q)}_{s,(k)}$
(see Prop. \ref{s3-p1}), it is not difficult to see that for all strictly increasing $J$, $\abs{J}=q$, we have
\begin{equation} \label{s3-e24}
\begin{split}
\sum^{n-1}_{j=1}\ol U^{\,*,\psi_0}_{j,H_n}\ol U_{j,H_n}\alpha_{k,J}=-\sideset{}{'}\sum_{\substack{\abs{P}=q,\\1\,\leqslant\, \, j\,,\,t\,\leqslant \, n-1}}
\sigma^{jtP}_{J}
\Bigr[\Bigr(\mu_{j,\,t}-\frac{i}{\sqrt{2}}\lambda_j\delta_{j,\,t}\beta\Bigr)+&\sqrt{2}i\lambda_j\delta_{j,\,t}\frac{\pr}{\pr\theta}\Bigr]\alpha_{k,P}\\
&+\epsilon_kP_{k,J}\alpha_k
\end{split}
\end{equation}
on $D_{\log k}$, where $\epsilon_k$ is a sequence tending to zero with $k\To\infty$ and $P_{k,J}$ is a second order differential operator and all the derivatives of the coefficients of $P_{k,J}$ are uniformly bounded in $k$ on $D_{\log k}$. By letting $k\To\infty$ in \eqref{s3-e24} we get
\begin{equation} \label{s3-e25}
\sum^{n-1}_{j=1}\ol U^{\,*,\psi_0}_{j,H_n}\ol U_{j,H_n}\alpha_{J}=-\sideset{}{'}\sum_{\substack{\abs{P}=q,\\1\,\leqslant\, \, j\,,\,t\,\leqslant \, n-1}}
\sigma^{jtP}_{J}
\Bigr[\Bigr(\mu_{j,\,t}-\frac{i}{\sqrt{2}}\lambda_j\delta_{j,\,t}\beta\Bigr)+\sqrt{2}i\lambda_j\delta_{j,\,t}\frac{\pr}{\pr\theta}\Bigr]\alpha_{P}
\end{equation}
on $H_n$, for all strictly increasing $J$, $\abs{J}=q$. From this and the explicit formula of $\Box^{(q)}_{b,H_n}$ (see \eqref{s3-e23}),
we conclude that $\Box^{(q)}_{b,H_n}\alpha=0$. The proposition follows.
\end{proof}

Now, we can prove the main result of this section.
In analogy to \eqref{s2-e2} we define the extremal functions $S^{(q)}_{J,H_n}$ on the Heisenberg group along the direction $d\ol z_J$ is defined by
\begin{equation} \label{s2-e271}
S^{(q)}_{J,H_n}(0)=\sup\big\{\abs{\alpha_J(0)}^2;\Box^{(q)}_{b,H_n}\alpha=0, \norm{\alpha}_{\psi_0}=1\big\}\,.
\end{equation}
where $\alpha=\sum'_{\abs{J}=q}\alpha_Jd\ol z_J$.
\begin{thm} \label{s3-t2}
We have
\[
\limsup_{k\To\infty}k^{-n}\pit^{(q)}_k(0)\leqslant \sideset{}{^\prime}\sum_{\abs{J}=q}S^{(q)}_{J,H_n}(0)\,.
\]
\end{thm}

\begin{proof}
Fix a strictly increasing $J$, $\abs{J}=q$. We claim that
\begin{equation} \label{s3-e28}
\limsup_{k\To\infty}k^{-n}S^{(q)}_{k,J}(0)\leqslant  S^{(q)}_{J,H_n}(0).
\end{equation}
The definition \eqref{s2-e2} of the extremal function yields a sequence $\alpha_{k_j}\in H^q_b(X, L^{k_j})$,  $k_1<k_2<\ldots$\,,
such that $\norm{\alpha_{k_j}}=1$ and
\begin{equation}\label{s3-e281}
 \lim_{j\To\infty}k_j^{-n}\abs{\alpha_{k_j,J}(0)}^2=\limsup_{k\To\infty}k^{-n}S^{(q)}_{k,J}(0)\,,
\end{equation}
where $\alpha_{k_j,J}$ is the component
of $\alpha_{k_j}$ in the direction of $e_J$. On $D_{\log k_j}$, put
\[\beta_{k_j}=k_j^{-\frac{n}{2}}F^*_{k_j}(e^{-k_jR}\alpha_{k_j})\in F^*_{k_j}\Omega^{0,q}(D_{\log k_j})\,.\]
It is easy to see that
\[
\text{$\norm{\beta_{k_j}}_{k_jF^*_{k_j}\phi_0,D_{\log k_j}}\leqslant1$ and $\Box^{(q)}_{s,(k_j)}\beta_{k_j}=0$ on $D_{\log k_j}$}\,.
\] Proposition~\ref{s3-p3} yields a subsequence $\set{\beta_{k_{j_s}}}$ of $\set{\beta_{k_j}}$ such that for each $J$, $\beta_{k_{j_s},J}$ converges uniformly with all derivatives on any compact subset of $H_n$ to a smooth function $\beta_J$. Set
$\beta=\sum'_{\abs{P}=q}\beta_Pd\ol z^P$. Then we have
$\Box^{(q)}_{b,H_n}\beta=0$ and, by \eqref{s3-e17-1}, $\norm{\beta}_{\psi_0}\leqslant  1$. Thus,
\begin{equation} \label{s3-e29}
|\beta_J(0)|^2\leqslant \frac{|\beta_J(0)|^2}{\norm{\beta}_{\psi_0}^2}\leqslant  S^{(q)}_{J,H_n}(0).
\end{equation}
Note that
\begin{equation} \label{s3-e291}
\lim_{j\To\infty}k_j^{-n}\abs{\alpha_{k_j,J}(0)}^2=\lim_{s\To\infty}|\beta_{k_{j_s},J}(0)|^2=|\beta_J(0)|^2\,.
\end{equation}
The estimate \eqref{s3-e28} follows from \eqref{s3-e281}, \eqref{s3-e29} and \eqref{s3-e291}. Finally,  Lemma~\ref{s2-l1} and \eqref{s3-e28} imply the conclusion of the theorem.
\end{proof}

\section{The Szeg\"{o} kernel function on the Heisenberg group $H_n$}

In this section, we will use the same notations as in section 2.3 and we work with the assumption that
condition $Y(q)$ holds at each point of $X$.
The main goal of this section is to compute $\sum'_{\abs{J}=q}S^{(q)}_{J,H_n}(0)$.

\subsection{The partial Fourier transform}

Let $u(z, \theta)\in\Omega^{0,q}(H_n)$  with $\norm{u}_{\psi_0}=1$
and $\Box^{(q)}_{b,H_n}u=0$. Put
$v(z, \theta)=u(z, \theta)e^{-\frac{\beta}{2}\theta}$
and set $\Phi_0=\sum^{n-1}_{j,t=1}\mu_{j,\,t}\ol z_jz_t$.
We have \[\int_{H_n}\!\abs{v(z,\theta)}^2e^{-\Phi_0(z)}dv(z)dv(\theta)=1\,.\]
Let us denote by $L^2_{(0,q)}(H_n, \Phi_0)$ the competion of $\Omega_c^{(0,q)}(H_n)$
with respect to the norm $\|\cdot\|_{\Phi_0}$, where
\[
\|u\|^2_{\Phi_0}=\int_{H_n}\abs{u}^2e^{-\Phi_0}dv(z)dv(\theta)\,,\quad u\in\Omega_c^{(0,q)}(H_n)\,.
\]
Choose $\chi(\theta)\in C^\infty_0(\Real)$ so that $\chi(\theta)=1$ when $\abs{\theta}<1$ and $\chi(\theta)=0$ when $\abs{\theta}>2$ and set $\chi_j(\theta)=\chi(\theta/j)$, $j\in\mathbb{N}$. Let
\begin{equation} \label{s4-e11}
\hat v_j(z, \eta)=\int_{\Real}\!v(z,\theta)\chi_j(\theta)e^{-i\theta\eta}dv(\theta)\in\Omega^{0,q}(H_n),\ j=1,2,\ldots.
\end{equation}
From Parseval's formula, we have
\begin{align*}
&\int_{H_n}\!\abs{\hat v_j(z,\eta)-\hat v_t(z,\eta)}^2e^{-\Phi_0(z)}dv(\eta)dv(z)\\
&=4\pi\int_{H_n}\!\abs{v(z,\theta)}^2\abs{\chi_j(\theta)-\chi_t(\theta)}^2e^{-\Phi_0(z)}dv(\theta)dv(z)\To0,\  j,t\To\infty.
\end{align*}
Thus, there is $\hat v(z, \eta)\in L^2_{(0,q)}(H_n, \Phi_0)$ such tht $\hat v_j(z,\eta)\To\hat v(z, \eta)$ in $L^2_{(0,q)}(H_n, \Phi_0)$. We call $\hat v(z, \eta)$ the Fourier transform of $v(z, \theta)$ with respect to $\theta$. Formally,
\begin{equation} \label{s4-e12}
\hat v(z, \eta)=\int_{\Real}\! e^{-i\theta\eta}v(z,\theta)dv(\theta).
\end{equation}
Moreover, we have
\begin{equation} \label{s4-e12-1}
\begin{split}
&\int_{H_n}\!\abs{\hat v(z, \eta)}^2e^{-\Phi_0(z)}dv(z)dv(\eta)=\lim_{j\To\infty}\int_{H_n}\!\abs{\hat v_j(z, \eta)}^2e^{-\Phi_0(z)}dv(z)dv(\eta)\\
&=4\pi\lim_{j\To\infty}\int_{H_n}\!\abs{u(z, \theta)e^{-\frac{\beta}{2}\theta}\chi_j(\theta)}^2e^{-\Phi_0(z)}dv(z)dv(\theta)\\
&=4\pi\int_{H_n}\!\abs{u(z, \theta)}^2e^{-\psi_0(z,\theta)}dv(z)dv(\theta)=4\pi<\infty.
\end{split}
\end{equation}
From Fubini's theorem,
$\int_{\Complex^{n-1}}\!\abs{\hat v(z,\eta)}^2e^{-\Phi_0(z)}dv(z)<\infty$
for almost all $\eta\in\Real$. More precisely, there is a negligeable set $A_0\subset\Real$
such that $\int_{\Complex^{n-1}}\!\abs{\hat v(z,\eta)}^2e^{-\Phi_0(z)}dv(z)<\infty$,
for every $ \eta\notin A_0$.

Let $s\in L^2_{(0,q)}(H_n, \Phi_0)$.
Assume that $\int\!\abs{s(z, \eta)}^2dv(\eta)<\infty$ and $\int\!\abs{s(z, \eta)}dv(\eta)<\infty$ for all $z\in\Complex^{n-1}$. Then, from Parseval's formula, we can check that
\begin{equation} \label{s4-e12-2}
\begin{split}
&\iint\!\langle\hat v(z, \eta),s(z, \eta)\rangle e^{-\Phi_0(z)}dv(\eta)dv(z)\\
&=\iint\!\langle u(z, \theta)e^{-\frac{\beta}{2}\theta}, \int\! e^{i\theta\eta}s(z, \eta)dv(\eta)\rangle e^{-\Phi_0(z)}dv(\theta)dv(z).
\end{split}
\end{equation}

We pause and introduce some notations. For fixed $\eta\in\Real$, put
\begin{equation} \label{s4-e13}
\Phi_\eta=-\sqrt{2}\eta\sum^{n-1}_{j=1}\lambda_j\abs{z_j}^2+\sum^{n-1}_{j,t=1}\mu_{j,\,t}\ol z_jz_t\in C^\infty(\Complex^{n-1};\, \Real).
\end{equation}
We take the Hermitian metric $\langle\,\cdot\,,\cdot\,\rangle$ on the bundle $\Lambda^{0,q}T^*\Complex^{n-1}$ of $(0, q)$ forms of $\Complex^{n-1}$  so that $\{d\ol z_J; \text{$\abs{J}=q$, $J$ strictly increasing}\}$ is an orthonormal basis. We also let $\Omega^{0,q}(\Complex^{n-1})$ denote the space of smooth sections of $\Lambda^{0,q}T^*\Complex^{n-1}$ over $\Complex^{n-1}$. Let $\Omega^{0,q}_c(\Complex^{n-1})$ be the subspace of $\Omega^{0,q}(\Complex^{n-1})$ whose elements have compact support in $\Complex^{n-1}$ and let $(\ |\ )_{\Phi_\eta}$ be the inner product on $\Omega^{0,q}_c(\Complex^{n-1})$ defined by
\[(f\ |\ g)_{\Phi_\eta}=\int_{\Complex^{n-1}}\!\langle f, g\rangle e^{-\Phi_\eta(z)}dv(z)\,,\quad f, g\in\Omega^{0,q}_c(\Complex^{n-1})\,.\]
Let
\begin{equation} \label{s4-e14}
\Box^{(q)}_{\Phi_\eta}=\ol{\pr}^{\,*,\Phi_\eta}\ddbar+\ddbar\,\ol{\pr}^{\,*,\Phi_\eta}:\Omega^{0,q}(\Complex^{n-1})\To\Omega^{0,q}(\Complex^{n-1})
\end{equation}
be the complex Laplacian with respect to $(\ |\ )_{\Phi_\eta}$, where $\ddbar^{\,*,\Phi_\eta}$ is the formal adjoint of $\ddbar$ with respect to $(\ |\ )_{\Phi_\eta}$. We can check that
\begin{equation} \label{s4-e15}
\begin{split}
\Box^{(q)}_{\Phi_\eta}&=\sum^{n-1}_{t=1}\Bigr(-\frac{\pr}{\pr z_t}-\sqrt{2}\lambda_t\ol z_t\eta+\sum^{n-1}_{j=1}\mu_{j,\,t}\ol z_j\Bigr)\frac{\pr}{\pr\ol z_t}\\
&+\sum^{n-1}_{j,t=1}d\ol z_j\wedge (d\ol z_t\wedge)^*\Bigr(\mu_{j,\,t}-\sqrt{2}\lambda_j\eta\delta_{j,\,t}\Bigr).
\end{split}
\end{equation}
Now, we return to our situation. We identify $\Lambda^{0,q}T^*\Complex^{n-1}$ with $\Lambda^{0,q}T^*H_n$. Set
\begin{equation}\label{s4-e161}
\alpha(z, \eta)=\hat v(z, \eta)\exp\Bigl[\Bigl(\frac{i\beta}{2\sqrt{2}}-\frac{1}{\sqrt{2}}\eta\Bigl) \sum^{n-1}_{j=1}\lambda_j\abs{z_j}^2\Bigr]\,.
\end{equation}
We remind that $\hat v(z, \eta)$ is given by \eqref{s4-e12}. 

\begin{thm} \label{s4-t1}
For almost all $\eta\in\Real$, we have
$\int_{\Complex^{n-1}}\!\abs{\alpha(z, \eta)}^2e^{-\Phi_\eta(z)}dv(z)<\infty$
and
\begin{equation} \label{s4-e17}
\Box^{(q)}_{\Phi_\eta}\alpha(z, \eta)=0
\end{equation}
in the sense of distributions. 
Thus $\alpha(z,\eta)\in\Omega^{0,q}(\Complex^{n-1})$
for almost all $\eta\in\Real$.
\end{thm}

\begin{proof}
Let $A_0\subset\Real$ be as in the discussion after \eqref{s4-e12-1}.
Thus, for all $\eta\notin A_0$,
\[\int_{\Complex^{n-1}}\!\abs{\hat v(z,\eta)}^2e^{-\Phi_0(z)}dv(z)=\int_{\Complex^{n-1}}\!\abs{\alpha(z, \eta)}^2e^{-\Phi_\eta(z)}dv(z)<\infty\,.\]
We only need to prove the second statement of the theorem. Let
$f\in\Omega^{0,q}_c(\Complex^{n-1})$.
Put
$h(\eta)=\int_{\Complex^{n-1}}\!\langle\alpha(z,\eta), \Box^{(q)}_{\Phi_\eta}f(z)\rangle e^{-\Phi_\eta(z)}dv(z)$
if $\eta\notin A_0$, $h(\eta)=0$ if $\eta\in A_0$. We can check that
\begin{equation} \label{s4-e18-1}
\abs{h(\eta)}^2\leqslant \int_{\Complex^{n-1}}\!\abs{\alpha(z,\eta)}^2e^{-\Phi_\eta(z)}dv(z)\int_{\Complex^{n-1}}\!\big\vert\Box^{(q)}_{\Phi_\eta}f\big\vert^2e^{-\Phi_\eta(z)}dv(z).
\end{equation}
For $R>0$, put $\varphi_R(\eta)=\mathds{1}_{[-R,R]}(\eta)\ol h(\eta)$,
where $\mathds{1}_{[-R,R]}(\eta)=1$ if $-R\leqslant \eta\leqslant  R$, $\mathds{1}_{[-R,R]}(\eta)=0$ if $\eta<-R$ or $\eta>R$. From \eqref{s4-e18-1}, we have
\begin{equation} \label{s4-e18-2}
\begin{split}
&\int\!\abs{\varphi_R(\eta)}^2dv(\eta)=\int^R_{-R}\!\abs{h(\eta)}^2dv(\eta)\\
&\leqslant  C\iint\!\abs{\alpha(z,\eta)}^2e^{-\Phi_\eta(z)}dv(\eta)dv(z)=C\iint\!\abs{\hat v(z, \eta)}^2e^{-\Phi_0(z)}dv(\eta)dv(z)<\infty\,,
\end{split}
\end{equation}
where $C>0$. Thus, $\varphi_R(\eta)\in L^2(\Real)\cap L^1(\Real)$. Set
$\lambda\abs{z}^2:=\sum^{n-1}_{j=1}\lambda_j\abs{z_j}^2$.
We have
\begin{equation} \label{s4-e19}
\begin{split}
&\int_{\Real}\!h(\eta)\varphi_R(\eta)dv(\eta)=\int^R_{-R}\!\abs{h(\eta)}^2dv(\eta)\\
&=\iint\!\langle\alpha(z,\eta), \Box^{(q)}_{\Phi_\eta}f(z)\rangle e^{-\Phi_\eta(z)}\varphi_R(\eta)dv(\eta)dv(z)\\
&=\iint\!\langle\hat v(z,\eta),\ e^{-\big(\frac{i\beta}{2\sqrt{2}}-\frac{1}{\sqrt{2}}\eta\big)\lambda\abs{z}^2}\Box^{(q)}_{\Phi_\eta}f(z)\ol\varphi_R(\eta)\rangle e^{-\Phi_0(z)}dv(\eta)dv(z)\\
&\stackrel{\eqref{s4-e12-2}}{=}\iint\!\langle u(z,\theta),\int\!
e^{i\theta\eta}
e^{-\big(\frac{i\beta}{2\sqrt{2}}-\frac{1}{\sqrt{2}}\eta\big)\lambda\abs{z}^2
+\frac{\beta}{2}\theta}\,\Box^{(q)}_{\Phi_\eta}(f\ol\varphi_R)dv(\eta)\rangle e^{-\psi_0(z, \theta)}dv(\theta)dv(z).
\end{split}
\end{equation}
%
From Lemma~\ref{s4-l1} below, we know that
\begin{equation} \label{s4-e22}
\begin{split}
&\int\! e^{i\theta\eta}e^{-\big(\frac{i\beta}{2\sqrt{2}}-\frac{1}{\sqrt{2}}\eta\big)\lambda\abs{z}^2
+\frac{\beta}{2}\theta}\,\Box^{(q)}_{\Phi_\eta}(f(z)\ol\varphi_R(\eta))dv(\eta)\\
&=\Box^{(q)}_{b,H_n}\Bigr(\int\! e^{i\theta\eta}e^{-\big(\frac{i\beta}{2\sqrt{2}}-\frac{1}{\sqrt{2}}\eta\big)\lambda\abs{z}^2
+\frac{\beta}{2}\theta}\,f(z)\ol\varphi_R(\eta)dv(\eta)\Bigr).
\end{split}
\end{equation}
Put
\begin{equation} \label{s4-e23}
S(z,\theta)=\int\! e^{i\theta\eta}e^{-\big(\frac{i\beta}{2\sqrt{2}}-\frac{1}{\sqrt{2}}\eta\big)\lambda\abs{z}^2
+\frac{\beta}{2}\theta}\,f(z)\ol\varphi_R(\eta)dv(\eta).
\end{equation}
From \eqref{s4-e22} and \eqref{s4-e23}, \eqref{s4-e19} becomes
\begin{equation} \label{s4-e24}
\int^R_{-R}\!\abs{h(\eta)}^2dv(\eta)
=\iint\!\langle u(z,\theta),\Box^{(q)}_{b,H_n}S(z,\theta)\rangle e^{-\psi_0(z,\theta)}dv(\theta)dv(z).
\end{equation}
Choose $\chi(\theta)\in C^\infty_0(\Real)$ so that $\chi(\theta)=1$ when $\abs{\theta}<1$ and $\chi(\theta)=0$ when $\abs{\theta}>2$. Then,
\begin{equation} \label{s4-e25}
\begin{split}
\int^R_{-R}\!\abs{h(\eta)}^2dv(\eta)
&=\lim_{j\To\infty}\iint\!\langle u(z,\theta),\chi(\tfrac{\theta}{j})\Box^{(q)}_{b,H_n}S(z,\theta)\rangle e^{-\psi_0(z,\theta)}dv(\theta)dv(z)\\
&=\lim_{j\To\infty}
\Bigr(\iint\!\langle\Box^{(q)}_{b,H_n}u(z,\theta),\chi(\tfrac{\theta}{j})S(z,\theta)\rangle e^{-\psi_0(z,\theta)}dv(\theta)dv(z)\\
&+
\iint\!\langle u(z,\theta),[\chi(\tfrac{\theta}{j})\ ,\Box^{(q)}_{b,H_n}]S(z,\theta)\rangle e^{-\psi_0(z,\theta)}dv(\theta)dv(z)\Bigr)\\
&=\lim_{j\To\infty}\iint\!\langle u(z,\theta), [\chi(\tfrac{\theta}{j})\ ,\Box^{(q)}_{b,H_n}]S(z,\theta)\rangle e^{-\psi_0(z,\theta)}dv(\theta)dv(z).
\end{split}
\end{equation}
We can check that $[\chi(\frac{\theta}{j})\ ,\Box^{(q)}_{b,H_n}]$ is a first order partial differential operator and all the coefficients of $[\chi(\frac{\theta}{j})\ ,\Box^{(q)}_{b,H_n}]$ converge to $0$ as $j\To\infty$ uniformly in $\theta$ and locally uniformly in $z$. Moreover, from Parseval's formula, \eqref{s4-e18-2} and \eqref{s4-e23}, we can check that
\begin{align*}
&\sum_{\abs{\alpha}\leqslant 1}\int\!\abs{\pr^\alpha_{x,\theta}S}^2e^{-\psi_0}dv(\theta)dv(z)\\
&\leqslant  C\sum_{\abs{\alpha}\leqslant  1}\iint\!(1+\abs{z}+\abs{\eta}+\abs{z}\abs{\eta})^2\abs{\pr^\alpha_xf}^2\abs{\varphi_R(\eta)}^2e^{-\Phi_\eta(z)}dv(z)dv(\eta)\\
&\leqslant  \Td C\int\!\abs{\varphi_R(\eta)}^2dv(\eta)<\infty,
\end{align*}
whit constants $C>0$, $\Td C>0$. Thus,
\[\lim_{j\To\infty}\iint\!\langle u(z,\theta), [\chi(\tfrac{\theta}{j})\ ,\Box^{(q)}_{b,H_n}]S(z,\theta)\rangle e^{-\psi_0(z,\theta)}dv(\theta)dv(z)=0.\]
From this and \eqref{s4-e25}, we conclude that
$\int^R_{-R}\!\abs{h(\eta)}^2dv(\eta)=0$.
Letting $R\To\infty$, we get $h(\eta)=0$ almost everywhere. We have proved that for a given $f(z)\in\Omega^{0,q}_c(\Complex^{n-1})$, $\int_{\Complex^{n-1}}\!\langle\alpha(z,\eta),\Box^{(q)}_{\Phi_\eta}f(z)\rangle e^{-\Phi_\eta(z)}dv(z)=0$ almost everywhere.

Let us consider the Sobolev space $W^2(\Complex^{n-1})$ of distributions in $\Complex^{n-1}$ whose derivatives of order $\leqslant2$ are in $L^2$. The space of forms of type $(0,q)$ with coefficients in this space is accordingly denoted $W^2_{(0,q)}(\Complex^{n-1})$.
Since $W^2_{(0,q)}(\Complex^{n-1})$
is separable and $\Omega^{0,q}_c(\Complex^{n-1})$ is dense in $W^2_{(0,q)}(\Complex^{n-1})$,
we can find $f_j\in\Omega^{0,q}_c(\Complex^{n-1})$, $j=1,2,\ldots$\,, such that
$\set{f_1,f_2,\ldots}$
is a dense subset of $W^2_{(0,q)}(\Complex^{n-1})$. Moreover, we can take $\set{f_1,f_2,\ldots}$ so
that for all $g\in\Omega^{0,q}_c(\Complex^{n-1})$ with
${\rm supp\,}g\subset B_r:=\set{z\in\Complex^{n-1};\, \abs{z}<r}$, $r>0$,
we can find $f_{j_1}, f_{j_2},\ldots$\,, ${\rm supp\,}f_{j_t}\subset B_r$, $t=1,2,\ldots$\,, such that
$f_{j_t}\To g$ for $t\To\infty$ in $W^2_{(0,q)}(\Complex^{n-1})$.

Now, for each $j$, we can repeat the method above and find a measurable set $A_j\supset A_0$, $\abs{A_j}=0$ ($A_0$ is as in the beginning of the proof), such that
$(\alpha(z,\eta)\ |\ \Box^{(q)}_{\Phi_\eta}f_j(z))_{\Phi_\eta}=0$ for all $\eta\notin A_j$.
Put $A=\bigcup_jA_j$. Then, $\abs{A}=0$ and for all $\eta\notin A$,
$(\alpha(z,\eta)\ |\ \Box^{(q)}_{\Phi_\eta}f_j(z))_{\Phi_\eta}=0$ for all $j$.
Let $g\in\Omega^{0,q}_c(\Complex^{n-1})$ with
${\rm supp\,}g\subset B_r$. From the discussion above,
we can find $f_{j_1}, f_{j_2},\ldots$, ${\rm supp\,}f_{j_t}\subset B_r$, $t=1,2,\ldots$, such that
$f_{j_t}\To g$ in $W^2_{(0,q)}(\Complex^{n-1})$, $t\To\infty$.
Then, for $\eta\notin A$,
\begin{align*}
(\alpha(z,\eta)\ |\ \Box^{(q)}_{\Phi_\eta}g)_{\Phi_\eta}&=(\alpha(z,\eta)\ |\ \Box^{(q)}_{\Phi_\eta}(g-f_{j_t})))_{\Phi_\eta}
+(\alpha(z,\eta)\ |\ \Box^{(q)}_{\Phi_\eta}f_{j_t})_{\Phi_\eta}\\
&=(\alpha(z,\eta)\ |\ \Box^{(q)}_{\Phi_\eta}(g-f_{j_t})))_{\Phi_\eta}.
\end{align*}
Now,
\begin{equation}
\begin{split}
\abs{(\alpha(z,\eta)\ |\ \Box^{(q)}_{\Phi_\eta}(g-f_{j_t})))_{\Phi_\eta}}&=\abs{\int_{B_r}\!\langle\alpha(z,\eta),\Box^{(q)}_{\Phi_\eta}(g-f_{j_t})\rangle e^{-\Phi_\eta(z)}dv(z)}\\
&\leqslant  C\sum_{\abs{\alpha}
\leqslant  2}\int\!\abs{\pr^\alpha_x(g-f_{j_t})}^2dv(z)\To 0, \quad t\To\infty.
\end{split}
\end{equation}
Thus, for $\eta\notin A$,
$(\alpha(z,\eta)\ |\ \Box^{(q)}_{\Phi_\eta}g)_{\Phi_\eta}=0$
for all $g\in\Omega^{0,q}_c(\Complex^{n-1})$. The theorem follows.
\end{proof}

\begin{lem} \label{s4-l1}
Let $f\in\Omega^{0,q}_c(\Complex^{n-1})$. Let $\varphi(\eta)\in L^2(\Real)$ with compact support. Then, we have
\begin{equation*}
\begin{split}
&\int\! e^{i\theta\eta}e^{-\big(\frac{i\beta}{2\sqrt{2}}-\frac{1}{\sqrt{2}}\eta\big)\lambda\abs{z}^2
+\frac{\beta}{2}\theta}\,
\Box^{(q)}_{\Phi_\eta}f(z)\varphi(\eta)dv(\eta)\\
&=\Box^{(q)}_{b,H_n}\Bigr(\int\! e^{i\theta\eta}e^{-\big(\frac{i\beta}{2\sqrt{2}}-\frac{1}{\sqrt{2}}\eta\big)\lambda\abs{z}^2
+\frac{\beta}{2}\theta}\,f(z)\varphi(\eta)dv(\eta)\Bigr),
\end{split}
\end{equation*}
where $\lambda\abs{z}^2=\sum^{n-1}_{j=1}\lambda_j\abs{z_j}^2$.
\end{lem}

\begin{proof}
For any $g\in\Omega^{0,q}_c(\Complex^{n-1})$, we can check that
\begin{equation} \label{s4-e26}
\begin{split}
&\ol U_{t,H_n}\Bigr(\int\! e^{i\theta\eta}e^{-\big(\frac{i\beta}{2\sqrt{2}}-\frac{1}{\sqrt{2}}\eta\big)\lambda\abs{z}^2
+\frac{\beta}{2}\theta}\,g(z)\varphi(\eta)dv(\eta)\Bigr)\\
&=\Bigr(\frac{\pr}{\pr\ol z_t}+\frac{1}{\sqrt{2}}i\lambda_tz_t\frac{\pr}{\pr\theta}\Bigr)
\Bigr(\int\! e^{i\theta\eta}e^{-\big(\frac{i\beta}{2\sqrt{2}}-\frac{1}{\sqrt{2}}\eta\big)\lambda\abs{z}^2
+\frac{\beta}{2}\theta}\,g(z)\varphi(\eta)dv(\eta)\Bigr)\\
&=\int\! e^{i\theta\eta}e^{-\big(\frac{i\beta}{2\sqrt{2}}-\frac{1}{\sqrt{2}}\eta\big)\lambda\abs{z}^2
+\frac{\beta}{2}\theta}\,\frac{\pr g}{\pr \ol z_t}\varphi(\eta)dv(\eta),
\end{split}
\end{equation}
where $t=1,\ldots,n-1$,
\begin{equation} \label{s4-e27}
\begin{split}
\ol U^{\,*,\psi_0}_{t,H_n}&\Bigr(\int\! e^{i\theta\eta}e^{-\big(\frac{i\beta}{2\sqrt{2}}-\frac{1}{\sqrt{2}}\eta\big)\lambda\abs{z}^2
+\frac{\beta}{2}\theta}\,g(z)\varphi(\eta)dv(\eta)\Bigr)\\
=&\Bigr(-\frac{\pr}{\pr z_t}+\frac{1}{\sqrt{2}}i\lambda_t\ol z_t\frac{\pr}{\pr\theta}+\sum^{n-1}_{j=1}\mu_{j,\,t}\ol z_j-\frac{1}{\sqrt{2}}i\lambda_t\ol z_t\beta\Bigr)\\
&\Bigr(\int\! e^{i\theta\eta}e^{-\big(\frac{i\beta}{2\sqrt{2}}-\frac{1}{\sqrt{2}}\eta\big)\lambda\abs{z}^2
+\frac{\beta}{2}\theta}\,g(z)\varphi(\eta)dv(\eta)\Bigr)\\
=&\int\! e^{i\theta\eta}e^{-\big(\frac{i\beta}{2\sqrt{2}}-\frac{1}{\sqrt{2}}\eta\big)\lambda\abs{z}^2
+\frac{\beta}{2}\theta}\,(-\frac{\pr g}{\pr z_t}+\sum^{n-1}_{j=1}\mu_{j,\,t}\ol z_jg-\sqrt{2}\lambda_t\ol z_t\eta g)\varphi(\eta)dv(\eta),
\end{split}
\end{equation}
where $t=1,\ldots,n-1$, and
\begin{equation} \label{s4-e28}
\begin{split}
&\Bigr(\mu_{j,\,t}-\frac{1}{\sqrt{2}}i\lambda_j\delta_{j,\,t}\beta+\sqrt{2}i\lambda_j\delta_{j,\,t}\frac{\pr}{\pr\theta}\Bigr)\,\int\! e^{i\theta\eta}e^{-\big(\frac{i\beta}{2\sqrt{2}}-\frac{1}{\sqrt{2}}\eta\big)\lambda\abs{z}^2
+\frac{\beta}{2}\theta}\,g\varphi(\eta)dv(\eta)\\
&=\int\! e^{i\theta\eta}e^{-\big(\frac{i\beta}{2\sqrt{2}}-\frac{1}{\sqrt{2}}\eta\big)\lambda\abs{z}^2
+\frac{\beta}{2}\theta}\,(\mu_{j,\,t}g-\sqrt{2}\eta\lambda_j\delta_{j,\,t}g)\varphi(\eta)dv(\eta),
\end{split}
\end{equation}
where $j,t=1,\ldots,n-1$. From \eqref{s4-e26}, \eqref{s4-e27}, \eqref{s4-e28} and the explicit formulas for
$\Box^{(q)}_{b,H_n}$ and $\Box^{(q)}_{\Phi_\eta}$ (see \eqref{s3-e23} and \eqref{s4-e15}),
the lemma follows.
\end{proof}

\subsection{Estimates for the extremal function on the Heisenberg group}

We will use the same notations as before. For $\eta\in\Real$,
let us denote by $L^2_{(0,q)}(\Complex^{n-1}, \Phi_\eta)$ the competion of $\Omega_c^{(0,q)}(\Complex^{n-1})$
with respect to the norm $\|\cdot\|_{\Phi_\eta}$, where
\[
\|u\|^2_{\Phi_\eta}=\int_{\Complex^{n-1}}\!\abs{u}^2e^{-\Phi_\eta(z)}dv(z)\,,\quad u\in\Omega_c^{(0,q)}(\Complex^{n-1})\,.
\]

Let $B^{(q)}_{\Phi_\eta}:L^2_{(0,q)}(\Complex^{n-1}, \Phi_\eta)\To\Ker\Box^{(q)}_{\Phi_\eta}$
be the Bergman projection, i.e.~the orthogonal projection onto $\Ker\Box^{(q)}_{\Phi_\eta}$ with respect to $(\ |\ )_{\Phi_\eta}$. Let
$(B^{(q)}_{\Phi_\eta})^*$ be the adjoint of $B^{(q)}_{\Phi_\eta}$ with respect to $(\ |\ )_{\Phi_\eta}$.
We have
$B^{(q)}_{\Phi_\eta}=(B^{(q)}_{\Phi_\eta})^*=(B^{(q)}_{\Phi_\eta})^2$.
Let
\begin{equation*}
\begin{split}
B^{(q)}_{\Phi_\eta}(z,w)\in C^\infty(\Complex^{n-1}\times\Complex^{n-1};\, \mathscr L(\Lambda^{0,q}T^*_w\Complex^{n-1},\Lambda^{0,q}T^*_z\Complex^{n-1}))\\
(B^{(q)}_{\Phi_\eta}u)(z)=\int_{\Complex^{n-1}}\!B^{(q)}_{\Phi_\eta}(z,w)u(w)e^{-\Phi_\eta(w)}dv(w)\,,\quad
u\in L^2_{(0,q)}(\Complex^{n-1}, \Phi_\eta)
\end{split}
\end{equation*}
be the distribution kernel of $B^{(q)}_{\Phi_\eta}$ with respect to $(\ |\ )_{\Phi_\eta}$.
We take the Hermitian metric $\langle\,\cdot\,,\cdot\,\rangle$ on $T^{1, 0}_z\Complex^{n-1}$, $z\in\Complex^{n-1}$, so that
$\frac{\pr}{\pr z_j}$, $j=1,\ldots,n-1$, is an orthonormal basis.
Let
\begin{equation} \label{s4-e29-1}
M_{\Phi_\eta}:T^{1, 0}_z\Complex^{n-1}\To T^{1, 0}_z\Complex^{n-1}\,,\quad z\in\Complex^{n-1}
\end{equation}
be the linear map defined by
$\langle M_{\Phi_\eta} U, V\rangle=\langle\pr\ddbar\Phi_\eta, U\wedge\ol V\rangle$,
$U, V\in T^{1, 0}_z\Complex^{n-1}$. Put
\begin{equation} \label{s4-e30}
\begin{split}
\Real_q:=&\,\{\eta\in\Real;\, \mbox{$M_{\Phi_\eta}$ has exactly $q$ negative eigenvalues}\\
&\quad\mbox{and $n-1-q$ positive eigenvalues}\}.
\end{split}
\end{equation}
The following result is essentially well-known (see Wu-Zhang \cite{WZ98}, Berman~\cite{Be04} and Ma-Marinescu~\cite[\S 8.2]{MM07}).
\begin{thm} \label{s4-t2}
If $\eta\notin \Real_q$, then $B^{(q)}_{\Phi_\eta}(z,z)=0$,
for all $z\in\Complex^{n-1}$.
If $\eta\in \Real_q$, let $(Z_j(\eta))_{j=1}^{n-1}$
be an orthonormal frame of $T^{1, 0}_z\Complex^{n-1}$, for which  $M_{\Phi_\eta}$ is diagonal. We assume that
$M_{\Phi_\eta}Z_j(\eta)=\nu_j(\eta)Z_j(\eta)$ for $j=1,\ldots,n-1$,
with $\nu_j(\eta)<0$ for $j=1,\ldots,q$ and $\nu_j(\eta)>0$ for $j=q+1,\ldots,n-1$.
Let $(T_j(\eta))_{j=1}^{n-1}$,
denote the basis of $T^{*0,1}_z\Complex^{n-1}$, which is dual to $(\ol Z_j(\eta))_{j=1}^{n-1}$. Then,
\begin{equation} \label{s4-e32}
B^{(q)}_{\Phi_\eta}(z,z)=e^{\Phi_\eta(z)}(2\pi)^{-n+1}\abs{\nu_1(\eta)}\cdots\abs{\nu_{n-1}(\eta)}\prod^q_{j=1}T_j(\eta)\wedge (T_j(\eta)\wedge)^*.
\end{equation}
In particular,
\begin{equation} \label{s4-e33}
\begin{split}
{\rm Tr\,}B^{(q)}_{\Phi_\eta}(z,z):&=\sideset{}{'}\sum_{\abs{J}=q}\langle B^{(q)}_{\Phi_\eta}(z,z)d\ol z_J,d\ol z_J\rangle\\
&=e^{\Phi_\eta(z)}(2\pi)^{-n+1}\abs{\nu_1(\eta)}\cdots\abs{\nu_{n-1}(\eta)}\mathds{1}_{\Real_q}(\eta)\\
&=e^{\Phi_\eta(z)}(2\pi)^{-n+1}\abs{\det M_{\Phi_\eta}}\mathds{1}_{\Real_q}(\eta),
\end{split}
\end{equation}
where $\mathds{1}_{\Real_q}(\eta)$ is the characteristic function of $\Real_q$.
\end{thm}

\begin{rem} \label{s4-r1}
We recall that
$\Phi_\eta=-\sqrt{2}\eta\sum^{n-1}_{j=1}\lambda_j\abs{z_j}^2+\sum^{n-1}_{j,\,t=1}\mu_{j,\,t}\,\ol z_jz_t$.
Since $Y(q)$ holds, we conclude that $\Real_q\subset[-R, R]$ for some $R>0$.
\end{rem}

We return to our situation.
Let $u(z, \theta)\in\Omega^{0,q}(H_n)$, $\norm{u}_{\psi_0}=1$, $\Box^{(q)}_{b,H_n}u=0$. As before,
let $\hat v(z, \eta)$ be the Fourier transform of $u(z, \theta)e^{-\frac{\beta}{2}\theta}$ with respect to $\theta$. From Theorem~\ref{s4-t1}, we know that for $\alpha$ defined in \eqref{s4-e161} we have
\begin{equation} \label{s4-e34}
\begin{split}
\alpha(z,\eta)
\in\Ker\Box^{(q)}_{\Phi_\eta}\cap L^2_{(0,q)}(\Complex^{n-1}, \Phi_\eta)\cap \Omega^{0,q}(\Complex^{n-1})
\end{split}
\end{equation}
for almost all $\eta\in\Real$. Thus, $\alpha(z, \eta)=\int_{\Complex^{n-1}}\!B^{(q)}_{\Phi_\eta}(z, w)\alpha(w,\eta)e^{-\Phi_\eta(w)}dv(w)$ for almost all $\eta\in\Real$. Put
$\hat v(z, \eta)=\sum'_{\abs{J}=q}\hat v_J(z, \eta)d\ol z_J$.
\begin{lem} \label{s4-l2}
Let $J$ be a strictly increasing index, $\abs{J}=q$, and $z\in\Complex^{n-1}$. Then, for almost all $\eta\in\Real$, the following estimate holds:
\begin{equation} \label{s4-e37}
\abs{\hat v_J(z, \eta)}^2\leqslant  e^{\sqrt{2}\,\eta\sum^{n-1}_{j=1}\lambda_j\abs{z_j}^2}\langle B^{(q)}_{\Phi_\eta}(z, z)d\ol z_J,d\ol z_J\rangle\int_{\Complex^{n-1}}\!\abs{\hat v(w, \eta)}^2e^{-\Phi_0(w)}dv(w)\,.
\end{equation}

\end{lem}

\begin{proof}
Let $\varphi\in C^\infty_0(\Complex^{n-1})$ such that $\int_{\Complex^{n-1}}\!\varphi(z)dv(z)=1$, $\varphi\geqslant0$, $\varphi(z)=0$ if $\abs{z}>1$. Put $f_j(z)=j^{2n-2}\varphi(jz)e^{\Phi_\eta(z)}$, $j=1,2,\ldots$.
Then, \[\int_{\Complex^{n-1}}\!f_j(z)e^{-\Phi_\eta(z)}dv(z)=1\,,\quad f_j(z)\To \delta_0\] in the sense of distributions with
respect to $(\ |\ )_{\Phi_\eta}$, that is,
$(h(z)\ |\ f_j(z))_{\Phi_\eta}\To h(0)$, $j\To\infty$,
for all $h\in C^\infty(\Complex^{n-1})$. Thus, for almost all $\eta\in\Real$,
\begin{equation} \label{s4-e38}
\begin{split}
&\abs{e^{-\frac{\eta}{\sqrt{2}}\sum^{n-1}_{j=1}\lambda_j\abs{z_{0,j}}^2}\hat v_J(z_0, \eta)}=\abs{\alpha_J(z_0, \eta)}=\lim_{j\To\infty}\abs{(\alpha(z, \eta)\ |\ f_j(z-z_0)d\ol z_J)_{\Phi_\eta}}\\
&=\lim_{j\To\infty}\abs{(B^{(q)}_{\Phi_\eta}\alpha\ |\ f_j(z-z_0)d\ol z_J)_{\Phi_\eta}}=\lim_{j\To\infty}\abs{(\alpha\ |\ B^{(q)}_{\Phi_\eta}(f_j(z-z_0)d\ol z_J))_{\Phi_\eta}},
\end{split}
\end{equation}
for all $z_0=(z_{0,1},z_{0,2},\ldots,z_{0,n-1})\in\Complex^{n-1}$. Now,
\begin{equation} \label{s4-e39}
\abs{(\alpha(z, \eta)\ |\ B^{(q)}_{\Phi_\eta}(f_j(z-z_0)d\ol z_J))_{\Phi_\eta}}^2\leqslant \norm{\alpha}^2_{\Phi_\eta}\norm{B^{(q)}_{\Phi_\eta}(f_j(z-z_0)d\ol z_J)}^2_{\Phi_\eta}
\end{equation}
and
\begin{equation} \label{s4-e40-0}
\begin{split}
\norm{\alpha}^2_{\Phi_\eta}&\norm{B^{(q)}_{\Phi_\eta}(f_j(z-z_0)d\ol z_J)}^2_{\Phi_\eta}
=\norm{\hat v}^2_{\Phi_0}\norm{B^{(q)}_{\Phi_\eta}(f_j(z-z_0)d\ol z_J)}^2_{\Phi_\eta}\\
&=\norm{\hat v}^2_{\Phi_0}(B^{(q)}_{\Phi_\eta}(f_j(z-z_0)d\ol z_J)\ |\ B^{(q)}_{\Phi_\eta}(f_j(z-z_0)d\ol z_J))_{\Phi_\eta}\\
&\longrightarrow\norm{\hat v}^2_{\Phi_0}\langle B^{(q)}_{\Phi_\eta}(z_0,z_0)d\ol z_J,d\ol z_J\rangle,\ \ j\To\infty.
\end{split}
\end{equation}
From \eqref{s4-e38}, \eqref{s4-e39} and \eqref{s4-e40-0}, we get for all $z_0\in\Complex^{n-1}$,
\[\abs{e^{-\frac{\eta}{\sqrt{2}}\sum^{n-1}_{j=1}\lambda_j\abs{z_{0,j}}^2}\hat v_J(z_0, \eta)}^2\leqslant \norm{\hat v}^2_{\Phi_0}\langle B^{(q)}_{\Phi_\eta}(z_0,z_0)d\ol z_J,d\ol z_J\rangle\]
for almost all $\eta\in\Real$. The lemma follows.
\end{proof}
Put $u(z, \theta)=\sum'_{\abs{J}=q}u_J(z, \theta)d\ol z_J$.
\begin{prop} \label{s4-p1}
For $\abs{J}=q$, $J$ is strictly increasing, we have
\begin{equation} \label{s4-e40}
\abs{u_J(0, 0)}^2\leqslant \frac{1}{4\pi}\int_{\Real}\!\langle B_{\Phi_\eta}(0,0)d\ol z_J,d\ol z_J\rangle dv(\eta).
\end{equation}
\end{prop}
\begin{proof}
Let $\chi\in C^\infty_0(\Real)$, $\int_\Real\!\chi dv(\theta)=1$, $\chi\geqslant0$ and $\chi_\varepsilon\in C^\infty_0(\Real)$, $\chi_\varepsilon(\theta)=\frac{1}{\varepsilon}\chi(\frac{\theta}{\varepsilon})$. Then,
$\chi_\varepsilon\To\delta_0$, $\varepsilon\To0^+$
in the sense of distributions.
Let $\hat\chi_\varepsilon:=\int e^{-i\theta\eta}\chi_\varepsilon(\theta)dv(\theta)$
be the Fourier transform of $\chi_\varepsilon$. We can check that
$\abs{\hat\chi_\varepsilon(\eta)}\leqslant  1$ for all $\eta\in\Real$, $\hat\chi_\varepsilon(\eta)=\hat\chi(\varepsilon\eta)$ and
$\lim_{\varepsilon\To0}\hat\chi_\varepsilon(\eta)=\lim_{\varepsilon\To0}\hat\chi(\varepsilon\eta)=\hat\chi(0)=1$.
Let $\varphi(z)$ be as in the proof of Lemma~\ref{s4-l2}. Put
$g_j(z)=j^{2n-2}\varphi(jz)e^{\Phi_0(z)}$, $j=1,2,\ldots$.
Then, for $J$ is strictly increasing, $\abs{J}=q$, we have
\begin{equation} \label{s4-e41}
u_J(0,0)=\lim_{j\To\infty}\lim_{\varepsilon\To0^+}\int_{H_n}\!\langle u(z,\theta)e^{-\frac{\beta}{2}\theta},\chi_\varepsilon(\theta)g_j(z)d\ol z_J\rangle e^{-\Phi_0(z)}dv(z)dv(\theta).
\end{equation}
From \eqref{s4-e12-2}, we see that
\begin{equation} \label{s4-e42}
\begin{split}
&\iint\!\langle u(z,\theta)e^{-\frac{\beta}{2}\theta},\chi_\varepsilon(\theta)g_j(z)d\ol z_J\rangle e^{-\Phi_0(z)}dv(z)dv(\theta)\\
&=\frac{1}{4\pi}\iint\!\langle\hat v(z, \eta),\hat\chi_\varepsilon(\eta)g_j(z)d\ol z_J\rangle e^{-\Phi_0(z)}dv(\eta)dv(z).
\end{split}
\end{equation}
From \eqref{s4-e37} and Theorem~\ref{s4-t2}, we see that
\begin{equation*}
\abs{\hat v_J(z, \eta)}^2
\leqslant  e^{\sqrt{2}\eta\sum^{n-1}_{j=1}\lambda_j\abs{z_j}^2}\langle B_{\Phi_\eta}(z, z)d\ol z_J,d\ol z_J\rangle\mathds{1}_{\Real_q}(\eta)\int_{\Complex^{n-1}}\abs{\hat v(w, \eta)}^2e^{-\Phi_0(w)}dv(w),
\end{equation*}
for almost all $\eta\in\Real$.
Thus, for fixed $j$,
$\iint\!\abs{\langle\hat v,g_jd\ol z_J\rangle}e^{-\Phi_0(z)}dv(\eta)dv(z)<\infty$.
From this and Lebesque dominated convergence theorem, we conclude that
\begin{equation} \label{s4-e42-1}
\begin{split}
\lim_{\varepsilon\To0^+}&\iint\!\langle\hat v(z, \eta),\hat\chi_\varepsilon(\eta)g_j(z)d\ol z_J\rangle e^{-\Phi_0(z)}dv(\eta)dv(z)\\
&=\iint\!\langle\hat v(z, \eta),g_j(z)d\ol z_J\rangle e^{-\Phi_0(z)}dv(\eta)dv(z).
\end{split}
\end{equation}
From \eqref{s4-e42} and \eqref{s4-e42-1}, \eqref{s4-e41} becomes
\begin{equation} \label{s4-e42-2}
u_J(0, 0)=\lim_{j\To\infty}\frac{1}{4\pi}\iint\!\langle\hat v(z, \eta),g_j(z)d\ol z_J\rangle e^{-\Phi_0(z)}dv(\eta)dv(z).
\end{equation}
Put $f_j(\eta)=\frac{1}{4\pi}\int\!\langle\hat v(z, \eta),g_j(z)d\ol z_J\rangle e^{-\Phi_0(z)}dv(z)$.
Since $\hat v(z, \eta)\in\Omega^{0,q}(\Complex^{n-1})$ for almost all $\eta$, we have
$\lim_{j\To\infty}f_j(\eta)=\frac{1}{4\pi}\hat v_J(0, \eta)$
almost everywhere. Now,
\begin{equation} \label{s4-e42-3}
\begin{split}
\abs{f_j(\eta)}&=\frac{1}{4\pi}\abs{\int\!\langle\hat v(z, \eta),g_j(z)d\ol z_J\rangle e^{-\Phi_0(z)}dv(z)}\\
&=\frac{1}{4\pi}\abs{\int_{\abs{z}\leqslant \frac{1}{j}}\!
\langle\hat v(z, \eta),j^{2n-2}\varphi(jz)d\ol z_J\rangle dv(z)}\\
&\leqslant \frac{1}{4\pi}\Bigr(\int_{\abs{z}\leqslant \frac{1}{j}}\!
\abs{\hat v(z, \eta)}^2e^{-\Phi_0(z)}j^{2n-2}dv(z)
\Bigr)^{\frac{1}{2}}
\Bigr(\int_{\abs{z}\leqslant \frac{1}{j}}\!\abs{\varphi(jz)}^2e^{\Phi_0(z)}j^{2n-2}dv(z)\Bigr)^{\frac{1}{2}}\\
&\leqslant  C_1\Bigr(\int_{\abs{z}\leqslant 1}\!
\abs{\hat v(\tfrac{z}{j}, \eta)}^2e^{-\Phi_0(z/j)}dv(z)\Bigr)^{\frac{1}{2}}\\
&\leqslant  C_2\Bigr(\int_{\abs{z}\leqslant 1}\!e^{\sqrt{2}\eta\sum^{n-1}_{t=1}\lambda_t\abs{\frac{z_t}{j}}^2}\abs{{\rm Tr\,}B_{\Phi_\eta}(\tfrac{z}{j},\tfrac{z}{j})}\mathds{1}_{\Real_q}(\eta)dv(z)\Bigr)^{\frac{1}{2}}\\
&\times\Bigr(\int_{\Complex^{n-1}}\abs{\hat v(w, \eta)}^2e^{-\Phi_0(w)}dv(w)\Bigr)^{\frac{1}{2}}\ \ (\mbox{here we used \eqref{s4-e37} and Theorem~\ref{s4-t2}})\\
&\leqslant  C_3\Bigr(\int_{\Complex^{n-1}}\abs{\hat v(w, \eta)}^2e^{-\Phi_0(w)}dv(w)\Bigr)^{\frac{1}{2}}\mathds{1}_{\Real_q}(\eta),
\end{split}
\end{equation}
where $C_1, C_2, C_3$ are positive constants. From this and the Lebesgue dominated convergence theorem, we conclude that
\[
u_J(0, 0)=\lim_{j\To\infty}\int f_j(\eta)dv(\eta)=\int\lim_{j\To\infty}f_j(\eta)dv(\eta)
=\frac{1}{4\pi}\int\hat v_J(0, \eta)dv(\eta)\,.
\]
Thus,
\begin{equation} \label{s4-e48}
\abs{u_J(0,0)}\leqslant \frac{1}{4\pi}\int\!\abs{\hat v_J(0,\eta)}dv(\eta).
\end{equation}
Since $\iint\!\abs{\hat v(w, \eta)}^2e^{-\Phi_0(w)}dv(\eta)dv(w)=4\pi$ we obtain from Lemma~\ref{s4-l2} that
\begin{equation} \label{s4-e49}
\begin{split}
\abs{\int\!\abs{\hat v_J(0,\eta)}dv(\eta)}^2&\leqslant 4\pi\int\!\frac{\abs{\hat v_J(0,\eta)}^2}{\int\!\abs{\hat v(w, \eta)}^2e^{-\Phi_0(w)}dv(w))}dv(\eta)\\
&\leqslant 4\pi\int\!\langle B^{(q)}_{\Phi_\eta}(0,0)d\ol z_J,d\ol z_J\rangle dv(\eta).
\end{split}
\end{equation}
Estimtes \eqref{s4-e48} and \eqref{s4-e49} yield the conclusion.
\end{proof}

From Proposition~\ref{s4-p1}, we know that for all
$u(z, \theta)=\sum'_{\abs{J}=q}u_J(z, \theta)d\ol z_J\in\Omega^{0,q}(H_n)$,
satisfying $\norm{u}_{\psi_0}=1$, $\Box^{(q)}_{b,H_n}u=0$, we have
\[\abs{u_J(0,0)}^2\leqslant \frac{1}{4\pi}\int\!\langle B^{(q)}_{\Phi_\eta}(0,0)d\ol z_J,d\ol z_J\rangle dv(\eta).\]
Thus, $S^{(q)}_{J,H_n}(0)\leqslant \frac{1}{4\pi}\int\!\langle B^{(q)}_{\Phi_\eta}(0,0)d\ol z_J,d\ol z_J\rangle dv(\eta)$
for all strictly increasing $J$, $\abs{J}=q$.
Hence $\sum'_{\abs{J}=q}S^{(q)}_{J,H_n}(0)\leqslant \frac{1}{4\pi}\int\!{\rm Tr\,}B^{(q)}_{\Phi_\eta}(0,0)dv(\eta)$.
From this and Theorem~\ref{s4-t2}, we get

\begin{thm} \label{s4-t3}
We have
$\sum'_{\abs{J}=q}S^{(q)}_{J,H_n}(0)\leqslant\frac{1}{2(2\pi)^{n}}\int_{\Real_q}\!\abs{\det M_{\Phi_\eta}}dv(\eta)$,
where $M_{\Phi_\eta}$ is as in \eqref{s4-e29-1} and $\Real_q$ is as in \eqref{s4-e30}.
\end{thm}

\subsection{The Szeg\"{o} kernel function on the Heisenberg group}
In the rest of this section, we calculate the extremal function for the Heisenberg group (see Theorem \ref{s4-t4}).
For $\eta\in\Real$, we can find $z_j(\eta)=\sum^{n-1}_{t=1}a_{j,t}(\eta)z_t$, $j=1,\ldots,n-1$,
such that $\Phi_\eta=\sum^{n-1}_{j=1}\nu_j(\eta)\abs{z_j(\eta)}^2$,
where $\nu_1(\eta),\ldots,\nu_{n-1}(\eta)$, are the eigenvalues of $M_{\Phi_\eta}$, $a_{j,t}(\eta)\in\Complex$, $j, t=1,\ldots,n-1$. If $\eta\in \Real_q$, we assume that
$\nu_1(\eta)<0,\ldots,\nu_q(\eta)<0,\nu_{q+1}(\eta)>0,\ldots,\nu_{n-1}(\eta)>0$.
The following is essentially well-known (see~\cite{Be04}).

\begin{prop} \label{s4-p2}
Put
\begin{equation} \label{s4-e53}
\alpha(z, \eta)=\frac{1}{\sqrt{2}}C_0\abs{\det M_{\Phi_\eta}}\mathds{1}_{\Real_q}(\eta)e^{\nu_1(\eta)\abs{z_1(\eta)}^2+\cdots+\nu_q(\eta)\abs{z_q(\eta)}^2}
d\ol{z_1(\eta)}\wedge\cdots\wedge d\ol{z_q(\eta)},
\end{equation}
where
$C_0=(2\pi)^{1-\frac{n}{2}}\Bigr(\int_{\Real_q}\abs{\det M_{\Phi_\eta}}dv(\eta)\Bigr)^{-\frac{1}{2}}$.
Then, $\Box^{(q)}_{\Phi_\eta}\alpha(z, \eta)=0$
and
\begin{equation} \label{s4-e55}
\int_{\Complex^{n-1}}(1+\abs{z}^2)^{m'}\abs{\pr^m_x\alpha(z, \eta)}^2e^{-\Phi_\eta(z)}dv(z)<\infty
\end{equation}
and the value $\int_{\Complex^{n-1}}(1+\abs{z}^2)^{m'}\abs{\pr^m_x\alpha(z, \eta)}^2e^{-\Phi_\eta(z)}dv(z)$ can be bounded by some positive continuous function of the eigenvalues of $M_{\Phi_\eta}$, $\eta\in \Real_q$,
for all $m\in\mathbb N_0^{2n-2}$, $m'\in\mathbb N_0$. Moreover, we have
\begin{equation} \label{s4-e56}
\int_{\Complex^{n-1}}\abs{\alpha(z, \eta)}^2e^{-\Phi_\eta(z)}dv(z)=\pi\Bigr(\int_{\Real_q}\abs{\det M_{\Phi_\eta}}dv(\eta)\Bigr)^{-1}\abs{\det M_{\Phi_\eta}}\mathds{1}_{\Real_q}(\eta).
\end{equation}
\end{prop}

Set
\begin{equation} \label{s4-e57}
u(z,\theta)=\frac{1}{2\pi}\int e^{i\theta\eta+\frac{\beta\theta}{2}+\big(\frac{\eta}{\sqrt{2}}-
\frac{i\beta}{2\sqrt{2}}\big)\lambda\abs{z}^2}\alpha(z, \eta)\,dv(\eta)\in\Omega^{0,q}(H_n),
\end{equation}
where $\alpha(z, \eta)$ is as in \eqref{s4-e53} and $\lambda\abs{z}^2:=\sum^{n-1}_{j=1}\lambda_j\abs{z_j}^2$.

\begin{prop} \label{s4-p3}
We have that
\begin{equation} \label{s4-e58}
\Box^{(q)}_{b,H_n}u=0,
\end{equation}
\begin{equation} \label{s4-e59}
\norm{u}_{\psi_0}=1
\end{equation}
and
\begin{equation} \label{s4-e60}
\abs{u(0,0)}^2=\frac{1}{2(2\pi)^{n}}\int_{\Real_q}\abs{\det M_{\Phi_\eta}}dv(\eta).
\end{equation}
Moreover, we have
\begin{equation} \label{s4-e61}
\int_{H_n}\abs{\pr^m_x\pr^{m'}_\theta u(z, \theta)}^2e^{-\psi_0(z, \theta)}dv(z)dv(\theta)<\infty
\end{equation}
and $\int_{H_n}\abs{\pr^m_x\pr^{m'}_\theta u(z, \theta)}^2e^{-\psi_0(z, \theta)}dv(z)dv(\theta)$ is bounded above by some positive continuous function of the eigenvalues of $M_{\Phi_\eta}$, $\eta\in \Real_q$, $\beta$ and $\lambda_j$, $j=1,\ldots,n-1$, for all $m\in\mathbb N_0^{2n-2}$, $m'\in\mathbb N_0$.
\end{prop}

\begin{proof}
In view of the proof of Lemma~\ref{s4-l1}, we see that
\[\Box^{(q)}_{b,H_n}u(z, \theta)=\frac{1}{2\pi}\int e^{i\theta\eta+\frac{\beta\theta}{2}+\big(\frac{\eta}{\sqrt{2}}-
\frac{i\beta}{2\sqrt{2}}\big)\lambda\abs{z}^2}(\Box^{(q)}_{\Phi_\eta}\alpha)(z, \eta)dv(\eta)=0\,,\]
which implies \eqref{s4-e58}. Now,
\begin{equation} \label{s4-e62}
\begin{split}
&\int\abs{u(z, \theta)}^2e^{-\psi_0(z, \theta)}dv(z)dv(\theta)\\
&\quad=\frac{1}{(2\pi)^2}\int\abs{\int e^{i\theta\eta+\frac{\beta\theta}{2}+\big(\frac{\eta}{\sqrt{2}}-
\frac{i\beta}{2\sqrt{2}}\big)\lambda\abs{z}^2}\alpha(z, \eta)dv(\eta)}^2e^{-\beta\theta-\Phi_0(z)}dv(\theta)dv(z)\\
&\quad=\frac{1}{(2\pi)^2}\int\abs{\int e^{i\theta\eta+\frac{\eta}{\sqrt{2}}\lambda\abs{z}^2}\alpha(z, \eta)dv(\eta)}^2dv(\theta)e^{-\Phi_0(z)}dv(z).
\end{split}
\end{equation}
From Parseval's formula, we have
\begin{equation} \label{s4-e63}
\begin{split}
&\frac{1}{(2\pi)^2}\int\abs{\int e^{i\theta\eta+\frac{\eta}{\sqrt{2}}\lambda\abs{z}^2}\alpha(z, \eta)dv(\eta)}^2dv(\theta)=\frac{1}{\pi}\int e^{\sqrt{2}\eta\lambda\abs{z}^2}\abs{\alpha(z, \eta)}^2dv(\eta).
\end{split}
\end{equation}
In view of \eqref{s4-e63}, \eqref{s4-e62} becomes
\[\int\abs{u(z, \theta)}^2e^{-\psi_0(z, \theta)}dv(z)dv(\theta)=\frac{1}{\pi}\iint\abs{\alpha(z, \eta)}^2e^{-\Phi_\eta(z)}dv(z)dv(\eta).\]
From \eqref{s4-e56}, we can check that
$\frac{1}{\pi}\iint\abs{\alpha(z, \eta)}^2e^{-\Phi_\eta(z)}dv(z)dv(\eta)=1$ so we infer
\eqref{s4-e59}.
We obtain \eqref{s4-e60} from the following
\begin{equation*}
\begin{split}
\abs{u(0,0)}^2&=\frac{1}{(2\pi)^2}\abs{\int\alpha(0, \eta)dv(\eta)}^2=\frac{1}{2(2\pi)^2}C^2_0\Bigr(\int_{\Real_q}\abs{\det M_{\Phi_\eta}}dv(\eta)\Bigr)^2\\
&=\frac{1}{2(2\pi)^{n}}\int_{\Real_q}\abs{\det M_{\Phi_\eta}}dv(\eta).
\end{split}
\end{equation*}
Finally, from \eqref{s4-e55}, \eqref{s4-e57}, Parseval's formula and the statement after \eqref{s4-e55}, we get \eqref{s4-e61} and the last statement of this proposition.
\end{proof}

From Proposition~\ref{s4-p3} and Theorem~\ref{s4-t3}, we get the main result of this section:

\begin{thm} \label{s4-t4}
We have
$\sum'_{\abs{J}=q}S^{(q)}_{J,H_n}(0)=\frac{1}{2(2\pi)^{n}}\int_{\Real_q}\!\abs{\det M_{\Phi_\eta}}dv(\eta)$,
where $M_{\Phi_\eta}$ is as in \eqref{s4-e29-1} and $\Real_q$ is as in \eqref{s4-e30}.
\end{thm}

\section{Szeg\"{o} kernel asymptotics and weak Morse inequalities on CR manifolds}

In this section we first study the properties of the Hermitian form $M^\phi_p$ introduced in Definition \ref{s1-d3}, especially its dependence of local trivializations. We then prove \eqref{s1-e19}, i.e. the second part of Theorem \ref{t-main1} (cf. Theorem \ref{s5-t1}). Finally, we prove Theorem \ref{t-main2}.

We assume that
condition $Y(q)$ holds.
Let $s$ be a local trivializing section of $L$ on an open subset
$D\subset X$. Let $\phi\in C^\infty(D;\, \Real)$ be the weight of the Hermitian metric $h^L$ relative to $s$, that is, the pointwise norm of $s$ is $\abs{s(x)}_{h^L}^2=e^{-\phi(x)}$, for $x\in D$.
Until further notice, we work on $D$.
Recall that
$M^\phi_p$, $p\in D$, is the Hermitian quadratic form on $T^{1, 0}_pX$ defined by
\[
M^\phi_p(U, \ol V)=\frac{1}{2}\Big\langle U\wedge\ol V, d(\ddbar_b\phi-\pr_b\phi)(p)\Big\rangle\,,\quad U, V\in T^{1, 0}_pX\,.
\]
\begin{lem} \label{s5-l1}
For any $U, V\in T^{1, 0}_pX$, pick $\mU, \mV\in C^\infty(D;\, T^{1, 0}X)$ that satisfy $\mU(p)=U$,
$\mV(p)=V$. Then,
\begin{equation} \label{s5-e2}
M^\phi_p(U, \ol V)=-\frac{1}{2}\big\langle\big[\,\mU, \ol{\mV}\,\big](p), \ddbar_b\phi(p)-\pr_b\phi(p)\big\rangle
+\frac{1}{2}\big(\mU\ol{\mV}+\ol{\mV}\mU\big)\phi(p).
\end{equation}

\end{lem}

\begin{proof}
Recall that for a $1$-form $\alpha$ and vector fields $V_1$, $V_2$ we have
\begin{equation} \label{s5-e3}
\langle V_1\wedge V_2, d\alpha\rangle=V_1(\langle V_2, \alpha\rangle)-V_2(\langle V_1, \alpha\rangle)-\langle[V_1, V_2], \alpha\rangle,
\end{equation}
Taking $V_1=\mU$, $V_2=\ol{\mV}$ and $\alpha=\ddbar_b\phi-\pr_b\phi$ in
\eqref{s5-e3}, we get
\begin{equation} \label{s5-e4}
\begin{split}
&\big\langle\mU\wedge\ol{\mV}, d(\ddbar_b\phi-\pr_b\phi)\big\rangle\\
&\quad=\mU\big(\big\langle\ol{\mV},\ddbar_b\phi-\pr_b\phi\big\rangle)-\ol{\mV}\big(\big\langle\mU,\ddbar_b\phi-\pr_b\phi\big\rangle\big)-\big\langle\big[\mU, \ol{\mV}\,\big], \ddbar_b\phi-\pr_b\phi\big\rangle.
\end{split}
\end{equation}
Note that $\langle\ol{\mV}, \ddbar_b\phi-\pr_b\phi\rangle=\langle\ol{\mV},\ddbar_b\phi\rangle=\ol{\mV}\phi$
and $\langle\mU, \ddbar_b\phi-\pr_b\phi\rangle=\langle\mU, -\pr_b\phi\rangle=-\mU\phi$.
From this observation, \eqref{s5-e4} becomes
$\langle\mU\wedge\ol{\mV}, d(\ddbar_b\phi-\pr_b\phi)\rangle=(\mU\ol{\mV}+\ol{\mV}\mU)\phi-\langle[\,\mU, \ol{\mV}\,], \ddbar_b\phi-\pr_b\phi\rangle$.
The lemma follows.
\end{proof}

The definition of $M^\phi_p$ depends on the choice of local trivializations.
Let $\Td D$ be another local trivialization with
$D\cap\Td D\neq\emptyset$. Let $\Td s$ be a local trivializing section of $L$ on the open subset $\Td D$ and the pointwise norm of $\Td s$ is $\abs{\Td s(x)}_{h^L}^2=e^{-\Td\phi(x)}$, $\Td\phi\in C^\infty(\Td D; \Real)$.
Since $\Td s=gs$ on $D\cap\Td D$, for some non-zero CR function $g$, we
can check that
\begin{equation} \label{s5-e6}
\Td\phi=\phi-2\log{\abs{g}}\quad \text{on $D\cap\Td D$}\,.
\end{equation}

\begin{prop} \label{s5-p1}
For $p\in D\cap\Td D$, we have
\begin{equation} \label{s5-e7}
M^\phi_p=M^{\Td\phi}_p+i\Big(\,\frac{Tg}{g}-\frac{T\,\ol g}{\ol g}\,\Bigr)(p)\mathcal{L}_p\,.
\end{equation}
where $T$ is the real vector field on $X$ defined by \eqref{s1-d11}.
\end{prop}

\begin{proof}
From \eqref{s5-e6}, we can check that $\ddbar_b\Td\phi=\ddbar_b\phi-\frac{\ddbar_b\ol g}{\ol g}$ and $\pr_b\Td\phi=\pr_b\phi-\frac{\pr_bg}{g}$ on $D\cap\Td D$.
From above, we have
\begin{equation} \label{s5-e9}
\langle[U, \ol V\,], \ddbar_b\phi-\pr_b\phi\rangle=\langle[U, \ol V\,], \ddbar_b\Td\phi-\pr_b\Td\phi\rangle+
\Big\langle[U, \ol V\,], \frac{\ddbar_b\ol g}{\ol g}-\frac{\pr_bg}{g}\Big\rangle,
\end{equation}
where $U, V\in C^\infty(D\cap\Td D;\, T^{1, 0}X)$.
From \eqref{s5-e6}, we have
\begin{equation} \label{s5-e10}
\begin{split}
(U\ol V+\ol VU)\phi&=(U\ol V+\ol VU)(\Td\phi+2\log{\abs{g}}) \\
&=(U\ol V+\ol VU)\Td\phi+\frac{\ol VUg}{g}+\frac{U\ol V\ol g}{\ol g}\ \ (\mbox{since $\ol Vg=0$, $U\ol g=0$}) \\
&=(U\ol V+\ol VU)\Td\phi-\frac{[U, \ol V\,]g}{g}+\frac{[U, \ol V\,]\ol g}{\ol g}.
\end{split}
\end{equation}
From \eqref{s5-e9}, \eqref{s5-e10} and \eqref{s5-e2}, we see that
\begin{equation} \label{s5-e11}
\begin{split}
M^\phi_p(U(p), \ol V(p))=&M^{\Td\phi}_p(U(p), \ol V(p))-\Big\langle[U, \ol V\,](p), \frac{1}{2}\frac{\ddbar_b\ol g}{\ol g}(p)-\frac{1}{2}\frac{\pr_bg}{g}(p)\Big\rangle\\
&-\frac{1}{2}\frac{[U, \ol V\,]g}{g}(p)+\frac{1}{2}\frac{[U, \ol V\,]\ol g}{\ol g}(p).
\end{split}
\end{equation}
We write $[U, \ol V\,]=Z+\ol W+\alpha(x)T$, where $Z, W\in C^\infty(D\cap\Td D;\, T^{1, 0}X)$ and
$\alpha(x)\in C^\infty(D\cap\Td D;\, \Complex)$. We can check that
$\alpha(p)=-2i\mathcal{L}_p(U(p), \ol V(p))$.
Since $\ol Wg=0$ and $Zg=\langle Z, \pr_bg\rangle=\langle[U, \ol V\,], \pr_bg\rangle$, we have
\begin{equation} \label{s5-e12}
[U, \ol V\,]g(p)=Zg(p)+\alpha(p)Tg(p)=\langle[U, \ol V\,](p), \pr_b g(p)\rangle-2i\mathcal{L}_p(U(p), \ol V(p))Tg(p).
\end{equation}
Similarly, we have
\begin{equation} \label{s5-e13}
[U, \ol V\,]\ol g(p)=\langle[U, \ol V\,](p), \ddbar_b\ol g(p)\rangle-2i\mathcal{L}_p(U(p), \ol V(p))T\ol g(p).
\end{equation}
Combining \eqref{s5-e12}, \eqref{s5-e13} with \eqref{s5-e11}, we get
\[M^\phi_p(U(p), \ol V(p))=M^{\Td\phi}_p(U(p), \ol V(p))+i\mathcal{L}_p(U(p), \ol V(p))\Bigl(\frac{Tg}{g}-\frac{T\ol g}{\ol g}\Bigr)(p).\]
The proposition follows.
\end{proof}


\noindent
Recall that $\Real_{\phi(p),q}$ was defined in \eqref{s1-e15}. From \eqref{s5-e7}, we see that
\begin{equation} \label{s5-e14}
\begin{split}
\Real_{\Td\phi(p),q}&=\Real_{\phi(p),q}+i\big(\tfrac{Tg}{g}-\tfrac{T\ol g}{\ol g}\big)(p)\\
&=\big\{s+i\big(\tfrac{Tg}{g}-\tfrac{T\ol g}{\ol g}\big);\, s\in\Real_{\phi(p),q}\big\}.
\end{split}
\end{equation}

Recall that $\det(M^\phi_x+s\mathcal{L}_x)$ denotes the product of all the eigenvalues of $M^\phi_x+s\mathcal{L}_x$.
From \eqref{s5-e7} and \eqref{s5-e14}, we see that the function
$x\mapsto\int_{\Real_{\phi(x),q}}\abs{\det(M^\phi_x+s\mathcal{L}_x)}ds$
does not depend on the choice of $\phi$. Thus, the function $x\To\int_{\Real_{\phi(x),q}}\abs{\det(M^\phi_x+s\mathcal{L}_x)}ds$ is well-defined.
Since $M^\phi_x$ and $\mathcal{L}_x$ are continuous functions of $x$, we conclude that
$x\To\int_{\Real_{\phi(x),q}}\abs{\det(M^\phi_x+s\mathcal{L}_x)}ds$ is a continuous function of $x$.

\begin{rem} \label{s5-r3}
We take local coordinates
$(x, \theta)=(z, \theta)=(x_1,\ldots,x_{2n-2}, \theta)$, $z_j=x_{2j-1}+ix_{2j}$, $j=1,\ldots,n-1$, as in \eqref{s1-e20} and \eqref{s1-e21}
defined on some neighborhood of $p$. Then, it is straight forward to see that
$\mathcal{L}_p=\sum^{n-1}_{j=1}\lambda_jdz_j\otimes d\ol z_j$
and $M^\phi_p=\sum^{n-1}_{j,t=1}\mu_{j,\,t}dz_t\otimes d\ol z_j$.
Thus,
\begin{equation} \label{s5-e18-1}
\int_{\Real_{\phi(p),q}}\abs{\det(M^\phi_p+s\mathcal{L}_p)}ds=
\int_{\Real_{\phi(p),q}}\abs{\det\left(\mu_{j,\,t}+s\delta_{j,\,t}\lambda_j\right)^{n-1}_{j,t=1}}ds
\end{equation}
and
\begin{equation} \label{s5-e19}
\begin{split}
\Real_{\phi(p),\,q}=\big\{s\in\Real;&\, \mbox{the matrix $\left(\mu_{j,\,t}+s\delta_{j,\,t}\lambda_j\right)^{n-1}_{j,\,t=1}$ has $q$ negative eigenvalues} \\
&\quad\mbox{and $n-1-q$ positive eigenvalues}\big\}.
\end{split}
\end{equation}
\end{rem}
We prove now the precise bound \eqref{s1-e19} which is one of the main results of this work.
\begin{thm} \label{s5-t1}
We have for all $p\in X$
\[\limsup_{k\To\infty}k^{-n}\pit^{(q)}_k(p)\leqslant\frac{1}{2(2\pi)^{n}}\int_{\Real_{\phi(p),\,q}}\abs{\det(M^\phi_p+s\mathcal{L}_p)}ds\, .
\]
\end{thm}

\begin{proof}
For $p\in X$, let $(x, \theta)=(z, \theta)=(x_1,\ldots,x_{2n-1})$, $z_j=x_{2j-1}+ix_{2j}$, $j=1,\ldots,n-1$, be the coordinate as in \eqref{s1-e20} and
\eqref{s1-e21} defined on some neighborhood of $p$. From Theorem~\ref{s3-t2}, we have that
$\limsup_{k\To\infty}k^{-n}\pit^{(q)}_k(0)\leqslant \sum'_{\abs{J}=q}S^{(q)}_{J,H_n}(0)$.
From Theorem~\ref{s4-t4}, we know that
$\sum'_{\abs{J}=q}S^{(q)}_{J,H_n}(0)=\frac{1}{2(2\pi)^{n}}\int_{\Real_q}\!\abs{\det M_{\Phi_\eta}}dv(\eta)$,
where $M_{\Phi_\eta}$ is as in \eqref{s4-e29-1} and $\Real_q$ is as in \eqref{s4-e30}. Thus,
\begin{equation} \label{s5-e21}
\limsup_{k\To\infty}k^{-n}\pit^{(q)}_k(0)\leqslant\frac{1}{2(2\pi)^{n}}\int_{\Real_q}\!\abs{\det M_{\Phi_\eta}}dv(\eta).
\end{equation}
From \eqref{s4-e29-1}, \eqref{s4-e30} and the definition of $\Phi_\eta$ (see \eqref{s4-e13}), we see that
\begin{equation} \label{s5-e22}
\det M_{\Phi_\eta}=\det\left(\mu_{j,\,t}-\sqrt{2}\eta\lambda_j\delta_{j,\,t}\right)^{n-1}_{j,\,t=1}
\end{equation}
and
\begin{equation} \label{s5-e23}
\begin{split}
\Real_q=\big\{\eta\in\Real;\, &\text{the matrix $\left(\mu_{j,\,t}-\sqrt{2}\eta\delta_{j,\,t}\lambda_j\right)^{n-1}_{j,\,t=1}$ has $q$ negative eigenvalues} \\
&\text{and $n-1-q$ positive eigenvalues}\big\}.
\end{split}
\end{equation}
Note that $dv(\eta)=\sqrt{2}d\eta$. From this and \eqref{s5-e22}, \eqref{s5-e23}, \eqref{s5-e18-1}, \eqref{s5-e19}, it is easy to see that
$\int_{\Real_q}\!\abs{\det M_{\Phi_\eta}}dv(\eta)=\int_{\Real_{\phi(p),\,q}}\abs{\det(M^\phi_p+s\mathcal{L}_p)}ds$.
From this and \eqref{s5-e21}, the theorem follows.
\end{proof}

%

\begin{proof}[Proof of Theorem \ref{t-main2}]
By \eqref{s2-e11}-\eqref{s1-e13} we have $\dim H^q_b(X,L^k)=\int_X\!\pit^{(q)}_k(x)dv_X(x)$.
In view of Theorem~\ref{s3-t1}, $\sup_k k^{-n}\pit^{(q)}_k(\cdot)$ is integrable on $X$.
Thus, we can apply Fatou's lemma and we get by using Theorem \ref{s5-t1}:
\begin{equation*}
\begin{split}
&\limsup_{k\To\infty}\,k^{-n}\dim H^q_b(X,L^k)\leqslant \int_X\!\limsup_{k\To\infty}k^{-n}\pit^{(q)}_k(x)dv_X(x)\\
&\quad\leqslant\frac{1}{2(2\pi)^{n}}\int_X\Bigr(\int_{\Real_{\phi(x),q}}\abs{\det(M^\phi_x+s\mathcal{L}_x)}ds\Bigr)dv_X(x).
\end{split}
\end{equation*}
The theorem follows.
\end{proof}

\section{Strong Morse inequalities on CR manifolds}

In this section, we will establish the strong Morse inequalities on CR manifolds. We first recall some well-known facts.
Until further notice, we assume that $Y(q)$ holds.
We know \cite[Th.\,7.6]{Ko65}, \cite[Th.\,5.4.11--12]{FK72}, \cite[Cor.\,8.4.7--8]{CS01} that if $Y(q)$ holds, then $\Box^{(q)}_{b,k}$ has a discrete
spectrum, each eigenvalue occurs with finite multiplicity and all eigenforms are smooth. For $\lambda\in\Real$, let
$\cH^q_{b,\,\leqslant\lambda}(X, L^k)$ denote the space spanned by the eigenforms of $\Box^{(q)}_{b,k}$ whose eigenvalues are bounded by
$\lambda$ and denote by $\pit^{(q)}_{k,\,\leqslant \lambda}$ the Szeg\"{o} kernel function of the space $\cH^q_{b,\,\leqslant \lambda}(X,L^k)$.
Similarly, let $\cH^q_{b,>\lambda}(X, L^k)$ denote the space spanned by the eigenforms of $\Box^{(q)}_{b,k}$ whose eigenvalues are
$>\lambda$.

Let $Q_b$ be the
Hermitian form on $\Omega^{0,q}(X, L^k)$ defined for $u, v\in\Omega^{0,q}(X, L^k)$ by
\[
Q_b(u, v)=(\ddbar_{b,k}u\ |\ \ddbar_{b,k}v)_k+(\ddbar^*_{b,k}u\ |\ \ddbar^*_{b,k}v)_k+(u\ |\ v)_k=(\Box^{(q)}_{b,k}u\ |\ v)_k+(u\ |\ v)_k\,.
\]
Let $\ol{\Omega^{0,q}(X, L^k)}$ be the completion
of $\Omega^{0,q}(X, L^k)$ under $Q_b$ in $L^2_{(0,q)}(X, L^k)$.
For $\lambda>0$, we have the orthogonal spectral decomposition with
respect to $Q_b$:
\begin{equation} \label{s6-e1}
\ol{\Omega^{0,q}(X, L^k)}=\cH^q_{b,\,\leqslant \lambda}(X, L^k)\oplus\ol{\cH^q_{b,\,>\lambda}(X, L^k)},
\end{equation}
where $\ol{\cH^q_{b,\,>\lambda}(X, L^k)}$ is the completion of $\cH^q_{b,\,>\lambda}(X, L^k)$ under $Q_b$ in $L^2_{(0,q)}(X, L^k)$.

Let
$u\in\ol{\cH^q_{b,\,>\lambda}(X, L^k)}\cap\Omega^{0,q}(X, L^k)$. There are $f_j\in \cH^q_{b,\,>\lambda}(X, L^k)$, $j=1,2,\ldots$, such that $Q_b(f_j-u)\To 0$, as $j\To\infty$.
From this, we can check that $(\Box^{(q)}_{b,k}f_j\ |\ f_j)_k\To(\Box^{(q)}_{b,k}u\ |\ u)_k$, as $j\To\infty$, and
\begin{equation} \label{s6-e2}
\norm{u}^2=\lim_{j\To\infty}\norm{f_j}^2=\lim_{j\To\infty}(f_j\ |\ f_j)_k
\leqslant \lim_{j\To\infty}\frac{1}{\lambda}\big(\Box^{(q)}_{b,k}f_j\ \big|\ f_j\big)_k=\frac{1}{\lambda}\big(\Box^{(q)}_{b,k}u\ \big|\ u\big)_k\,.
\end{equation}
We return to our situation. We will use the same notations as in section 3. For a given point $p\in X$, let $s$ be a local trivializing section of $L$ on an open neighborhood of $p$ and $\abs{s}^ 2=e^{-\phi}$. Let $(x, \theta)=(z, \theta)=(x_1,\ldots,x_{2n-2},\theta)$, $z_j=x_{2j-1}+ix_{2j}$,
$j=1,\ldots,n-1$, be the local coordinates as in \eqref{s1-e20} and \eqref{s1-e21} defined on an open set $D$ of $p$.
Note that $(x(p), \theta(p))=0$. We identify $D$ with some open set of $H_n$.

Let $u(z, \theta)=\sum'_{\abs{J}=q}u_J(z, \theta)d\ol z_J\in\Omega^{0,q}(H_n)$ be as in \eqref{s4-e57} and Proposition~\ref{s4-p3}.
From \eqref{s4-e61} and the statement after \eqref{s4-e61}, we know that the value
\[\int\abs{\pr^m_x\pr^{m'}_\theta u}^2e^{-\psi_0}dv(z)dv(\theta)\]
is finite and can be bounded by some positive continuous function of the eigenvalues of $M_{\Phi_\eta}$, $\eta\in\Real_q$, $\beta$ and $\lambda_j$, $j=1,\ldots,n-1$, for all $m'\in\mathbb N_0$, $m\in\mathbb N_0^{2n-2}$. Since $X$ is compact, we deduce that for every
$m\in\mathbb N_0^{2n-2}$, $m'\in\mathbb N_0$, we can find $M_{m,m'}>0$ independent of the point $p$, such that
\begin{equation} \label{s6-e3}
\int_{H_n}\abs{\pr^m_x\pr^{m'}_\theta u}^2e^{-\psi_0}dv(\theta)dv(z)< M_{m,m'}.
\end{equation}
Set
$\beta_k(z, \theta)=\chi_k(\sqrt{k}z, k\theta)\sum'_{\abs{J}=q}u_J(\sqrt{k}z, k\theta)e_J(z, \theta)\in\Omega^{0,q}(D)$. Here $\chi$ is a smooth function, $0\leqslant \chi\leqslant  1$, supported on $D_1$
which equals one on $D_{\frac{1}{2}}$ and
\[
\chi_k(z, \theta)=\chi\Big(\frac{z}{\log k}, \frac{\theta}{\sqrt{k}\log k}\Big).
\]
We remind that $(e_j)_{j=1,\ldots,\,n-1}$ denotes the basis of $T^{*0,1}X$,
which is dual to $(\ol U_j)_{j=1,\ldots,\,n-1}$, where $(U_j)_{j=1,\ldots,\,n-1}$ are as in \eqref{s1-e20}. We notice that for $k$ large,
${\rm Supp\,}\beta_k\subset D_{\frac{\log k}{\sqrt{k}}}$. From Proposition~\ref{s3-p1} and \eqref{s3-e23}, we have
\begin{equation} \label{s6-e5}
(\Box^{(q)}_{s,(k)})(F^*_k\beta_k)=\Box^{(q)}_{b,H_n}\bigr(\chi_k(z, \theta)u(z, \theta)\bigr)+\varepsilon_kP_k(F^*_k\beta_k),
\end{equation}
where $\varepsilon_k$ is a sequence tending to zero with $k\To\infty$ and $P_k$ is a second order differential operator and all
the derivatives of the coefficients of $P_k$ are uniform bounded in $k$. Note that
$\Box^{(q)}_{b,H_n}u=0$
and
$\sup_{(z, \theta)\in D_{\log k}}\abs{kF^*_k\phi_0-\psi_0}\To0$ as $k\To\infty$ ($\phi_0$ is as in \eqref{s1-e23}).
From this, \eqref{s6-e5} and \eqref{s6-e3}, we deduce that there is a sequence $\delta_k>0$, independent of the point $p$ and
tending to zero such that
\begin{equation} \label{s6-e6}
\big\|\Box^{(q)}_{s,(k)}(F^*_k\beta_k)\big\|_{kF^*_k\phi_0}\leqslant \delta_k.
\end{equation}
Similarly, we have for all $m\in\mathbb N$
\begin{equation} \label{s6-e7}
\big\|(\Box^{(q)}_{s,(k)})^m(F^*_k\beta_k)\big\|_{kF^*_k\phi_0}\To0\ \ \mbox{as $k\To\infty$}\,.
\end{equation}
Now define $\alpha_k\in\Omega^{0,q}(X, L^k)$ by
\begin{equation} \label{s6-e8}
\alpha_k(z, \theta)=s^kk^{\frac{n}{2}}e^{kR}\beta_k(z, \theta)\ ,
\end{equation}
where $R(z, \theta)$ is as in \eqref{s1-e22}.
We can check that
\begin{equation} \label{s6-e9}
k^{-n}\abs{\alpha_k(0, 0)}^2=\abs{\beta_k(0, 0)}^2=\abs{u(0, 0)}^2=\frac{1}{2(2\pi)^{n}}\int_{\Real_{\phi(p),\,q}}\abs{\det(M^\phi_p+s\mathcal{L}_p)}ds
\end{equation}
for all $k$, and
\begin{equation} \label{s6-e10}
\begin{split}
\norm{\alpha_k}^2&=\int k^ne^{k(R+\ol R)}\abs{\beta_k}^2e^{-k\phi}m(z, \theta)dv(z)dv(\theta)\\
&=\int k^ne^{-k\phi_0}\abs{\beta_k}^2m(z, \theta)dv(z)dv(\theta)\\
&=\int e^{-kF^*_k\phi_0}\abs{\chi_k(z, \theta)}^2\abs{u(z, \theta)}^2
m\big(\tfrac{z}{\sqrt{k}}, \tfrac{\theta}{k}\big)
dv(z)dv(\theta)\\
&\To\int\abs{u}^2e^{-\psi_0(z, \theta)}dv(z)dv(\theta)=1,\ \ \mbox{as $k\To\infty$},
\end{split}
\end{equation}
where $m(z, \theta)dv(z)dv(\theta)$ is the volume form. Note that $m(0, 0)=1$. Moreover, we have
\begin{equation} \label{s6-e11}
\begin{split}
&(\tfrac{1}{k}\Box^{(q)}_{b,k}\alpha_k\ |\ \alpha_k)_k\\
&\quad=\int k^n\langle\frac{1}{k}\Box^{(q)}_s\beta_k,\beta_k\rangle e^{-k\phi_0}m(z, \theta)dv(z)dv(\theta)\ \ (\text{by \eqref{s3-e7-0}})\\
&\quad=\int\langle\frac{1}{k}F^*_k(\Box^{(q)}_s\beta_k),F^*_k\beta_k\rangle_{F^*_k}e^{-kF^*_k\phi_0}(F^*_km)dv(z)dv(\theta)\\
&\quad=\int\langle(\Box^{(q)}_{s,(k)})F^*_k\beta_k,F^*_k\beta_k\rangle_{F^*_k}e^{-kF^*_k\phi_0}(F^*_km)dv(z)dv(\theta)\ \ (\text{by \eqref{s3-e9}}).
\end{split}
\end{equation}
From \eqref{s6-e6} and the fact that $\norm{F^*_k\beta_k}_{kF^*_k\phi_0}\leqslant  1$, we deduce that
there is a sequence $\mu_k>0$, independent of the point $p$ and
tending to zero such that
\begin{equation} \label{s6-e12}
\big(\tfrac{1}{k}\Box^{(q)}_{b,k}\alpha_k\ \big|\ \alpha_k\big)_k\leqslant \mu_k.
\end{equation}
Similarly, from \eqref{s6-e7}, we can repeat the procedure above with minor changes and get
\begin{equation} \label{s6-e13}
\big\|\big(\tfrac{1}{k}\Box^{(q)}_{b,k}\big)^m\alpha_k\big\|\To0\ \ \mbox{as $k\To\infty$},
\end{equation}
for all $m\in\mathbb N$. Now, we can prove

\begin{prop} \label{s6-p1}
Let $\nu_k>0$ be any sequence with
$\lim_{k\To\infty}\frac{\mu_k}{\nu_k}=0$,
where $\mu_k$ is as in \eqref{s6-e12}. Then, $\liminf_{k\To\infty}k^{-n}\pit^{(q)}_{k,\,\leqslant  k\nu_k}(0)\geqslant\frac{1}{2(2\pi)^{n}}\int_{\Real_{\phi(p),\,q}}\abs{\det(M^\phi_p+s\mathcal{L}_p)}ds$.
\end{prop}

\begin{proof}
Let $\alpha_k$ be as in \eqref{s6-e8}. By \eqref{s6-e1}, we have
$\alpha_k=\alpha^{1}_k+\alpha^2_k$,
where $\alpha^1_k\in\cH^q_{b,\le k\nu_k}(X, L^k)$, $\alpha^2_k\in\ol{\cH^q_{b,>k\nu_k}(X, L^k)}$. From \eqref{s6-e2}, we have
\[\norm{\alpha^2_k}^2\leqslant \frac{1}{k\nu_k}\big(\Box^{(q)}_{b,k}\alpha^2_k\ \big|\ \alpha^2_k\big)_k\leqslant \frac{1}{k\nu_k}\big(\Box^{(q)}_{b,k}\alpha_k\ \big|\ \alpha_k\big)_k\leqslant \frac{\mu_k}{\nu_k}\To 0,\]
as $k\To\infty$.
Thus, $\lim_{k\To\infty}\norm{\alpha^2_k}=0$. Since $\norm{\alpha_k}\To1$ as $k\To\infty$, we get
\begin{equation} \label{s6-e15}
\lim_{k\To\infty}\norm{\alpha^1_k}=1.
\end{equation}
Now, we claim that
\begin{equation} \label{s6-e16}
\lim_{k\To\infty}k^{-n}\abs{\alpha^2_k(0)}^2=0.
\end{equation}
On $D$, we write $\alpha^2_k=s^kk^{\frac{n}{2}}e^{kR}\beta^2_k$,
$\beta^2_k\in\Omega^{0,q}(D)$. From \eqref{s3-e12} and the proof of Lemma~\ref{s3-l1}, we see that
\begin{equation} \label{s6-e17}
\abs{F^*_k\beta^2_k(0)}^2\leqslant  C_{n-1,r}\Bigr(\norm{F^*_k\beta^2_k}^2_{kF^*_k\phi_0,D_r}+\norm{(\Box^{(q)}_{s,(k)})F^*_k\beta^2_k}^2_{kF^*_k\phi_0,n-1,D_r}\Bigr),
\end{equation}
for some $r>0$. Now, we have
\begin{equation} \label{s6-e18}
\norm{F^*_k\beta^2_k}^2_{kF^*_k\phi_0,D_r}\leqslant \norm{\alpha^2_k}^2\To0,\ \ \mbox{as $k\To\infty$}.
\end{equation}
Moreover, from \eqref{s3-e12} and using induction, we get
\begin{equation} \label{s6-e19}
\norm{(\Box^{(q)}_{s,(k)}F^*_k\beta^2_k}^2_{kF^*_k\phi_0,n-1,D_r}\leqslant  C'\sum^n_{m=1}\norm{(\Box^{(q)}_{s,(k)})^mF^*_k\beta^2_k}^2_{kF^*_k\phi_0,D_{r'}}\,,
\end{equation}
for some $r'>0$, where $C'>0$ is independent of $k$. We can check that for all $m\in\mathbb N$,
\begin{equation} \label{s6-e20}
\Big\|(\Box^{(q)}_{s,(k)})^mF^*_k\beta^2_k\Big\|^2_{kF^*_k\phi_0,D_{r'}}\leqslant \Big\|(\tfrac{1}{k}\Box^{(q)}_{b,k})^m\alpha^2_k\Big\|^2
\leqslant \Big\|(\tfrac{1}{k}\Box^{(q)}_{b,k})^m\alpha_k\Big\|^2\To0\ \ \text{as $k\To\infty$}.
\end{equation}
Here we used \eqref{s6-e13}.
Combining \eqref{s6-e20}, \eqref{s6-e19}, \eqref{s6-e18} with \eqref{s6-e17}, we get
\[\lim_{k\To\infty}\abs{F^*_k\beta^2_k(0)}^2=\lim_{k\To\infty}\abs{\beta^2_k(0)}^2=
\lim_{k\To\infty}k^{-n}\abs{\alpha^2_k(0)}^2=0.\]
Hence \eqref{s6-e16} follows. From this and \eqref{s6-e9}, we conclude
\begin{equation} \label{s6-e21}
\lim_{k\To\infty}k^{-n}\abs{\alpha^1_k(0)}^2=\frac{1}{2(2\pi)^{n}}\int_{\Real_{\phi(p),\,q}}\abs{\det(M^\phi_p+s\mathcal{L}_p)}ds.
\end{equation}
Now,
\[
k^{-n}\pit^{(q)}_{k,\,\leqslant  k\nu_k}(0)\geqslant k^{-n}\frac{\abs{\alpha^1_k(0)}^2}{\norm{\alpha^1_k}^2}\To
\frac{1}{2(2\pi)^{n}}\int_{\Real_{\phi(p),\,q}}\abs{\det(M^\phi_p+s\mathcal{L}_p)}ds,\ \ \mbox{as $k\To\infty$}.
\]
The proposition follows.
\end{proof}
\begin{prop}  \label{s6-p2}
Let $\nu_k>0$ be any sequence with $\nu_k\To0$, as $k\To\infty$. Then,
\[\limsup_{k\To\infty}k^{-n}\pit^{(q)}_{k,\,\leqslant  k\nu_k}(0)\leqslant\frac{1}{2(2\pi)^{n}}\int_{\Real_{\phi(p),\,q}}\abs{\det(M^\phi_p+s\mathcal{L}_p)}ds.\]
\end{prop}

\begin{proof}
The proof is a simple modification of the proof of Theorem~\ref{s5-t1} and in what follows these modifications will be presented.
Let $\alpha_k\in\cH^q_{b,\leqslant  k\nu_k}(X, L^k)$ with $\norm{\alpha_k}=1$. On $D$, we write
$\alpha_k=s^kk^{\frac{n}{2}}e^{kR}\beta_k$,
$\beta_k\in\Omega^{0,q}(D)$. From \eqref{s3-e12} and using induction, we get
\begin{equation} \label{s6-e22}
\norm{F^*_k\beta_k}^2_{kF^*_k\phi_0,s+1,D_r}
\leqslant  C_{r,s}\Bigr(\norm{F^*_k\beta_k}^2_{kF^*_k\phi_0,D_{2r}}+\sum^{s+1}_{m=1}
\norm{(\Box^{(q)}_{s,(k)})^mF^*_k\beta_k}^2_{kF^*_k\phi_0,D_{2r}}\,\Bigr).
\end{equation}
We can check that
$\big\|(\Box^{(q)}_{s,(k)})^mF^*_k\beta_k\big\|^2_{kF^*_k\phi_0,D_{2r}}\leqslant \big\|(\frac{1}{k}\Box^{(q)}_{b,k})^m\alpha_k\big\|
\leqslant \nu_k^m\To0$.
Thus, the conclusion of Proposition~\ref{s3-p3} is still valid and the rest of the argument goes through word by word.
\end{proof}
\begin{proof}[Proof of Theorems \ref{t-main3}, \ref{t-main4} and \ref{t-main5}]
We can repeat the proof of Theorem~\ref{s3-t1} and conclude that
for any sequence $(\nu_k)$ with $\nu_k\To0$, as $k\To\infty$,
there is a constant $C_0$ independent of $k$, such that
$k^{-n}\pit^{(q)}_{k,\,\leqslant  k\nu_k}(x_0)\leqslant  C_0$ for all $x_0\in X$.
From this, Proposition~\ref{s6-p1} and Proposition~\ref{s6-p2} and the fact that the sequence $(\mu_k)$ in \eqref{s6-e12} is independent of the point $p$, we get Theorem \ref{t-main3}. By integrating Theorem \ref{t-main3} we obtain Theorem \ref{t-main4}.
By applying the algebraic Morse inequalities \cite[Lemma\,3.2.12]{MM07} to the $\ddbar_{b,k}$-complex \eqref{s1-e7}
we deduce in view of Theorem \ref{t-main4} the strong Morse inequalities of Theorem \ref{t-main5}.
\end{proof}

\section{Examples}

In this section, some examples are collected. The aim is to illustrate the main results in some simple situations. First, we state
our main results in the embedded case.

\subsection{The main results in the embedded cases}

Let $M$ be a relatively compact open subset with $C^\infty$ boundary $X$ of a
complex manifold $M'$ of dimension $n$ with a smooth Hermitian metric $\langle\,\cdot\,,\cdot\,\rangle$. Furthermore, let $(L,h^L)$ be a Hermitian holomorphic
line bundle over $M'$ and let $\phi$ be a local weight of the metric $h^L$, i.\ e.\ for a local trivializing
section $s$ of $L$ on an open subset $D\subset M'$, $\abs{s(x)}_{h^L}^2=e^{-\phi}$. If we restrict $L$ on the boundary $X$, then $L$
is a CR line bundle over the CR manifold $X$. For $p\in X$, let $M^\phi_p$ be as in Definition~\ref{s1-d3}.

\begin{prop} \label{s7-p1}
For $U, V\in T^{1, 0}_pX$, we have
$M^\phi_p(U, \ol V)=\big\langle\pr\ddbar\phi(p), U\wedge\ol V\big\rangle$.
\end{prop}

\begin{proof}
Let $r\in C^\infty(X;\ \Real)$ be a defining function of $X$. For $U, V\in T^{1, 0}_pX$, pick $\mU, \mV\in C^\infty(M';\, T^{1, 0}M')$ that satisfy $\mU(p)=U$,
$\mV(p)=V$ and $\mU(r)=\mV(r)=0$ in a neighborhood of $p$ in $M'$. From \eqref{s5-e3}, we have
\begin{equation} \label{s7-e2}
\begin{split}
2\big\langle\mU\wedge\ol{\mV}, \pr\ddbar\phi\big\rangle
&=\big\langle\mU\wedge\ol{\mV}, d(\ddbar\phi-\pr\phi)\big\rangle\\
&=\mU\big(\big\langle\ol{\mV},\ddbar\phi-\pr\phi\big\rangle\big)-\ol{\mV}\big(\big\langle\mU,\ddbar\phi-\pr\phi\big\rangle\big)-\big\langle\big[\mU, \ol{\mV}\,\big], \ddbar\phi-\pr\phi\big\rangle.
\end{split}
\end{equation}
Note that
$\langle\ol{\mV}, \ddbar\phi-\pr\phi\rangle=\langle\ol{\mV},\ddbar\phi\rangle=\ol{\mV}\phi$
and $\langle\mU, \ddbar\phi-\pr\phi\rangle=\langle\mU, -\pr\phi\rangle=-\mU\phi$.
From this observation, \eqref{s7-e2} becomes
\begin{equation} \label{s7-e3}
2\big\langle\mU\wedge\ol{\mV}, \pr\ddbar\phi\big\rangle=\big(\mU\ol{\mV}+\ol{\mV}\mU\big)\phi-\big\langle\big[\mU, \ol{\mV}\,\big], \ddbar\phi-\pr\phi\big\rangle.
\end{equation}
Since $\mU(r)=\mV(r)=0$ in a neighborhood of $p$ in $M'$, we have
\[\big(\mU\ol{\mV}+\ol{\mV}\mU\big)\phi(p)=\big(\mU|_X\ol{\mV}|_X+\ol{\mV}|_X\mU|_X\big)\phi|_X(p)\] and
\[\big\langle\big[\mU, \ol{\mV}\,\big], \ddbar\phi-\pr\phi\big\rangle(p)=
\big\langle\big[\mU|_X, \ol{\mV}|_X\big], \ddbar_b\phi|_X-\pr_b\phi|_X\big\rangle(p),\]
where $\mU|_X$ is the restriction to $X$ of $\mU$ and similarly for $\ol{\mV}$ and $\phi$.
From this observation and \eqref{s7-e3}, we conclude that
\begin{equation} \label{s7-e4}
2\big\langle\mU\wedge\ol{\mV}, \pr\ddbar\phi\big\rangle(p)=\big(\mU|_X\ol{\mV}|_X+\ol{\mV}|_X\mU|_X\big)\phi|_X(p)-\big\langle\big[\mU|_X, \ol{\mV}|_X\big], \ddbar_b\phi|_X-\pr_b\phi|_X\big\rangle(p).
\end{equation}
From \eqref{s7-e4} and Lemma~\ref{s5-l1}, the proposition follows.
\end{proof}

We denote by  $R^L_X$ the restriction of $R^L$ to $T^{1,0}X$.
As before, let $\mathcal{L}_p$ be the Levi form of $X$ at $p\in X$. We define the set $\Real_{\phi(p),\,q}$ as in
\eqref{s1-e15}. Set
\begin{equation}
I^q(X,L):=\int_X\int_{\Real_{\phi(x),q}}\!\abs{\det\bigl(R^L_X+s\mathcal{L}_x\bigr)}ds\,dv_X(x)\,.
\end{equation}

Now, we can reformulate Theorem~\ref{t-main2} and Theorem~\ref{t-main5}:
\begin{thm} \label{s7-t1}
If condition $Y(q)$ holds, then
\begin{equation}\label{s7-e41}
\dim H^q_b(X, L^k)
\leqslant\frac{k^n}{2(2\pi)^{n}} I^q(X,L)+o(k^n),
\end{equation}
If condition $Y(j)$ holds, for all $j=0,1,\ldots,q$, then
\begin{equation}\label{s7-e5}
\sum^q_{j=0}(-1)^{q-j}{\rm dim\,}H^j_b(X, L^k)
\leqslant\frac{k^n}{2(2\pi)^{n}}\sum^q_{j=0}(-1)^{q-j}I^j(X,L)+o(k^n).
\end{equation}
If condition $Y(j)$ holds, for all $j=q,q+1,\ldots,n-1$, then
\begin{equation*}
\sum^{n-1}_{j=q}(-1)^{q-j}{\rm dim\,}H^j_b(X, L^k)
\leqslant\frac{k^n}{2(2\pi)^{n}}\sum^{n-1}_{j=q}(-1)^{q-j}I^j(X,L)+o(k^n).
\end{equation*}
\end{thm}
%
%
%
%

\begin{proof}[Proof of Theorem \ref{th-app1}]
Since $L$ is positive, $R^L_X+s\mathcal{L}_x$ is positive if $\abs{s}$ small, so $\abs{\Real_{\phi(p),0}}>0$ for all
$p\in X$. By the hypothesis of Theorem \ref{th-app1} we have $\lambda_{n_--1}=\lambda_{n_-}<0<\lambda_{n_-+1}=\lambda_{n_-+2}$ at each point of $X$. This implies, that if $R^L_X+s\mathcal{L}_x$ cannot have exactly one negative eigenvalue at any point of $X$ (note that $s$ takes both negative and positive values). Thus, $\Real_{\phi(p),1}=\emptyset$ for all $p\in X$. Hence,
the strong Morse inequalities \eqref{s7-e5} for $q=1$ imply the conclusion.
\end{proof}

\noindent
\begin{proof}[Proof of Theorem \ref{th-strip}]
Note that $X$ and $L$ satisfy the conditions of Theorem \ref{th-app}.
Hence there exists $c>0$ such that $\dim H^0_b(X,L^k)\geqslant ck^n$,
for $k$ sufficiently large.

On the other hand, every CR function on $X$ extends
locally to a holomorphic function in a small open set of $M$. For $b<c$,
set $M_b^c=\{b<\rho<c\}$. Thus, there exist
$b<a<c$ such that the restriction morphism $H^0(M_b^c,E)\to H^0(X,E)$ is
an isomorphism for any holomorphic line bundle $E\to M$. Moreover, we
know by the Andreotti-Grauert isomorphism theorem \cite{AG:62} that the restriction $H^0(M,E)\to H^0(M_b^c,E)$ is
an isomorphism. Thus there exist $C,c>0$ and $k_0\in\N$ such that
\begin{equation} \label{gm3.60}
ck^n\leqslant\dim H^0(M,L^k)\leqslant Ck^n\,,\quad \text{for $k\geqslant k_0$}\,,
\end{equation}
which we write
\[
\dim H^0(M,L^k)\sim k^n\,, \quad k\to\infty\,.
\]
Now, $M$ is a $(n-2)$-concave manifold in the sense
of \cite{AG:62}, in particular Andreotti-pseudo-concave (see
\cite[Def.\,3.4.3]{MM07}).
By \eqref{gm3.60} and \eqref{gm3.6} we obtain that $\varrho_k=n$ for large $k$ and the desired conclusions follow.
\end{proof}
Theorem \ref{th-strip} is a consequence of Theorem \ref{th-app} and of the fact that CR section extend
to a $(n-2)$-convex-concave strip around the given CR manifold.
By extending the sections as far as possible we obtain the following results.
\begin{cor}\label{gr-riem0}
Let $M$ be a projective manifold, $n=\dim_\Complex M$, and let $X$ be a compact hypersurface, such that there exist an open set $U\subset M$ and $\rho\in C^\infty(U,\Real)$ with $X=\rho^{-1}(0)\subset U$ and $d\rho|_X\neq0$.
Let $L\to M$ be a holomorphic line bundle over $M$. We assume that the curvature form
$R^L$ of $L$ and the Levi form $\pr\ddbar\rho|_X$ satisfy the assumptions of Theorem~\ref{th-app}.
Then there exist a branched covering $\pi:\widetilde{M}\to M$, a divisor $H\subset \widetilde{M}$ and an integer $d=d(L)$ such that for $\widetilde{L}:=\pi^*L$ holds:
\[
\dim H^0(\widetilde{M}\setminus H,\widetilde{L}^k)=\dim H^0(\widetilde{M},\widetilde{L}^k\otimes[kdH])\sim k^n\,,\quad k\to\infty\,.
\]
\end{cor}
\begin{proof}
Let us first observe that under the given hypotheses, there exist $b<0<c$ such that $M'=\{b<\rho<c\}$
is a $(n-2)$-convex-concave strip which fulfills the assumptions of Theorem \ref{th-strip}. By \eqref{gm3.60}, $\dim H^0(M',L^k)\sim k^n$, $k\to\infty$.
Since $M'$ is Andreotti-pseudoconcave, a theorem of Dingoyan \cite{Ding99,Ding08}
shows that there exist a branched covering $\pi:\widetilde{M}\to M$ with a section $S$ on $M$, a divisor $H\subset\widetilde{M}$ and an integer $d$ such that
holomorphic sections of $\pi^*L^k$ over $S(M')$ extends to a holomorphic section of $\pi^*L^k$ over
$\widetilde{M}\setminus H$ or of $\widetilde{L}^k\otimes[kdH]$ over $\widetilde{M}$.
Thus, the restriction morphisms $H^0(\widetilde{M}\setminus H,\widetilde{L}^k)\to H^0(S(M'),\widetilde{L}^k)$ and $H^0(\widetilde{M},\widetilde{L}^k\otimes[kdH])\to H^0(S(M'),\widetilde{L}^k)$ are isomorphisms. On the other hand,
$\dim H^0(S(M'),\widetilde{L}^k)\sim k^n$, $k\to\infty$ and the assertion follows.
\end{proof}
\begin{rem}
There are several criteria for a line bundle on a compact manifold to be big (Siu, Ji-Shiffman, Bonavero see e.\ g.\ \cite[Ch.\,2]{MM07}).
Corollary \ref{gr-riem0} asserts roughly that if a line bundle $L$ is positive along a well-chosen hypersurface
then by passing to a branched covering there exist a divisor $H$ and an integer $d$ such that $L\otimes[dH]$ is big.
\end{rem}
If one knows that $X$ has a neighbourhood $M'$, which is a $(n-2)$-convex-concave strip such that any meromorphic function on $M'$ is rational, then \cite{Ding99,Ding08} shows that there is no need to pass to a covering.
\begin{cor}\label{gr-riem1}
Assume the same hypotheses as in Corollary \ref{gr-riem0} for $M$, $X$ and $L$.
Assume moreover, that there exists a $(n-2)$-convex-concave strip $M'=\{b<\rho<c\}$ such that any meromorphic function on $M'$ is rational.
Then there exists a divisor $H\subset M$ and an integer $d$ such that
\[
\dim H^0(M\setminus H,L^k)=\dim H^0(M,L^k\otimes[kdH])\sim k^n\,,\quad k\to\infty\,.
\]
\end{cor}

\subsection{Holomorphic line bundles over a complex torus}

Let
\[T_n:=\Complex^n/(\sqrt{2\pi}\mathbb Z^n+i\sqrt{2\pi}\mathbb Z^n)\]
be the flat torus and let $L_\lambda$ be the holomorphic
line bundle over $T_n$
with curvature the $(1,1)$-form
$\Theta_\lambda=\sum^n_{j=1}\lambda_jdz_j\wedge d\ol z_j$,
where $\lambda_j$, $j=1,\ldots,n$, are given non-zero integers. More precisely, $L_\lambda:=(\Complex^n\times\Complex)/_\sim$\,, where
$(z, \theta)\sim(\Td z, \Td\theta)$ if
\[
\Td z-z=(\alpha_1,\ldots,\alpha_n)\in \sqrt{2\pi}\mathbb Z^n+i\sqrt{2\pi}\mathbb Z^n\,,\quad
\Td\theta=\textstyle\exp\big(\sum^n_{j=1}\lambda_j(z_j\ol\alpha_j+\tfrac{1}{2}\abs{\alpha_j}^2\,)\big)\theta\,.
\]
We can check that $\sim$ is an equivalence relation and $L_\lambda$ is a holomorphic line bundle over $T_n$.
For $[(z, \theta)]\in L_\lambda$
we define the Hermitian metric by
\[
\big\vert[(z, \theta)]\big\vert^2:=\abs{\theta}^2\textstyle\exp(-\sum^n_{j=1}\lambda_j\abs{z_j}^2)
\]
and it is easy to see that this definition is independent of the choice of a representative $(z, \theta)$ of $[(z, \theta)]$. We denote by $\phi_\lambda(z)$ the weight of this Hermitian fiber metric. Note that $\pr\ddbar\phi_\lambda=\Theta_\lambda$.
From now on, we assume that $\lambda_j<0$, for $j=1,\ldots, n_-$ and $\lambda_j>0$, for $j=n_-+1,\ldots, n$.

Let $L^*_\lambda$ be the
dual bundle of $L_\lambda$ and let $\norm{\,\cdot\,}_{L^*_\lambda}$ be the norm of $L^*_\lambda$ induced by the Hermitian fiber metric on $L_\lambda$. Consider the compact CR manifold of dimension $2n+1$ $X=\{v\in L^*_\lambda;\, \norm{v}_{L^*_\lambda}=1\}$; this is the boundary of Grauert tube of $L^*_\lambda$.

Let $\pi:L^*_\lambda\To T_n$
be the natural projection from $L^*_\lambda$ onto $T_n$. Let $L_\mu$ be another holomorphic
line bundle over $T_n$ determined by the constant curvature form
$\Theta_\mu=\sum^n_{j=1}\mu_jdz_j\wedge d\ol z_j$,
where $\mu_j$, $j=1,\ldots,n$, are given non-zero integers. The pullback line bundle $\pi^*L_\mu$ is a holomorphic line bundle over $L^*_\lambda$. The Hermitian fiber metric $\phi_\mu$ on $L_\mu$ induces a Hermitian fiber metric on $\pi^*L_\mu$ that we
shall denote by $\psi$. If we restict $\pi^*L_\mu$ on $X$, then $\pi^*L_\mu$ is a CR line bundle over the CR manifold $X$.

The part of $X$ that lies over a fundamental domain of $T_n$ can be represented in local holomorphic coordinates
$(z, \xi)$, where $\xi$ is the fiber coordinates, as the set of all $(z, \xi)$ such that
$r(z, \xi):=\abs{\xi}^2\exp(\sum^n_{j=1}\lambda_j\abs{z_j}^2)-1=0$
and the fiber metric $\psi$ may be written as $\psi(z, \xi)=\sum^n_{j=1}\mu_j\abs{z_j}^2$.
We can identify $\mathcal{L}_p$ with $\frac{1}{\norm{dr(p)}}\sum^n_{j=1}\lambda_jdz_j\wedge d\ol z_j$.
It is easy to see that
$\pr\ddbar\psi(p)|_{T^{1, 0}X}=\sum^n_{j=1}\mu_jdz_j\wedge d\ol z_j$.
We get for all $p\in X$, $s\in\Real$,
\[
\pr\ddbar\psi(p)|_{T^{1, 0}X}+s\mathcal{L}_p=\sum^n_{j=1}\Big(\mu_j+\frac{s}{\norm{dr(p)}}\lambda_j\Big)dz_j\wedge d\ol z_j\,.
\]
Thus, if $\mu_j=\lambda_j$, $j=1,\ldots,n$, and $q\neq n_-, n-n_-$, then $\Real_{\phi(p),\,q}=\emptyset$, for all $p\in X$. From this and Theorem~\ref{s7-t1}, we obtain

\begin{thm} \label{s7-t3}
If $\mu_j=\lambda_j$, $j=1,\ldots,n$, and $q\neq n_-, n-n_-$, then
\[\dim H^q_b(X, (\pi^*L_\mu)^k)=o(k^{n+1})\,,\quad \text{as $k\to\infty$}\,.\]
\end{thm}

If $\mu_j=\abs{\lambda_j}$, $j=1,\ldots,n$, we can check that $\abs{\Real_{\phi(p),0}}>0$, for all $p\in X$, where $\abs{\Real_{\phi(p),0}}$ denotes the Lebesgue measure of $\Real_{\phi(p),0}$. Moreover, if $q>0$ and $q\neq n_-, n-n_-$, then $\Real_{\phi(p),\,q}=\emptyset$, for all $p\in X$. From this observation, \eqref{s7-e41} for $q=0$ and \eqref{s7-e5} for $q=1$, 
we obtain

\begin{thm} \label{s7-t4}
If $\mu_j=\abs{\lambda_j}$, $j=1,\ldots,n$, and $Y(0)$, $Y(1)$ hold, then
\[\dim H^0_b(X, (\pi^*L_\mu)^k)\sim k^{n+1}\,,\quad \text{as $k\to\infty$}\,.\]
\end{thm}

\subsection{Compact Heisenberg groups: non-embedded cases} Next we consider compact analogues of the Heisenberg group $H_n$. Let $\lambda_1,\ldots,\lambda_{n-1}$ be given non-zero integers.
Let $\mathscr CH_n=(\Complex^{n-1}\times\Real)/_\sim$\,, where
$(z, \theta)\sim(\Td z, \Td\theta)$ if
\[
\textstyle
\Td z-z=\alpha\in\sqrt{2\pi}\mathbb Z^{n-1}+i\sqrt{2\pi}\mathbb Z^{n-1}\,,\quad
\Td\theta-\theta+i\sum^{n-1}_{j=1}\lambda_j(z_j\ol\alpha_j-\ol z_j\alpha_j)\in\pi\mathbb Z\,.
\]
We can check that $\sim$ is an equivalence relation
and $\mathscr CH_n$ is a compact manifold of dimension $2n-1$. The equivalence class of $(z, \theta)\in\Complex^{n-1}\times\Real$ is denoted by
$[(z, \theta)]$.
For a given point $p=[(z, \theta)]$, we define
$T^{1, 0}_p\mathscr CH_n$ to be the space spanned by
\[
\textstyle
\big\{\frac{\pr}{\pr z_j}-i\lambda_j\ol z_j\frac{\pr}{\pr\theta},\ \ j=1,\ldots,n-1\big\}.
\]
It is easy to see that the definition above is independent of the choice of a representative $(z, \theta)$ for $[(z, \theta)]$.
Moreover, we can check that $T^{1, 0}\mathscr CH_n$ is a CR structure. Thus, $(\mathscr CH_n, T^{1, 0}\mathscr CH_n)$ is a compact CR manifold of dimension $2n-1$. We take a Hermitian metric $\langle\,\cdot\,,\cdot\,\rangle$ on the complexified tangent bundle $\Complex T\mathscr CH_n$ such that
\[
\Big\lbrace
\tfrac{\pr}{\pr z_j}-i\lambda_j\ol z_j\tfrac{\pr}{\pr\theta}\,, \tfrac{\pr}{\pr\ol z_j}+i\lambda_jz_j\tfrac{\pr}{\pr\theta}\,, \tfrac{\pr}{\pr\theta}\,;\, j=1,\ldots,n-1\Big\rbrace
\]
 is an orthonormal basis. The dual basis of the complexified cotangent bundle is
\[
\Big\lbrace
dz_j\,,\, d\ol z_j\,,\, \omega_0:=d\theta+\textstyle\sum^{n-1}_{j=1}(i\lambda_j\ol z_jdz_j-i\lambda_jz_jd\ol z_j); j=1,\ldots,n-1
\Big\rbrace\,.
\]

The Levi form $\mathcal{L}_p$ of $\mathscr CH_n$ at $p\in\mathscr CH_n$ is given by
$\mathcal{L}_p=\sum^{n-1}_{j=1}\lambda_jdz_j\wedge d\ol z_j$.
From now on, we assume that $\lambda_1<0,\ldots,\lambda_{n_-}<0, \lambda_{n_-+1}>0,\ldots,\lambda_{n-1}>0$. Thus, the Levi form
has constant signature $(n_-, n-1-n_-)$.

Now, we construct a CR line bundle over $\mathscr CH_n$. Let $L=(\Complex^{n-1}\times\Real\times\Complex)/_\equiv$ where $(z,\theta,\eta)\equiv(\Td z, \Td\theta, \Td\eta)$ if
\[
(z,\theta)\sim(\Td z, \Td\theta)\,,\quad
\Td\eta=\eta \exp(\textstyle\sum^{n-1}_{j=1}\mu_j(z_j\ol\alpha_j+\frac{1}{2}\abs{\alpha_j}^2))\,,\quad\text{for $\alpha=\Td z-z$}\,.
\]
where $\mu_1,\ldots,\mu_{n-1}$, are given non-zero integers. We can check that $\equiv$ is an equivalence relation and
$L$ is a CR line bundle over $\mathscr CH_n$. For $(z, \theta, \eta)\in\Complex^{n-1}\times\Real\times\Complex$ we denote
$[(z, \theta, \eta)]$ its equivalence class.
It is easy to see that the pointwise norm
\[
\big\lvert[(z, \theta, \eta)]\big\rvert^2:=\abs{\eta}^2\exp\big(-\textstyle\sum^{n-1}_{j=1}\mu_j\abs{z_j}^2\big)
\]
is well-defined. In local coordinates $(z, \theta, \eta)$, the weight function of this metric is
$\phi=\sum^{n-1}_{j=1}\mu_j\abs{z_j}^2$. Note that
\[
\textstyle\ddbar_b=\sum^{n-1}_{j=1}d\ol z_j\wedge(\frac{\pr}{\pr\ol z_j}+i\lambda_jz_j\frac{\pr}{\pr\theta})\,,\quad
\pr_b=\sum^{n-1}_{j=1}dz_j\wedge(\frac{\pr}{\pr z_j}-i\lambda_j\ol z_j\frac{\pr}{\pr\theta}).
\]
Thus
$d(\ddbar_b\phi-\pr_b\phi)=2\sum^{n-1}_{j=1}\mu_jdz_j\wedge d\ol z_j$ and $M^\phi_p=\sum^{n-1}_{j=1}\mu_j dz_j\wedge d\ol z_j$.
Hence
\[
\textstyle
M^\phi_p+s\mathcal{L}_p=\sum^n_{j=1}(\mu_j+s\lambda_j)dz_j\wedge d\ol z_j\,,\quad\text{for all $p\in\mathscr CH_n$, $s\in\Real$}.
\]
Thus, if $\mu_j=\lambda_j$, for all $j$, and $q\neq n_-, n-1-n_-$, then $\Real_{\phi(p),\,q}=\emptyset$, for all $p\in X$. From this and Theorem~\ref{t-main2}, we obtain

\begin{thm} \label{s7-t5}
If $\mu_j=\lambda_j$, $j=1,\ldots,n-1$, and $q\neq n_-, n-1-n_-$, then
\[\dim H^q_b(\mathscr CH_n, L^k)=o(k^n)\,,\quad \text{as $k\to\infty$}\,.\]
\end{thm}
If $\mu_j=\abs{\lambda_j}$ for all $j$, we can check that $\abs{\Real_{\phi(p),0}}>0$, for all $p\in X$, where $\abs{\Real_{\phi(p),0}}$ denotes the Lebesgue measure of $\Real_{\phi(p),0}$. Moreover, if $q>0$ and $q\neq n_-, n-1-n_-$, then $\Real_{\phi(p),\,q}=\emptyset$, for all $p\in X$. From this observation, the weak Morse inequalities (Theorem \ref{t-main2}) for $q=0$ and the strong Morse inequalities (Theorem~\ref{t-main5}), we obtain

\begin{thm} \label{s7-t6}
If $\mu_j=\abs{\lambda_j}$, $j=1,\ldots,n-1$, and $Y(0)$, $Y(1)$ hold, then
\[
\dim H^0_b(\mathscr CH_n, L^k)\sim k^n\,,\quad \text{as $k\to\infty$}\,.
\]
\end{thm}

\smallskip

\noindent
{\small\emph{
\textbf{Acknowledgements.} The first-named author is grateful to Bo Berndtsson and Robert Berman for several useful conversations and to
the Department of Mathematics, Chalmers University of Technology and the University of G\"{o}teborg for
offering excellent working conditions during December 2008-December 2010. We would like to thank P. Dingoyan for discussions about the extension of holomorphic sections \cite{Ding99}. We also thank the referee for many detailed remarks that have helped to improve the presentation.
}}

\end{document}